\documentclass{amsart}

\usepackage[english]{babel}

\usepackage{soul}
\newcommand{\comment}[1]{}
\usepackage[letterpaper,top=2cm,bottom=2cm,left=3cm,right=3cm,marginparwidth=1.75cm]{geometry}

\usepackage{amsmath, amssymb, amsthm}
\usepackage{graphicx}
\newtheorem{Example}{Example}
\newtheorem{Lemma}{Lemma}
\newtheorem{Corollary}{Corollary}
\newcommand{\norm}[1]{\left\|{#1}\right\|}
\newtheorem{Definition}{Definition}
\newtheorem{Remark}{Remark}
\newtheorem{Theorem}{Theorem}
\newtheorem{Observation}{Observation}
\usepackage[colorlinks=true, allcolors=blue]{hyperref}

\newcommand{\R}{{\mathbb R}}

\title{Autonomous and non-autonomous unbounded attractors in evolutionary problems}
\author{Jakub Bana\'{s}kiewicz$^1$}
\address{$^1$Faculty of Mathematics and Computer Science, Jagiellonian University, ul. \L{}ojasiewicza 6, 30-348 Krakow, Poland}

\author{Alexandre N. Carvalho$^2$}
\address{$^2$Departamento de Matem\'{a}tica,
	Instituto de Ci\^{e}ncias Matem\'{a}ticas e de Computa\c{c}\~{a}o, 
	Universidade de S\~{a}o Paulo - Campus de S\~{a}o Carlos,
	Caixa Postal 668,
	13560-970  S\~{a}o Carlos   SP,
	Brazil}

\author{Juan Garcia-Fuentes$^3$}

\author{Piotr Kalita$^{1,3}$}

\address{$^{3}$Departamento de Ecuaciones Diferenciales y An\'{a}lisis Num\'{e}rico, Universidad de Sevilla, Campus Reina Mercedes, 41012, Sevilla, Spain}

\email{jakub.banaskiewicz@doctoral.uj.edu.pl}
\email{andcarva@icmc.usp.br}
\email{jgfuentes@us.es}
\email{piotr.kalita@ii.uj.edu.pl}

\begin{document}

\begin{abstract}
If the semigroup is slowly non-dissipative, i.e., its {\color{black} solutions} can diverge to infinity as time tends to infinity, one still can study its dynamics via the approach by the unbounded attractors - the counterpart of the classical notion of global attractors. We continue the development of this theory started by Chepyzhov and Goritskii \cite{chepyzhov}. We provide the abstract results on the unbouded attractor existence, and we study the properties of these attractors, as well as of unbounded $\omega$-limit sets in slowly non-dissipative setting. We also develop the pullback non-autonomous counterpart   of the unbounded attractor theory. The abstract theory that we develop  is illustrated by the analysis of the autonomous problem governed by the equation $u_t = Au + f(u)$. In particular, using the inertial manifold approach, we provide the criteria under which the unbounded attractor coincides with the graph of the Lipschitz function, or becomes close to the graph of the Lipschitz function for large argument.   
\end{abstract}
\maketitle\section{Introduction and summary of concepts}
Typically, to build the theory of the global attractors (and their non-autonomous counterparts) one needs the dynamical system to be dissipative, i.e. there should exist a bounded absorbing set \cite{BCL, Cholewa-Dlotko, robinson, Carvalho-Langa-Robinson-13, Henry}. However, starting from the work of Chepyzhov and Goritskii \cite{chepyzhov}, the concept of the global attractor has been generalized to the case of slowly-nondissipative systems, i.e. such that the orbits are not absorbed by one given bounded set, and they possibly diverge to infinity when time tends to infinity, but there is no blow-up in finite time. We remark, that in recent years, stemming from \cite{chepyzhov}, there appeared significant work in the framework of unbounded attractors, cf., \cite{Pimentel,PimentelRocha,bengal,carvalhopimentel,carvalhopimentel2,hell}. To illustrate the underlying concept we start from two simple motivating examples. 
\begin{Example}\label{ex1}
	Consider the ODE
	$$
	\begin{cases}
		x' = x,\\
		y' = -y.
	\end{cases}
	$$
	The point $(0,0)$ is the equilibrium. {\color{black} Solutions} starting from $(x,0)$ tends to (plus of minus) infinity in $x$ as time tends to infinity, and tend to $(0,0)$ as time tends to $-\infty$. Any {\color{black} solution} starting from $(0,y)$ goes to $(0,0)$ as time tends to infinity. All other {\color{black} solutions} are unbounded both in past and in future. The set $\{ (x,0)\,:\ x\in \R \}$ is invariant, attracting, and unbounded. Moreover every {\color{black} solution} in this set is bounded in the past. 
\end{Example}

\begin{Example}\label{ex2}
	Consider another ODE
	$$
	\begin{cases}
		x' = y,\\
		y' = -x,\\
		z'=-z.
	\end{cases}
	$$
    	Again, the point $(0,0,0)$ is the equilibrium. All other {\color{black} solutions} with $z=0$ are periodic - they are circles in $(x,y)$ centered at zero. In fact every $(x,y)$ circle centered at zero is a periodic {\color{black} solution}. The set $\{ (x,y,0)\, :\ (x,y)\in \R^2 \}$ is invariant, unbounded and attracting. Every {\color{black} solution} in this set is {\color{black} bounded} both in the past and in the future. {\color{black} All other {\color{black} solutions} are unbounded in the past and bounded in the future.}
\end{Example}

Above two examples show that it is natural to consider the points in the phase space through which there exists a {\color{black} solution} defined for all $t\in \mathbb{R}$ which is bounded in the past, and one can expect that the set of such points should be invariant, and, in appropriate sense, attracting. This observation allowed Chepyzhov and Goritskii \cite{chepyzhov} to define the concept of unbounded attractor for the problem governed by the PDE $u'=Au + f(u)$ as a the set of initial conditions of {\color{black} solutions} which bounded in the past. In this paper we follow the same concept. 

The key contribution of the present article is the refinement and extention in several directions of the results from \cite{chepyzhov}.  We pass to the detailed description of the contribution and novelty of our work in comparison with the previous research on unbounded attractors.
First of all we provide the new criteria for the semigroup of operators $\{S(t)\}_{t\geq 0}$ on a Banach space $X$ to have the unbounded attractor. This semigroup does not have necessarily to be governed by the differential equation, be it PDE or ODE. Hence, the result on the existence of the unbounded attractor, Theorem \ref{th:autonomus1}, the first important theorem of the present article, is formulated and proved for an abstract semigroup. We stress, that while the argument to prove this theorem mostly follows the lines of the corresponding proofs in \cite{chepyzhov}, the present contribution lies in finding the abstract criteria, cf., (H1)--(H3) in the sequel, which guarantee its existence. 

The important property needed in the proof is to be able to split the phase space of the problem $X$ as the sum of two spaces $X=E^+\oplus E^-$ such that the projection $P$ on $E^+$ of the {\color{black} solution} can possibly diverge to infinity, and its projection on $E^-$ denoted by $I-P$ enjoys the dissipative properties.
For the general setup of semigroups,  following \cite{chepyzhov}, under our criteria, the unbounded attractor $\mathcal{J}$ attracts bounded sets $B$ only in bounded sets in $E^+$, that is 
$$
\lim_{t\to\infty}\textrm{dist}\, (S(t)B \cap \{ \|Px\|\leq R\}, \mathcal{J}) = 0,
$$
as long as the sets $S(t)B \cap \{ \|Px\|\leq R\}$ stay nonempty. A natural question that appears is, when one can remove the set $\{ \|Px\|\leq R\}$ from the above attracting property. While in \cite{chepyzhov} the authors prove such convergence  
if the linear manifold $E^+$ is attracting for the {\color{black} solutions} that diverge to infinity, we provide more general criterion {\color{black}[See condition (A1) in Section \ref{attraction_all}}], and show such attraction in Theorem \ref{attracts:all}. The criterion says that the attraction in the whole space occurs  if there exists the attracting set with the  thickness in space $E^-$ tending to zero as $\|Pu\|$ tends to infinity. This result later becomes useful in the part on inertial manifolds for the problem governed by PDEs, see Section \ref{sec:outside}. 

The next series of our results is inspired by the works of \cite{bengal, hell} later extended and refined in \cite{PimentelRocha, Pimentel, carvalhopimentel2}. Namely, for the case of the semigroup governed by the following PDE in one space dimension
$$
u_t = u_{xx} + bu +f(u,u_x,x),
$$ 
on the space interval, with appropriate (typically Neumann) boundary data, and sufficiently large constant $b>0$ the authors there study the detailed structure of the unbounded attractor which consists of equilibria, their heteroclinic connections,  approprietely understood `equilibiria at infinity' and their connections, as well as heteroclinic connections from the equilibria to `equilibria at infinity'. We present the new abstract results on the global attractor structure for the general, not necessarily gradient case. Clearly, for such general case it is not possible to fully describe the structure of the  attractor. However, we prove in Section \ref{sec:structure}, that the unbounded attractor must consist of an invariant set $\mathcal{J}_b$, built from points through which passes a bounded, both in the past and in the future, {\color{black} solution} and the remainder, $\mathcal{J}\setminus\mathcal{J}_b$ which consists of the heteroclinics to infinity, namely of the points for which the alpha limits belong to $\mathcal{J}_b$ and for which the dynamics diverges to infinity as time tends to infinity.  We study the properties of these sets, and in particular we show that $\mathcal{J}_b$ enjoys  properties of the global attractor: {\color{black} it is compact under additional natural assumption (H4), which, however excludes the cases like Example \ref{ex2} above, cf. Theorem \ref{compact},} and by taking the set of points in the phase space  which upon evolution converge to it, one obtains a complete metric space on which it is a classical global attractor, cf. Lemma \ref{lem:8}.  

We conclude the abstract results on slowly-nondissipative autonomous problems with several observations, which are, to our knowledge, new. The first one is, that the unbounded attractor  coincides with a multivalued graph over $E^+$, which we call a multivalued inertial manifold in spirit of \cite{debusshetemam, jcrobinson}, cf., our Section \ref{sec:inertial}. The second observation concerns the unbounded $\omega$-limit sets, whose properties we study in Section \ref{sec:omega}, and which, to our knowledge, have not been studied before. Interestingly, while we prove the unbounded counterparts of the classical properties of $\omega$-limit sets, i.e. invariance, compactness, and attraction, we were not able to prove their connectedness, and we leave, for now, the question of their connectedness open. Finally, in Section \ref{sec:infty}, we discus, on an abstract level, the dynamics at infinity, which corresponds to the equilibria at infinity and their connections in \cite{bengal, hell, PimentelRocha, Pimentel, carvalhopimentel2}, and we prove, using the Hopf lemma, and homotopy invariance of the degree, that the attractor at infinity always coincides with the whole unit sphere in $E^+$.

The study on unbounded non-autonomous attractors has been initiated by Carvalho and Pimentel in \cite{carvalhopimentel}. The authors study there the non-autonomous problem governed by the equation
$$
u_t = u_{xx} + b(t)u + f(u).
$$ 
For this problem they formulate the definition of the unbounded pullback attractor, being the natural extension of the autonomous definition, and provide the result on its existence. We continue the study on pullback unbounded attractors: we give the general definition of such object and prove the result on its existence in the abstract framework, cf., Theorem \ref{thm:pullback},  for slowly non-dissipative processes. The unbounded pullback attractor is the family of sets $\{\mathcal{J}(t) \}_{t\in \mathbb{R}}$, each of them being unbounded, we obtain the result on its structure similar as in the autonomous situation. $\mathcal{J}(t) = \mathcal{J}_b(t) \cup (\mathcal{J}(t)\setminus \mathcal{J}_b(t))$, where $\mathcal{J}_b(t)$ shares the properties of classical pullback attractors, and the {\color{black} solutions} through points in $\mathcal{J}(t)\setminus \mathcal{J}_b(t)$ (at time $t$) that diverge to infinity forward in time. Note that the construction of sets $\mathcal{J}_b(t)$ is not the simple transition to non-autonomous case of the autonomous one. In the autonomous case the corresponding set $\mathcal{J}(t)$ is constructed using the concept of $\alpha$-limit sets, while in the non-autonomous situation we need to study the forward behavior of the points from the unbounded pullback attractor. Finally, we provide the construction of pullback $\omega$-limit sets in the unbounded non-autonomous framework. As we prove, there naturally appear two universes of sets for which one can define such $\omega$-limits: the backward bounded ones, and the ones whose forward {\color{black} solutions} stay bounded. We prove that for the first universe non-autonomous $\omega$-limit sets are attracting and invariant in $\mathcal{J}(t)$, while for the second one they are attracting and  invariant  in $\mathcal{J}_b(t)$. 

The second part of this paper is devoted to the study of the unbounded attractors for the dynamical system governed by the following differential equation in the Banach space $X$
$$
u_t = Au + f(u),
$$
where $A$ is the linear, closed and densely defined operator, which is sectorial, has compact resolvent and no eigenvalues on the imaginary axis. Denoting by $P$ its spectral projector on the space associated with $\{ \textrm{Re}\, \lambda > 0 \}$, which is our space $E^+$, we have
\begin{equation}\label{eq:hyper_0}
	\begin{split}
		&\|e^{A(t-s)}\| \leq M e^{\gamma_0(t-s)}\quad\textrm{for}\ \ \ t\geq s,\\
		&\|e^{A(t-s)}(I-P)\|\leq M e^{-\gamma_2(t-s)}\quad\textrm{for}\ \ \ \ t\geq s,\\
		&\|e^{A(t-s)}P\|\leq M e^{\gamma_1(t-s)}\quad\textrm{for}\ \ \ \ t\leq s,
	\end{split}
\end{equation}
where $\gamma_0, \gamma_1, \gamma_2, M$ are positive constants. Thus, our analysis is based on the general concept of the mild {\color{black} solutions} and the Duhamel formula, rather then on weak {\color{black} solutions}, and energy estimates as in \cite{chepyzhov}. Our first interest lies in showing, in Theorem \ref{thm:decay}, that if $f(u)$ tends to zero, with appropriate rate, as $\|Pu\|$ tends to infinity, then the thickness of the unbounded attractor also tends to zero. In such case one can remove the restriction in the definition of the unbounded attractor that the attraction takes place in balls.  Similar result has already been obtained in \cite{chepyzhov}, where, however,  the polynomial decay of $f$ is assumed, namely, $\|f(u)\| \leq \frac{C}{\|Pu\|^\alpha}$. We recover and improve this result in the framework of mild {\color{black} solutions}, by allowing for very slow decay, such as $\|f(u)\| \leq \frac{C}{\ln(\|Pu\|)}$ for large $\|Pu\|$. 

We continue the analysis by trying to answer the question whether the thickness of the unbounded attractor can tend to zero (and thus the attraction by the unbounded attractor can occur in the whole space) without the decay of $f$. We give the positive answer to this question with the help of the theory of  inertial manifolds. Inertial manifolds in the framework of the unbounded attractors have already been studied by Ben-Gal \cite{bengal}. In this classical approach the spectral gap can be taken between any two (sufficiently large) eigenvalues of $-A$, and leads to the inertial manifold which contains the unbounded attractor. We follow the different path: namely we choose the spectral gap exactly between the spaces $E^+$ and $E^-$. While this gives strong restriction on the Lipschitz constant of the nonlinearity, interestingly, as we show, such approach leads to the inertial manifold which exactly coincides with the unbounded attractor. We construct this manifold by the two methods: the Hadamard graph transform method, where we base our approach on the classical article \cite{mallet} (note that the nonautonomous version of the construction of inertial manifolds by the graph transform method valid in context of the unbounded attractors but based on the energy estimates rather than the Duhamel formula has been realized in \cite{chepyzhov_goritskii_2}), and the Lyapunov--Perron method, based on its modern and non-autonomous rendition in framework of inertial manifolds \cite{CLMO-S}. We find out, that while the graph transform method works only for the case $M=1$, from the numerical comparison of the restrictions for $L_f$ - the Lipschitz constants of $f$ - found in both methods, that the graph transform approach gives the better (i.e. larger, although still expected to be non-sharp, cf. \cite{Zelik, romanov}) bound on $L_f$ then the Lyapunov--Perron method. In our further analysis we demonstrate that it is in fact sufficient to assume that the nonlinearity has the small Lipschitz constant only for large $\|Pu\|$, while for small $\|Pu\|$ it does not have to be Lipschitz at all. While in such a case the unbounded attractor does not have to be a Lipschitz manifold, we demonstrate by the appropriate modification  of the nonlinearity $f$, that for large $\|Pu\|$ it has to stay close to a Lipschitz manifold, whence its thickness, although it may be nonzero contrary to the case of small global Lipschitz constant, tends to zero with increasing $\|Pu\|$.

In the last result of this article we discuss the asymptotic behavior at infinity, which is the more general version of the results from \cite{hell, carvalhopimentel}. Namely, we show that the rescaled problem at infinity is asymptotically autonomous, and hence its $\omega$-limits concide with the $\omega$-limits of the finite dimensional autonomous problem at infinity whose dynamics can be explicitly constructed. 

  We stress that while we endeavour on building the comprehensive theory of unbounded attractors, our approach still has limitations. While we consider only the possibility of grow-up {\color{black} solutions}, there are other possibilities of asymptotic behaviors which we could encompass in our theory (see \cite{czaja-dlotko} for a recent and comprehensive overview). Moreover, in most of the results we consider only directions that are dissipative (those in $E^-$) and expanding (those in $E^+$). It would be interesting to extend the analysis of the case of \textit{neutral} dimensions, in which the {\color{black} solutions} are not absorbed in finite time by a given bounded set,  but still stay bounded, such as in Example \ref{ex2}, especially that such problems has recently arisen a significant interest in fluid mechanics, cf. \cite{foiaslariosbiswas}. 
  
  The plan of this article is the following:  Section \ref{sec:autonomous} is devoted to the autonomous abstract theory of unbounded attractors. In the next Section \ref{nonauto} we study, on abstract level, the nonautonomous unbounded attractors in the abstract framework. Finally, in Section \ref{PDES} we present the results on unbounded attractors for the problem governed by the equation $u_t = Au + f(u)$.

\section{Unbounded attractors and $\omega$-limit sets. Autonomous case.}\label{sec:autonomous}

\subsection{Non-dissipative semigroups and their invariant sets}
If $X$ is a metric space ($\mathcal{P}(X)$ represents the subsets of $X$ and $\mathcal{P}_0(X)$ represents the non-empty subsets of $X$) then $\mathcal{P}(X)$ are its nonempty subsets and $\mathcal{B}(X)$ are its nonempty and bounded subsets. 

{\color{black}
\begin{Definition}
Let $X$ be a metric space. A family of mappings $\{ S(t):D(S(t)) \subset X \to X \}_{t\geq 0}$ is a \textbf{{\color{black} solution} operator family} if
\begin{itemize}
    \item $D(S(0))=X$ and $S(0) = I$ (identity on $X$),
    \item For each $x\in X$ there exists a $\tau_x\in (0,\infty]$ such that $x\in D(S(t))$ for all $t\in [0,\tau_x)$. 
    \item For every $x\in X$ and $t,s\in \R^+$ such that $t+s\in [0,\tau_x)$ we have $S(t+s)x = S(t)S(s)x$,
    \item If $E=\{(t,x)\in \R^+\times X: t\in [0,\tau_x)\}$, the mapping $E (t,x)\mapsto S(t)x\in X$ is continuous.
\end{itemize}

If $\tau_x=+\infty$ for all $x\in X$, $D(S(t))=X$ for all $t\geq 0$ and $E=\R^+\times X$. In this case we say that the family $\{ S(t)\}_{t\geq 0}$ is a \textbf{continuous semigroup}.
\end{Definition}
}

When we will apply the above definition to differential equations, the mapping $S(t)$ will assign to the initial data the value of the {\color{black} solution} at the time $t$. 
%Hence, we assume that there is no blow-up i.e. for every $x\in X$, the value $S(t)x$ is always well defined. However, our aim is to consider the nondissipative case. 

{\color{black}
\begin{Definition}
The function $[0,\tau_x)\ni t \mapsto S(t)x\in X$ is called a {\color{black} {\color{black} solution}}. 
\end{Definition}

\begin{Definition}
The function $\gamma:\mathbb{R}\to X$ is a {\color{black}global {\color{black} solution}} if $S(t)\gamma(s) = \gamma(s+t)$ for every $s\in \mathbb{R}$ and $t\geq 0$. 
\end{Definition}

\begin{Definition}
The {\color{black}global {\color{black} solution}}  $\gamma$ is bounded in the past if the set $\gamma(\mathbb{R}^-)=\{\gamma(t)\in X:t\leq 0\}$ is bounded. 
\end{Definition}

\begin{Definition}
The {\color{black}global {\color{black} solution}} $\gamma$ is bounded  if the set $\gamma(\mathbb{R})=\{\gamma(t)\in X:t\in \mathbb{R}\}$ is bounded. If $x=\gamma(0)$ we say that $\gamma:\R\to X$ is a global {\color{black} solution} through $x$.
\end{Definition}

The non-dissipativity of the problem can be manifested threefold:
\begin{itemize}
	\item[(A)] {\color{black} solutions} may cease to exist (e.g.\!\! \textit{blow up} to infinity) in finite time, that is, for some $x\!\in\! X$,\! $\tau_x\!<\!\infty$. 
	
    \item[(B)] {\color{black} solutions} may \textit{grow-up}, that is, it could happen that $\tau_x=\infty$ for some $x\in X$, with the set $\{ S(t)x: t\geq 0\}$ being unbounded (this is the case in Example \ref{ex1}),
    \item[(C)] all {\color{black} solutions} may stay bounded but, even so, there may not exist a bounded set which absorbs all of them (this is the case in Example \ref{ex2}).
\end{itemize}

We will assume that $\tau_x=\infty$ for all $x\in X$, that is, that $\{S(t)\}_{t\geq 0}$ is a semigroup.  Hence we exclude the situation of item (A) above.

Systems which exhibit the behaviors (B) or (C) for some {\color{black} solutions} are called \textit{slowly non-dissipative}, cf. \cite{bengal}.
In case of grow-up, i.e. in the case (B),  it is still possible to distinguish between two situations: either for some $x\in X$ we have that $\lim_{t\to \infty}d(S(t)x,x) = \infty$ ({\color{black} solution} diverge) or $\limsup_{t\to \infty}d(S(t)x,x)=+\infty$ but the {\color{black} solution} do not diverge (oscillate).} To see that the latter situation is possible, a simple non-autonomous example is 
$$
x'(t) = \sin(t) - t\cos(t),
$$
then if $x(0) = 0$ the {\color{black} solution} is $x(t) = t\sin(t)$, so we have the oscillations with the  amplitude growing to infinity as $t\to \infty$. Later we will introduce the assumption on the semigroup which exclude such situation. 

We define two notions of invariant sets for non-dissipative semigropups. The first definition follows \cite{chepyzhov} where it is called the maximal invariant set.
\begin{Definition}
Let $\{ S(t) \}_{t\geq 0} $ be a continuous semigroup. The set $\mathcal{I}$ is the maximal invariant set if
$$
\mathcal{I} = \{ x\in X\,:\textrm{ there is a global {\color{black} solution} } \gamma:\R\to X \ \textrm{through } x \textrm{ which is bounded in the past}\}  
$$
\end{Definition}
It is also possible to define another notion, which coincides with the classical concept in the theory of global attractors. 
%The second notion comes from \cite{foiaslariosbiswas} where it is called the weak-sigma attractor.
\begin{Definition}
Let $\{ S(t) \}_{t\geq 0} $ be a continuous semigroup. The set $\mathcal{I}_b$ maximal invariant set of bounded {\color{black} solutions} if
$$
\mathcal{I}_b = \{ x\in X\,:\textrm{ there is a global bounded {\color{black} solution} } \gamma:\R\to X \textrm{through } x \}  
$$
\end{Definition}
We make a simple observation.

\begin{Observation}
	We have $\mathcal{I}_b \subset \mathcal{I}$ with the possibility of strict inclusion. {\color{black} $\mathcal{I}$ is unbounded and $\mathcal{I}_b$ may be unbounded}. Both sets are invariant, i.e. $S(t) \mathcal{I} = \mathcal{I}$ and $S(t) \mathcal{I}_b = \mathcal{I}_b$ for all $t\geq 0$. Furthermore,   
	$$
	x \in \mathcal{I}\setminus \mathcal{I}_b \Longleftrightarrow x \in \mathcal{I}\ \ \textrm{and} \ \ \{ S(t)x\}_{t\geq 0}\ \ \textrm{is unbounded}.  
	$$
\end{Observation}

In Example \ref{ex2} above $\mathcal{I} = \mathcal{I}_b = \mathbb{R}^2$. In Example \ref{ex1} above $\mathcal{I} = \{0\}\times \mathbf{R}$ and $\mathcal{I}_b =\{ (0,0) \}$. Note that in Example \ref{ex1} $\mathcal{I}\setminus \mathcal{I}_b$ consists of heteroclinics to infinity as defined by \cite{bengal}, see also \cite{hell} and the articles  \cite{PimentelRocha, Pimentel, carvalhopimentel, carvalhopimentel2}.

\subsection{Existence of unbounded attractors}
The content of this chapter relies in putting to the abstract setting the results of Chepyzhov and Goritskii \cite{chepyzhov} obtained there for the particular initial and boundary value problem governed by the PDE $u' = \Delta u + \beta u + f(u)$ on a bounded domain with the homogeneous Dirichlet boundary data. We will frequently use in the proofs the following result which appears in \cite[Lemma 2.1]{chepyzhov}.
\begin{Lemma}\label{lemma:212}
	Let $\{L_n\}$ be a sequence of nested sets in a Banach space $X$, i.e. 
	$$
	L_1 \supset L_2 \supset L_3 \supset \ldots,$$
	such that every set $L_n$ lies in the $\epsilon_n$-neighborhood of some compact subset $K_n$, namely $L_n\subset \mathcal{O}_{\epsilon_n}(K_n)$, with $\lim_{n\to \infty}\epsilon_n= 0$. Then from any sequence of points $\{y_n\}$, where $y_n\in L_n$, we can choose a convergent subsequence. Moreover the set $L = \bigcap_{n=1}^\infty \overline{L}_n$ is compact.
\end{Lemma}

\begin{Definition}\label{def:gen_ass}
The semigroup $\{S(t)\}_{t\geq 0}$ is generalized asymptotically compact if for every $B\in \mathcal{B}(X)$ and every $t>0$ there exists a compact set $K(t,B)\subset X$ and  $\varepsilon(t,B)\to 0$ as $t\to \infty$ such that 
$$
S(t)B\subset \mathcal{O}_\varepsilon(K(t,B)).
$$
\end{Definition}

%\begin{Definition}
%The semigroup $\{S(t)\}_{t\geq 0}$ is generalized $\omega$-limit compact if for every $B\in \mathcal{B}(X)$ 
%there holds
%$$
%\lim_{t\to \infty} \alpha (S(t)B) = 0.
%$$
%\end{Definition}
%
%\begin{Lemma}
%The semigroup $\{S(t)\}_{t\geq 0}$ is generalized asymptotically compact if and only if it is generalized $\omega$-limit compact.
%\end{Lemma}
%\begin{proof}....
%\end{proof}

%\begin{Lemma}
%If the semiflow $\{ S(t) \}_{t\geq 0}$ is generalized asymptotically compact and ..... then if $\{ y_n \}$ is a a sequence .... then for $t_n\to \infty $ the sequence $\{ S(t_n)y_n \}$ has a convergent subsequence.
%\end{Lemma}

Now let $X$ be a Banach space and let $E^+$ and $E^-$ its subspaces. We assume that $E^+$ is finite dimensional and that $E^-$ is closed. Moreover, $X= E^+ \oplus E^-$, i.e. we can uniquely represent any $x\in X$ as $x=p+q$ with $p\in E^+$ and $q\in E^-$. Then we will denote $p = Px$ and $q=(I-P)x$. The closed graph theorem implies that the projections $P$ and $I-P$ are continuous. 
We will use the shorthand notation 
 $\{ \|Px\| \leq R\} = \{  x\in X\;:\  \|Px\|\leq R \}$  and $\{ \|(I-P)x\|\leq D\} = \{ x\in X\;:\  \|(I-P)x\|\leq D \}$.  
We will impose the following assumptions to the semigroup $\{S(t)\}_{t\geq 0}$, in order to prove the existence of an unbounded attractor 
\begin{itemize}
	\item[(H1)] There exist $D_1, D_2 > 0$ and a closed set $Q$ with 
	$\{ \|(I-P)x\|\leq D_1\} \subset Q \subset \{ \|(I-P)x\|\leq D_2\}$
	such that $Q$ absorbs bounded sets, and is positively invariant, i.e. for every $B\in \mathcal{B}(X)$ there exists $t_0(B) > 0$ such that $\bigcup_{t\geq t_0}S(t)B \subset Q$ and $S(t)Q \subset Q$ for every $t\geq 0$.
	\item[(H2)] There exist the constants  $R_0$ and $R_1$ with $0<R_0\leq R_1$ and an ascending family of closed and bounded sets $\{H_R\}_{R\geq R_0}$ with $H_R \subset Q$, such that
	\begin{enumerate}
		\item  for every $R\geq R_1$ we can find $S(R)\geq R_0$ such that  $\{ \|Px\|\leq S(R)\} \cap Q \subset H_R$ and moreover $\lim_{R\to \infty} S(R) = \infty$,
		\item for every $R\geq R_1$ we have $ H_{R}  \subset \{ \|Px\|\leq R\}$,
		\item and $S(t)(Q \setminus H_R) \subset Q \setminus H_R$ for every $t\geq 0$.
	\end{enumerate}
	\item[(H3)] The semigroup $\{S(t)\}_{t\geq 0}$ is generalized asymptotically compact.
\end{itemize}
Define the unbounded attractor as 
\begin{equation}\label{jdef}
\mathcal{J} = \bigcap_{t\geq 0}\overline{S(t)Q}
\end{equation}
It is clear that the unbounded attractor is a closed set. We begin with establishing the relation between the maximal invariant set $\mathcal{I}$ and the unbounded attractor $\mathcal{J}$. The argument follows the lines of the proof of \cite[Proposition 2.7 and Proposition 2.10 c)]{chepyzhov}. 
\begin{Theorem}\label{coincide}
	If (H1)--(H3) hold, then  $\mathcal{I} = \mathcal{J}$. 
	\end{Theorem}
 \begin{proof}  
 	We first prove that $\mathcal{I}\subset \mathcal{J}$. To this end assume that $u \in \mathcal{I}$ and denote by $\gamma$ the bounded in the past {\color{black} solution} such that $u = \gamma(0)$. Moreover let $T$ be such that $S(T)\{ \gamma(s)\,:\ s\leq 0 \} \subset Q$. If $t\geq 0$ then  $S(T+t)\{ \gamma(s)\,:\ s\leq 0 \} \subset S(t)Q$. But, since $u = S(T+t)\gamma(-T-t)$, then $u\in S(t)Q$ and, consequently $u\in \mathcal{J}$. 
 
% {\color{blue}
% I have problems to ensure that $J$ is non-empty. It appears to me that we need to assume that there is at least one global bounded {\color{red} solution} in $\{\|(I-P)x\|\leq S(R)\}\cap \{\|Px\|\leq D_1\}$ (for some $R$) and then use the reasoning below. 
% 
% To see that $\mathcal{J}$ is non-empty we observe that $S(t)H_R\cap H_R$ must be non-empty for all $t\geq 0$. This follows from the fact that $\{\|(I-P)x\|\leq S(R)\}\cap \{\|Px\|\leq D_1\}$ is connected and contained in $H_R$. Now it is enough to consider sequences $t_n\stackrel{n\to\infty}{\longrightarrow}\infty$ and $u_n \in H_R$ and use the generalized asymptotic compactness to conclude that there is a convergent subsequence.}
% 

 	We establish the converse inclusion. To this end {\color{black}take} $u\in \mathcal{J}$. Then there exists a sequence $\{y_n\} \subset Q$ such that $S(n)y_n \to u$ as $n\to \infty$. Since $\{ S(n)y_n \}$ is a convergent sequence, it is bounded, it follows that $\{ S(n)y_n \} \subset H_R$ for some $R\geq R_0$ and $u\in H_R$. By item 3 of (H2) this means that for every $n$ and for every $t\in [0,n]$ we have $S(t)y_n \in H_R$. As $S(n)y_n = S(1)S(n-1)y_n$ we deduce that $S(n-1)y_n \in H_R\cap S(n-1)H_R = H_R \cap S(n-1)Q$. By (H3) and by Lemma \ref{lemma:212}, it follows that, for a subsequence, $S(n-1)y_n \to z$, whence $S(1)z = u$ and $z\in H_R$.  Picking $t\geq 0$, for $n$ large enough $S(t)Q \ni S(t)S(- t+n-1)y_n =  S(t - t+n-1)y_n \to z$ and hence $z\in \mathcal{J}$. Proceeding recursively, we can construct the {\color{black}{\color{black} solution}} $\{ \gamma(t)\,:\ t\leq 0\} \subset Q$ such that $\gamma(0) = u$. As $u\in H_R$, by item 3 of (H2) it must be that $\gamma(t) \in H_R$ for every $t\leq 0$ and the proof is complete. 
 	\end{proof}

We formulate a simple lemma on the possible behavior of bounded sets upon evolution.
\begin{Lemma}\label{lemma:possibilities}
	Assume (H1)-(H2) and let $R\geq R_0$. If $B\in \mathcal{B}(X)$ then exactly one of three cases holds:
	\begin{enumerate}
		\item there exists $t_1 > 0$ such that for every $t\geq t_1$ 
		$$ S(t)B \subset Q \setminus H_R.$$ 
			\item there exists $t_1 > 0$ such that for every $t\geq t_1$ 
		$$ S(t)B  \cap H_R  \neq \emptyset\ \ \textrm{and}\ \  S(t)B  \cap (Q \setminus H_R)  \neq \emptyset.$$
			\item there exists $t_1 > 0$ such that for every $t\geq t_1$ 
		$$ S(t)B  \subset H_R .$$
	\end{enumerate}
\end{Lemma} 
\begin{proof}
	Suppose that (3) does not hold. This means that there exists $t_1 > t_0(B)$ such that $S(t_1)B \cap (Q\setminus H_R) \neq \emptyset$. By (H2) we deduce that $S(t)B \cap (Q\setminus H_R) \neq \emptyset$ for every $t\geq t_1$. If for some $t_2\geq t_1$ there holds $S(t_2) B \cap H_R = \emptyset$ then $S(t_2)B\subset Q\setminus H_R$ and (H2) implies that $S(t)B\subset Q\setminus H_R$ for every $t\geq t_2$, which completes the proof.
\end{proof}

The proof of the following result is established in \cite[Lemma 2.8]{chepyzhov}, here we present the proof that makes the explicit use of the Brouwer degree. This result, in particular, implies that the set $\mathcal{I} = \mathcal{J}$ is nonempty.  
\begin{Lemma}\label{lemma:211}
	Assume (H1)-(H3). For every $p \in E^+$ there exists $q\in E^-$ such that $p+q \in \mathcal{J}$.
\end{Lemma}
\begin{proof}
	Let $p\in E^+$. By item 1 of (H2) there exists $R\geq R_0$ such that $\{x\in E^+ \,:\ \|x\|\leq \|p\|+1 \} \subset H_R$. Define $B = \{ x\in E^+\, :\ \|x\| < R + 1\}$. Since, by item 2 of (H2), $p \in B$ it is clear that the Brouwer degree $\textrm{deg}(I,B,p)$ is equal to one, cf. \cite[Theorem 1.2.6 (1)]{Agarwal}. Pick $t>0$ and define the mapping 
	$[0,1]\times \overline{B} \ni (\theta,p) \mapsto P S(\theta t) p \in E^+$. This mapping is continuous. Moreover, as, by item 3 of (H2), $\partial B = \{ x\in E^+\, :\ \|x\| = R + 1 \} \subset Q\setminus H_R$ we deduce that $p \not\in PS(\theta t)\partial B$ for every $\theta \in [0,1]$. Hence, by the homotopy invariance of the Brouwer degree, cf. \cite[Theorem 1.2.6 (3)]{Agarwal} we deduce that $\textrm{deg}(P S(t),B,p) = 1$, whence, cf. \cite[Theorem 1.2.6 (2)]{Agarwal}, there exists $p_t\in B$ such that $PS(t) p_t = p$. Let $t_n\to \infty$. There exists a sequence $p_n \in B$ such that $PS(t_n)p_n = p$. Hence we can find $q_n \in E^-$ such that $y_n = p+q_n = S(t_n)p_n \in S(t_n) Q$.  From the fact that $p\in B$ we deduce by item 1 of (H2) that there exists $R' \geq R_1$ such that $y_n \in H_{R'}$ for every $n$. This means that $y_n \in S(t_n)H_{R'} \cap H_{R'}$. As $H_{R'}$ is bounded and the sequence of sets $S(t_n)H_{R'} \cap H_{R'}$ is nested, we can use (H3) and Lemma \ref{lemma:212} to deduce that, for a subsequence, $y_n \to y$. Since for every $t\geq 0$ we have $y_n \in S(t)Q$ for every $n$ such that $t\leq t_n$ it follows that $y\in \mathcal{J}$. The continuity of $P$ imples that $Py = p$ and the proof is complete. 
\end{proof}

The proof of the following result uses the argument of \cite[Proposition 3.4]{chepyzhov}
\begin{Theorem}\label{attraction}
	Let (H1)-(H3) hold. Let $R\geq R_1$ and let $B\in \mathcal{B}(X)$. Then either there exist $t_1>0$ such that for every $t\geq t_1$ we have $S(t)B\subset Q\setminus H_{R}$ or 
	$$\lim_{t\to\infty}\mathrm{dist}(S(t)B\cap \{\|Px\|\leq R\},\mathcal{J}) = 0.$$
\end{Theorem}
\begin{proof}
We can assume that there exists $t_1>0$ such that  $S(t)B \cap H_{R}\neq \emptyset$ for every $t\geq t_1(B)$ (i.e. 2. or 3. Lemma \ref{lemma:possibilities} hold). Then for $t\geq t_1$ the sets $S(t)B \cap H_{R} \subset S(t)B \cap \{ \|Px\| \leq R\}$ are nonempty. 
	For the proof by contradiction let us take the sequences $\{x_n\}\subset B$ and $t_n \to \infty$ such that $S(t_n)x_n \in S(t_n)B \cap \{ \|Px\|\leq R \}$ for every $n$, and
	$$\mathrm{dist} (S(t_n)x_n,\mathcal{J})>\varepsilon$$
	for some $\varepsilon>0$. By (H1) we can assume without loss of generality that $B\subset Q$. Denote $y_n=S(t_n)x_n$. Then  by item 1 of (H2) we can find $\overline{R}\geq R_0$ such that $\{y_n\}\subset H_{\overline{R}}$ and it is a bounded sequence. By item 3 of assumption (H2) we know that $\{x_n\}\subset H_{\overline{R}}$. We define the nested sequence of sets $L_n=S(t_n)Q\cap H_{\overline{R}}=S(t_n)H_{\overline{R}}\cap H_{\overline{R}}$, and by  assumption (H3) we know that for every $n$ there exist a compact set $K_n$ such that $L_n\subset \mathcal{O}_{\epsilon_n}(K_n)$, and $\varepsilon_n\rightarrow 0$ when $n\rightarrow \infty$. By Lemma \ref{lemma:212} there exists $y$ such that $y_n\rightarrow y$, so by definition, $y\in \mathcal{J}$, and then, we  arrive to a contradiction. 
\end{proof}
The above result guarantees only the attraction in the bounded sets. The next example of an ODE in 2D shows that under our assumptions the 'classical' notion of attraction, in the sense  $\lim_{t\to\infty}\mathrm{dist}(S(t)B,\mathcal{J}) = 0$, cannot hold. 
\begin{Example}
The vector field for the next system of two ODEs is defined only in the quadrant $\{ x\geq 0, y\geq 0\}$. In the remaining three quadrants it can be extended by the symmetry.
\begin{align*}
	& x'(t) = x(t),\\
	& y'(t) = \begin{cases}
		-y(t) + \frac{x(t)y(t)}{1+x(t)}\ \ \textrm{if}\ \ y(t)\in [0,1],\\
		-y(t) + (2-y(t))\frac{x(t)y(t)}{1+x(t)}\ \ \textrm{if}\ \ y(t)\in [1,2],\\
		-y(t)\ \ \textrm{if}\ \ y(t)\geq 2.
	\end{cases}
\end{align*}
Then 
\begin{itemize}
	\item the maximal invariant set is given by $\mathcal{J} = \{(x,0)\,:\ x\in \mathbb{R}\}$,
	\item the set $Q = \{ (x,y)\in \mathbb{R}^2\,:\ |y|\leq 2 \}$ is absorbing and positively invariant,
	\item the sets $\{  (x,y)\in \mathbb{R}^2\,:\ |y|\leq 2,|x|>a \}$ are positively invariant for every $a>0$,
	\item the image of a bounded  set by $S(t)$ is relatively compact, and we have the continuous dependance on initial data,
	\item the set $\mathcal{J}$ is not attracting as one can construct {\color{black} {\color{black} solutions}} satisfying $x(t)\to \infty$ and $y(t) \to c$ for every $c\in (0,1)$.
\end{itemize}
\end{Example}

Finally, the following result on $\sigma$-compactness is a counterpart of \cite[Proposition 2.6]{chepyzhov}
\begin{Lemma}\label{lem:compact}
	For every bounded and closed set  $B$ the set $\mathcal{J}\cap B$ is compact. Hence,    $\mathcal{J}$ is a countable sum of compact sets.
\end{Lemma}
\begin{proof}
	Since $\mathcal{J}\cap B$ is closed it is enough to show that it is relatively compact. To this end, it is sufficient to prove that for every $R\geq R_0$ the set $\mathcal{J}\cap H_R $ is relatively compact. Assume that $\{ u_n\} \subset \mathcal{J}\cap H_R$ is a sequence. By invariance of $\mathcal{J}$ there exists $x_n\in \mathcal{J}$ such that $u_n = S(n) x_n$. It must be that $x_n\in H_R$ and hence $u_n\in S(n)H_R \cap H_R = S(n)Q \cap H_R$. Using (H3) and Lemma  \ref{lemma:212} we obtain the assertion. 
\end{proof}
The last result is on the unbouded attractor minimality
\begin{Lemma}
	Assume (H1)-(H3). If $C$ is a closed set such that  for every $B\in \mathcal{B}(X)$ for which there exists $t_1>0$ and $R\geq R_1$ such that $S(t)B\cap \{\|Px\|\geq R\}$ is nonempty for $t\geq t_1$ we have $\lim_{t\to\infty}\mathrm{dist}(S(t)B\cap \{\|Px\|\leq R\},C) = 0,$ then $\mathcal{J}\subset C$. 
\end{Lemma}
\begin{proof}
	Suppose, for contradiction, that $v \in \mathcal{J} \setminus C$. There exists the global {\color{black} solution} $\gamma$ with $\gamma(0) = 0$ such that $\gamma((-\infty,0])$ is a bounded set. Moreover there exists $R\geq R_1$ such that $v = S(t)\gamma(-t) \in S(t) \gamma((-\infty,0]) \cap \{\|Px\|\leq R \}$. Hence
	$$\text{dist}(v,C) \leq  \lim_{t\rightarrow\infty} \text{dist}(S(t) \gamma((-\infty,0]) \cap \{\|Px\|\leq R \},C) =0,$$
	whence it must be that $v\in C$ which concludes the proof. 
\end{proof}
We summarize the results of this section in the following result.
\begin{Theorem}\label{th:autonomus1}
	Assume (H1)-(H3). Then the set $\mathcal{J}$ given by \eqref{jdef} is the unbounded attractor for the semiflow $\{ S(t) \}_{t\geq 0}$, that is:
	\begin{itemize}
		\item $\mathcal{J}$ is nonempty, closed, and invariant,
		\item intersection of $\mathcal{J}$ with any bounded set is relatively compact,
		\item if for some $R\geq R_1$, $B\in \mathcal{B}(X)$, and $t_1>0$ the sets $S(t)B \cap \{ \|Px\|\leq R \}$ are nonempty for every $t\geq t_1$ then 
		$$
		\lim_{t\to \infty} \mathrm{dist}(S(t)B\cap \{\|Px\|\leq R\},\mathcal{J}) = 0,
		$$
		and $\mathcal{J}$ is the minimal closed set with the above property.  
	\end{itemize}
Moreover $P\mathcal{J}=E^+$ and $\mathcal{J}$ is the maximal invariant set for $\{ S(t) \}_{t\geq 0}$, that is it consists of those points through which there exists a global {\color{black} solution} bounded in the past. 
	\end{Theorem}

\subsection{Attraction of all bounded sets}\label{attraction_all}
We will impose additional assumption.
\begin{enumerate}
    \item[(A1)] There exists a closed set  $Q_1\subset Q$ such that:
    \begin{enumerate}
        \item For every $\varepsilon>0$ there exist $r>0$ such that for every $p\in E^+$ for which $\norm{p}\geq r$ there holds
        $\mathrm{diam }( Q_1\cap\{x\in X: Px=p\} )\leq \varepsilon$.
        \item The set $Q_1$ is attracting, that is for every $B\in \mathcal{B}(x)$  we have 
        $\lim_{t\to\infty}\mathrm{dist}(S(t) B, Q_1) = 0.$
    \end{enumerate}
    
\end{enumerate}
We will prove following theorem.
\begin{Remark}\label{rem:1}
If the conditions (H1)-(H3) and (A1) hold, then $\mathcal{J}\subset Q_1.$
{\color{black} In fact, if $u\in \mathcal{J}=\mathcal{I}$, there is a global {\color{black} solution} $\gamma:\mathbb{R}\to X$ through $u$ ($\gamma(0)=u$) which is bounded in the past. As $\gamma((-\infty,0]) \subset S(t)\gamma((-\infty,0])$ then $\textrm{dist}(\gamma((-\infty,0]),Q_1) \leq \textrm{dist}(S(t)\gamma((-\infty,0]),Q_1)\to 0$ as $t\to \infty$. Hence, as $Q_1$ is closed $u\in Q_1$. }
\end{Remark}

\begin{Theorem}\label{attracts:all}
Assume that (H1)-(H3) and (A1) hold. Then for every bounded set $B$ we have
\begin{equation*}
  \lim_{t\to\infty}
  \mathrm{dist}(S(t)B,\mathcal{J})= 0.  
\end{equation*}
\end{Theorem}
\begin{proof}
Assume that there are sequences $t_n\to \infty$ and $u_n\in B$ such that $\inf_{j\in \mathcal{J}}\norm{ S(t_n)u_n- j}>\varepsilon,$ for some $\varepsilon<1.$
By Theorem \ref{th:autonomus1} it follows that $\norm{P S(t_n)u_n}\to \infty.$ By item (a) of (A1) 
we can pick $r>0$ such that  for every $x_0\in X$ for which $\norm{Px_0}>r$
we have $\mathrm{diam }(Q \cap \{x\in X: Px= Px_0\}) \leq \frac{\varepsilon}{4}.$ Then we pick $n_0$ such that $S(t_{n_0})u_{n_0} \in S(t_{n_0}) B \subset \mathcal{O}_{\frac{\varepsilon}{8}}(Q_1)$ and 
$\norm{ PS(t_{n_0})u_{n_0}}\geq r+1.$ So we find $v\in Q_1 $ such that 
$\norm{ S(t_{n_0})u_{n_0}-v}\leq \frac{\varepsilon}{4}.$
 Observe that $\norm{Pv} \geq r+\frac{1}{2}.$ Indeed if $\norm{Pv}$ would be less than $r+\frac{1}{2}$ then we would have that 
$\norm{ S(t_{n_0})u_{n_0}-v}\geq 
\norm{P(S(t_{n_0})u_{n_0}-v)}\geq \frac{1}{2}.$
By Lemma \ref{lemma:211} we can find $w\in\mathcal{J}$ such that $Pv = Pw.$ Observe that $w,v\in Q_1 \cap \{x\in X: Pw = Px\}$ and $\norm{Pv}>r.$ We deduce that $\norm{w-v}\leq \frac{\varepsilon}{4}.$
We observe that
$$\norm{S(t_{n_0})u_{n_0}-w}\leq 
\norm{S(t_{n_0})u_{n_0}-v}+\norm{v-w} 
\leq \frac{\varepsilon}{2}.
$$
which is contradiction.  
\end{proof}

\subsection{Structure of the unbounded attractor}\label{sec:structure}
Since the unbounded attractor is an invariant set, there exists a global (eternal) {\color{black} solution} bounded in the past in the attractor through each point in it. Some of these {\color{black} solutions} may be bounded and others may be unbounded in the future, and, due to the forward uniqueness, for a given point in $\mathcal{J}$ only one of these two possibilities may occur.  We remember, that a bounded invariant set $\mathcal{I}_b$ consisted of those points through which there exist the complete bounded {\color{black} solutions}. This gives us the decomposition of the unbounded attractor  into two disjoint sets $\mathcal{J} = \mathcal{I}_b \cup (\mathcal{J} \setminus \mathcal{I}_b)$. We will provide the properties of this decomposition. Observe that $\mathcal{J}$, as a nonempty and closed subset of a Banach space, is a complete metric space. We denote by $\mathcal{B}(\mathcal{J})$ the family of nonempty and bounded subsets of $\mathcal{J}$. We shall start from the definition of the $\alpha$-limit set.
\begin{Definition}
	Let $B\in \mathcal{B}(\mathcal{J})$. We define the $\alpha$-limit set of $B$ as
	$$
	\alpha(B) = \{  u\in \mathcal{J}\,:\ \textrm{there exists}\ t_n\to \infty\  \textrm{and} \ \ u_n \in \mathcal{J}\ \ \textrm{such that} \ \ S(t_n)u_n\in B\ \ \textrm{and}\ \ u_n\to u\}.
	$$
\end{Definition}
In the next result we establish the properties of the $\alpha$-limit set.
\begin{Theorem}\label{thm3}
	Let $B\in \mathcal{B}(\mathcal{J})$ and assume (H1)--(H3). Then $\alpha(B)$ is nonempty, compact, invariant, and attracts $B$ in the past, i.e. 
	$$
	\lim_{t\to \infty} \mathrm{dist}(S(t)^{-1}B,\alpha(B)) = 0. 
	$$
\end{Theorem}
\begin{proof}
	Pick $B\in \mathcal{B}(\mathcal{J})$. Let $u_n \in \mathcal{J}$ and $t_n\to \infty$ be such that $S(t_n)u_n \in B$. As $B$ is bounded, then $B\in H_R$ for some $R$ and it must be that $u_n \in H_R$. As $H_R\cap \mathcal{J}$ is a compact set it must be that, for a subsequence,  $u_n \to u$ for some $u\in \mathcal{J}\cap H_R$ and hence $\alpha(B)$ is nonempty. Moreover $\alpha(B) \subset \mathcal{J} \cap H_R$ so it must be relatively compact. To prove that it is closed pick a sequence $\{v_n\} \subset \alpha(B)$ with $v_n\to v$. There exist sequences $t_n^k\to \infty$ as $k\to \infty$ and $\{ v^n_k\}_{k=1}^\infty \subset \mathcal{J}$ with $S(t^n_k)v^n_k\in B$ and $v^n_k \to v_n$ as $k\to \infty$. It is enough to use a diagonal argument to get that $v\in \alpha(B)$. Hence $\alpha(B)$ is a compact set. To establish the invariance let us first prove that $S(t)\alpha(B) \subset \alpha(B)$. To this end let $u\in \alpha(B)$ and $t>0$. By continuity, $S(t)u_n \to S(t)u$, and $S(t_n-t)S(t)u_n = S(t_n)u_n \in B$, whence $S(t)u \in \alpha(B)$. To establish negative invariance also take $u\in \alpha(B)$ and $t>0$. There exist $v_n \subset \mathcal{J}\cap H_R$ such that $S(t)v_n = u_n$. By compactness, for a subsequence, $v_n\to v$ and by continuity $S(t)v_n \to S(t)v$, whence $S(t)v = u$. Moreover $S(t_n+t)v_n = S(t_n)u_n \in B$, whence $v\in \alpha(B)$, which ends the proof of invariance. The proof of attraction is also standard. For contradiction assume that there exists $t_n\to \infty$ and $u_n\in \mathcal{J}$ with $S(t_n)u_n\in B$ such that $\mathrm{dist}(u_n,\alpha(B)) \geq \varepsilon$ for every $n$ and some $\varepsilon > 0$. As $u_n \in \mathcal{J}\cap H_R$, by compactness we obtain that $u_n\to u$, for a subsequence and it must be $u\in\alpha(B)$ which is a contradiction.
\end{proof}
 An alternative approach to prove the above result would be to inverse time, which is possible on the unbounded attractor, and consider the $\omega$-limit set of inversed in time dynamical system, which can be possibly multivalued, as we do not assume the backward uniqueness. 
 
 As a consequence of the above result we obtain the following corollary
 \begin{Corollary}\label{cor1}
 	For every $B \in \mathcal{B}(\mathcal{J})$ there holds $\alpha(B)\subset \mathcal{I}_b$. In consequence, $\mathcal{I}_b$ is nonempty. 
 \end{Corollary}
 Define the set 
 $$
 \mathcal{J}_b = \bigcup_{R\geq R_0}\alpha(\mathcal{J}\cap H_R). 
 $$
For elemets of $\mathcal{J}$ we define $S(t)^{-1}$ as the inverse in $\mathcal{J}$ which, as we do not assume backward uniqueness, can be possibly multivalued. 
 \begin{Theorem}
 	There holds $\mathcal{J}_b = \mathcal{I}_b$. Moreover for every $B\in \mathcal{B}(\mathcal{J})$ there holds
 	 	$$
 	 \lim_{t\to \infty} \mathrm{dist}(S(t)^{-1}B, \mathcal{J}_b) = 0. 
 	 $$
 \end{Theorem}
\begin{proof}
	If $u \in \mathcal{J}_b$ then $u\in \alpha(\mathcal{J}\cap H_R)$ for some $R$, an invariant and compact set. Hence there exists a bounded global {\color{black} solution} through $u$ in $\mathcal{J}$ and thus $u\in \mathcal{I}_b$. To get the opposite inclusion assume that $u\in \mathcal{I}_b$ whence $\{S(t)u\}_{t\geq 0}$ is bounded and there exists $R>0$ such that $S(t)u \in \mathcal{J}\cap H_R$ for every $t\geq 0$. It means that $u\in \alpha (\mathcal{J}\cap H_R)$.
	
	To get the backward attraction observe that 
	$B\subset \mathcal{J}\cap H_R$ for some $R$ and hence $\alpha(B) \subset \alpha(J\cap H_R)$ which yields the assertion. 
\end{proof}

It is easy to see that if $\mathcal{J}_b$ is bounded then it has to be closed. In general, however, $\mathcal{J}_b$ can be unbounded and then it does not have to be a closed set as seen in the following example. 
\begin{Example}
	The vector field is defined only in the quadrant $\{ x\geq 0, y\geq 0 \}$. It can be extended to other quadrants by the symmetry.
Namely, we consider the following system of two ODEs,
\begin{align*}
&	y' = \begin{cases}
		&-y + 1 \ \ \textrm{for}\ \ y\geq 1,\\
		& 0 \ \ \textrm{for}\ \ y\in [0,1].
		\end{cases}\\
& \textrm{if}\ \ y>0\ \ \textrm{then}\ \	x' = \begin{cases}
		& \max\{3y,1\}x\ \ \textrm{for}\ \ x\in \left[0,\min\left\{\frac{1}{3y},1\right\}\right],\\
		& 1\ \ \textrm{for}\ \ x\in \left[\min\left\{\frac{1}{3y},1\right\},\frac{2}{3y}\right],\\
		& -3yx+3\ \ \textrm{for}\ \ x\in \left[\frac{2}{3y},\frac{1}{y}\right],\\
		& x-\frac{1}{y}\ \ \textrm{for}\ \ x\geq \frac{1}{y}.
	\end{cases} \ \ \textrm{and} \ \ x'=\max\{x,1\}\  \textrm{for}\ \ y=0.
\end{align*}
The unbounded attractor is given by $\mathcal{J} = \R\times [-1,1]$, and the bounded invariant set is given by 
$$\mathcal{J}_b = \{ (0,0)\} \cup \bigcup_{|y|\in (0,1]} \left[-\frac{1}{|y|},\frac{1}{|y|}\right]\times \{ y\}.$$\end{Example} We have proved that $\mathcal{J}_b$ attracts in the past the bounded sets in $\mathcal{J}$. We will next show that it attracts in the future such bounded sets from $X$, which upon evolution stay in a bounded set, namely we prove the following result.
    \begin{Theorem}\label{thm:omega}
    	If $B\in \mathcal{B}(X)$ is such that there exist $t_1>0$ and $R\geq R_0$ with $S(t)B\subset H_R$ for $t\geq t_1$, i.e. the case (3) from Lemma \ref{lemma:possibilities} holds, then
    	$$
    	\lim_{t\to \infty}\mathrm{dist}(S(t)B,\mathcal{J}_b) = 0.
    	$$
    	 \end{Theorem}
     \begin{proof}
Define the $\omega$-limit set $\omega(B)=\bigcap_{s\geq t_1}\overline{\bigcup_{t\geq s}S(t)B}$. We will show that this set is compact, nonempty, invariant, and it attracts $B$, whence it must be that $\omega(B) \subset \mathcal{J}_b$ and the assertion will follow. Observe that the family of sets 
$$
\left\{ \overline{\bigcup_{t\geq s}S(t)B} \right\}_{s\geq t_1}
$$
is decreasing. Moreover
$$
S(t)B = S(t_1 + t-t_1) B \subset S(t-t_1)H_R \cap H_R = S(t-t_1)Q \cap H_R.
$$
Hence 
$$
\overline{\bigcup_{t\geq s}S(t)B} \subset \overline{\bigcup_{t\geq s}S(t-t_1)Q} \cap H_R = \overline{S(s-t_1)Q} \cap H_R.
$$
If $x\in \overline{S(s-t_1)Q} \cap H_R$ then $x\in H_R$ and $x = \lim_{n\to \infty} x_n$, where $x_n \in S(s-t_1)Q$. We deduce that, for suficiently large $n$, there holds $x_n \in H_{\overline{R}}$, where $\overline{R}$ depends only on $R$. This means that $x_n\in S(s-t_1)H_{\overline{R}} \cap H_{\overline{R}}$, whence $x\in \overline{S(s-t_1)H_{\overline{R}}} \cap H_{\overline{R}}$. We are in position to use (H3) and Lemma \ref{lemma:212} to deduce that $\omega(B)$ is nonempty and compact. The proof that $\omega(B)$ attracts $B$ is standard and follows the lines of the argument in Theorem \ref{attraction}. Also the proof of invariance of $\omega(B)$ is classical, once we have the compactness, and follows the lines of the proof in Theorem \ref{thm3}.   
     \end{proof}
 In a straightforward way, by taking as $B$ the sets $\alpha(\mathcal{J}\cap H_R)$, we obtain the following characterization of $\mathcal{J}_b$
 $$
 \mathcal{J}_b = \bigcup_{\stackrel{B\in \mathcal{B}(X)}{ B \ \textrm{is type 3.}}} \omega(B),
 $$
 where the summation is made over all bounded sets which for some $R\geq R_0$ and some $t_1 > 0$ satisfy the assertion 3. of Lemma \ref{lemma:possibilities}. 
 
Remembering the decomposition $\mathcal{J} = \mathcal{J}_b \cup (\mathcal{J}\setminus \mathcal{J}_b)$ we know from Corollary \ref{cor1} that the set  $\mathcal{J}_b$ is nonempty. The set $\mathcal{J}\setminus \mathcal{J}_b$ can, however, be empty as the following example shows.
\begin{Example}Consider the following system of two ODEs
\begin{align*}
&	y' = \begin{cases}
		&-y+1\ \ \textrm{as}\ y>1,\\
		& 0\ \ \textrm{as}\ \ |y|\leq 1,\\
		& y + 1 \ \ \textrm{as}\ y<-1.
	\end{cases},\\
&x'=0.
\end{align*} 
Clearly $\mathcal{J} = \mathcal{J}_b = \mathbb{R}\times [-1,1]$. 
\end{Example}
We can make, however, the following easy observation which says that if $u\in \mathcal{J}\setminus\mathcal{J}_b$ then it must be $\lim_{t\to \infty}\|S(t)u\| = \infty$, i.e. it excludes the situation when $\lim_{t\to \infty}\|S(t)u\|$ does not exist, but $\|S(t)u\|$ is unbounded. This means that the unbounded attractor consists of the set $\mathcal{J}_b$, a global-attractor-like object, and the {\color{black} solutions} which are backward in time attracted to $\mathcal{J}_b$ and whose norm, forward in time, has to go to infinity.   
\begin{Lemma}\label{lemma:inf}
	If $u \in \mathcal{J}\setminus \mathcal{J}_b$ then $\lim_{t\to\infty} \|S(t)u\| = \infty$.
	\end{Lemma}
\begin{proof}
We know that $S(t)u$ is unbounded. If for some sequence $t_n\to \infty$ there holds $S(t_n)u\in H_R$ for some $R\geq R_0$ then it has to be $S(t)u\in H_R$ for every $t$ and we have the contradiction. 
\end{proof}
A similar simple argument allows us to split all points of $X$ into two sets, those points whose $\omega$-limits are well defined compact and invariant subsets of $\mathcal{J}_b$ and those points, whose {\color{black} solutions} are unbounded in the future.
  \begin{Lemma}
  	If $u \in X$ then either there exists $t_1>0$ and $R\geq R_1$such that $S(t)u \in H_R$ for $t\geq t_1$ and then $\omega(u) \subset \mathcal{J}_b$ is a compact and invariant set which attracts $u$, or $\|S(t)u\|\to \infty$ as $t\to \infty$.
  \end{Lemma}
  \begin{proof}
  	If there exists $R\geq R_0$ and $t\geq t_1$ such that $S(t)u \in H_R$ for every $t\geq t_1$ then the result holds by Theorem \ref{thm:omega}. Otherwise, by Lemma  \ref{lemma:possibilities} for every $R\geq R_0$ there exists $t_1$ such that $S(t)u\in Q\setminus H_R$ for $t\geq t_1$ and the proof is complete.  
  \end{proof}
 If we reinforce the item 3. of the assumption (H2) to its stronger version which states that the sets $Q\setminus H_R$ are not only positively invariant, but the evolution in them is expanding to infinity, then we can guarantee compactness (and hence closedness) and the simpler characterization of $\mathcal{J}_b$. We make the following assumption. 
 \begin{itemize}
 	\item[(H4)] For every $ R \geq R_0$  there exists $t=t(R)>0$ such that $S(t)(Q\setminus H_{R_0}) \subset Q\setminus H_{R}$.  
 \end{itemize}
\begin{Theorem}\label{compact}
	Assuming (H1)-(H4), the set $\mathcal{J}_b = \alpha(\mathcal{J}\cap H_{R_0})$ is compact. Moreover the set $\mathcal{J}\setminus \mathcal{J}_b$ in the decomposition of the unbounded attractor $\mathcal{J}_b \cup (\mathcal{J}\setminus \mathcal{J}_b)$ is nonempty. 
\end{Theorem}
\begin{proof}
It is sufficient to prove that if $R\geq R_0$ then $\alpha(\mathcal{J}\cap H_R) = \alpha(\mathcal{J}\cap H_{R_0})$. To see that this is true, take $u\in \alpha(\mathcal{J}\cap H_R)$. This means that there exists $u_n\to u$ in $\mathcal{J}$ and $t_n\to \infty$ such that $S(t_n)u_n \in \mathcal{J}\cap H_R$. Now (H4) implies that $S(t_n-t(R)-1)u_n \in \mathcal{J}\cap H_{R_0}$ and the assertion follows.  
	\end{proof}
Assumption (H4) also implies that the set of those points in $X$ whose $\omega$-limit sets are well defined and compact is a closed subset of the space $X$. Hence the set of those points constitutes a complete metric space, and then $\mathcal{J}_b$ is the global attractor of $\{ S(t) \}_{t\geq 0}$ on that space. 
\begin{Lemma}\label{lem:8}
	The set of points $u\in X$ which have the compact and invariant $\omega$-limit set $\omega(u)\subset \mathcal{J}_b$ is a closed set in $X$. 
	\end{Lemma}
\begin{proof}
	Assume for contradiction that $u_n\to u$ with $\lim_{t\to \infty}\mathrm{dist}(S(t)u_n,\mathcal{J}_b) = 0$ and $\lim_{t\to \infty}\|S(t)u\| = \infty$. For every large $R$ there exists $t_0>0$ such that for every $t\geq t_0$ we can find $n_0$ such that for every $n\geq n_0$ we have $\|S(t)u_n\|\geq R$. But we can find sufficiently large $R$ and $t_0$ such that this means that for some $t$ and for sufficiently large $n$ there holds $S(t)u_n\in Q\setminus H_{R_0}$ which means that $\|S(t)u_n\|\to \infty$, a contradiction. 
\end{proof}

\subsection{Multivalued inertial manifold}\label{sec:inertial}
The concept of multivalued inertial manifolds has been introduced in \cite{debusshetemam} and further developed in \cite{jcrobinson}.  While, classically, inertial manifolds require the so called spectral gap condition to exist, their multivalued counterpart, as shown in \cite{debusshetemam,jcrobinson}, does not require such condition, at the price of 'multivaluedness'. Motivated by Lemma \ref{lemma:211}, we show that the concept of multivalued inertial mainfold is compatible with unbounded attractors. 
Namely, Lemma \ref{lemma:211} states that for every $p\in E^+$ there exists at least one $q\in E^-$ such that $p+q\in \mathcal{J}$. This makes it possible to define the multivalued mapping $\Phi:E^+ \to \mathcal{P}(E^-)$ by
$$
\Phi (p) = \{q\in E^-\,:\ p+q\in \mathcal{J}  \}.
$$
Then Lemma \ref{lemma:211} states that for every $p\in E^+$ the set $\Phi(p)$ is nonempty. 
We continue by the analysis of $\Phi$. 
\begin{Lemma}\label{multi1}
	Under assumptions (H1)--(H3) the multifuction $\Phi$ has nonempty, compact values, closed graph, and is upper-semicontinuous. 
	\end{Lemma}
\begin{proof}
	The graph of $\Phi$ is $\mathcal{J}$, a closed set. Moreover $\Phi(p)$ is closed for every $p\in E^+$, an since by Lemma  \ref{lem:compact} it must be relatively compact, it is also compact. Since, by the same Lemma, for every bounded set $B \in \mathcal{B}(E^+)$ the set $\overline{\Phi(B)}$ is compact, using Proposition 4.1.16 from \cite{DMP1} we deduce that $\Phi$ is upper-semicontinuous, and hence also Hausdorff upper-semicontinuous, i.e. if $p_n\to p$ in $E^+$, then $\lim_{n\to \infty}\textrm{dist}(\Phi(p_n),\Phi(p)) \to 0$. 
\end{proof}
If, in addition to (H1)--(H3), we assume (A1), we arrive at the next result which follows directly from Remark \ref{rem:1}.
\begin{Remark}
Assume (H1)--(H3) and (A1). Then $$\lim_{\|p\|\to \infty} \mathrm{diam}\,\Phi(p) = 0.$$ 
\end{Remark}

\subsection{Unbounded $\omega$-limit sets and their properties}\label{sec:omega}
Following the classical definition we can define the $\omega$-limit sets for slowly non-dissipative dynamical systems 
\begin{Definition}
	An $\omega$-limit set of the set $B\subset \mathcal{B}(X)$ is the set
	$$
	\omega(B) = \{ x\in X\;:\ \textrm{there exists}\ \ t_n\to \infty\ \ \textrm{and}\ \ u_n\in B\ \ \textrm{such that}\ \ S(t_n)u_n\to x   \}
	$$
\end{Definition}

Using Lemma \ref{lemma:possibilities} and Theorem \ref{thm:omega} we can formulate the following result on $\omega$-limit sets.
\begin{Corollary}\label{corom}
	Assume (H1)--(H3). Let $B\in \mathcal{B}(X)$. 
If there exists $R\geq R_0$ and $t_1>0$ such that for every $t\geq t_1$ we have $S(t)B \subset H_{R}$ then $\omega(B) \subset \mathcal{J}_B$ is a nonempty, compact and invariant set which attracts $B$. If for every $R\geq R_0$ there exists $t_1>0$ such that for every $t\geq t_1$ we have $S(t)B \in Q\setminus H_{R}$ holds then $\omega(B)$ is empty.  
\end{Corollary}
	We continue by the analysis of the situation when both $S(t)B \cap H_R$ and $S(t)B \cap (Q\setminus H_R)$ are nonempty. The next result is a complement to Theorem \ref{thm:omega}. 
	\begin{Lemma}
		Assume (H1)--(H3) and let $B\in \mathcal{B}(X)$. If there exists $R\geq R_0$ and $t_1 > 0$ such that for every $t\geq t_1$ both sets $S(t)B \cap H_R$ and $S(t)B \cap (Q\setminus H_R)$ are nonempty (i.e. case (2) of Lemma \ref{lemma:possibilities} holds), then $\omega(B) \subset \mathcal{J}$ is a nonempty, closed, and invariant set which attracts $B$ in the bounded sets in $E^+$, i.e. 
		$$
		\lim_{t\to \infty}\mathrm{dist}(S(t)B\cap \{ \|Px\|\leq S\},\omega(B)) = 0\ \ \textrm{for every}\ \ S\geq 0.
		$$
	\end{Lemma}
	\begin{proof}Take $t_n\to \infty$. There exists $u_n \in B$ such that $S(t_n)u_n \in H_R$, whence, arguing as in the proof of Theorem \ref{attraction}, $S(t_n)u_n$ is relatively compact and hence $\omega(B)$ is nonempty. Its closedness follows by a diagonal argument, while the fact that $\omega(B) \subset \mathcal{J}$ follows from the definition of $\mathcal{J}$ and the fact that $Q$ is absorbing and positively invariant. The proof of invariance of $\omega(B)$ is straightforward and follows from the fact that we can enclose the convergent sequence $S(t_n)u_n$ in some $H_{\overline{R}}$. Finally the attraction in bounded sets is obtained in the same way as in Theorem  \ref{attraction}.
	\end{proof}

\subsection{The dynamics at infinity}\label{sec:infty} In \cite{bengal, PimentelRocha, carvalhopimentel} the authors define the dynamics at infinity by means of the Poincar\'{e} projection. Observe, however, that for the {\color{black} solutions} $u(t)$ that diverge to infinity, the projection $(I-P)u(t)$ remains bounded, and $\|Pu(t)\|$ tends to infinity. Hence, when we rescale those {\color{black} solutions} by $\|P u(t)\|$, i.e. we consider
$$
\frac{u(t)}{\|Pu(t)\|} = \frac{Pu(t)}{\|Pu(t)\|} + \frac{(I-P)u(t)}{\|Pu(t)\|},
$$
the infinite dimensional component, which belongs to $E^-$ tends to zero, while the term $\frac{Pu(t)}{\|Pu(t)\|}$ evolves on the unit sphere. We recover the dynamics at infinity by the analysis of the asymptotic behavior of this, rescaled, $Pu(t)$. Denoting the unit sphere in $E^+$ by $\mathbb{S}_{E^+} = \{ x\in E^+\; :\ \|x\| = 1\}$, we have the next result. Note that this result actually describes the dynamics at infinity if we assume  (H4) in addition to (H1)-(H3) but it is valid in the more general case.  
\begin{Lemma}\label{lem_infty} Assume (H1)--(H3). Then for every $R\geq R_0$
	$$
	\bigcap_{t\geq 0} 
	\overline{
		\bigcup_{x\in S(t)(Q\setminus H_R)}%
		\left\{ \frac{Px}{\|Px\|}
		\right\}
	} 
	= \mathbb{S}_{E^+}.
	$$
	Moreover  $\mathbb{S}_{E^+}$ is the smallest closed set such that for every $B\in \mathcal{B}(X)$ satisfying  1. of Lemma \ref{lemma:possibilities} for sufficiently large $R$, there holds
	$$\lim_{t\to \infty}\mathrm{dist}\left(\left\{\frac{Px}{\|Px\|}\,:\ x\in S(t)B\right\},\mathbb{S}_{E^+}\right) = 0. 
	$$ 
\end{Lemma}
\begin{proof}
	We will prove that 
	$$
	\bigcup_{x\in S(t)(Q\setminus H_R)}%
	\left\{ \frac{Px}{\|Px\|}
	\right\} = \mathbb{S}_{E^+}.
	$$
	By Lemma \ref{lemma:211}, the fact that $\mathcal{J}\subset Q$ and the invariance of $\mathcal{J}$  it follows that $
	\{ Px: x\in S(t)Q\} = E^+.
	$ 
	Now 
	$$
	E^+ = P S(t)Q = PS(t) H_R \cup PS(t)(Q\setminus H_R). 
	$$
	Since by (H3) the set $PS(t) H_R$ is bounded, there exists $R>0$ such that $\{ x\in E^+\,:\ \|x\|=R \} \subset PS(t)(Q\setminus H_R)$ and the assertion follows.
	
	To get the second assertion pick $R\geq R_0$ and choose $B = \{ x\in E^+ \, :\ \|x\| = \overline{R} \}$, where $\overline{R}$ is sufficiently large to get $S(t) B \subset Q\setminus H_R$ for every $t\geq 0$ (and hence $B$ satisfies 1. of Lemma \ref{lemma:possibilities}). Such choice of $\overline{R}$ is possible by (H2).  
	We will show that 
	$$
	\left\{\frac{Px}{\|Px\|}\,:\ x\in S(t)B \right\} = \mathbb{S}_{E^+}. 
	$$
	To this end choose $t > 0$ and consider the mapping $H:[0,1]\times \mathbb{S}_{E^+}\to \mathbb{S}_{E^+}$ defined as 
	$$
	H(\theta,x) = \frac{P S(\theta t) (\overline{R}x) }{\|P S(\theta t) (\overline{R}x)\|}.
	$$ 
	Clearly $H(0,\cdot)$ is an identity, and $H$ is continuous. Hence, by the Hopf lemma, the degree of the mapping $H(1,x)$ is equal to the degree if the identity of the sphere, that is, one. On the other hand assume for contradiction that the image of $H(1,\cdot)$ is not equal to the whole $\mathbb{S}_{E^+}$. Then, as the image  of $H(1,\cdot)$ is compact, it is included in  $\mathbb{S}_{E^+} \setminus U$ for some nonempty open set $U$. This set is contractible to one point, which in turn means that $H(1,\cdot)$ must be homotopic to a constant map, which has the degree zero, a contradiction.   
\end{proof}

We can also define $\omega$-limits in $\mathbb{S}_{E^+}$ of nonempty bounded sets $B \subset Q\setminus H_{R_0}$. 

$$
\omega_\infty(B) = \left\{ x\in E^+\,:\ \textrm{there exists}\ \ u_n\in B \ \ \textrm{and}\ \ t_n\to \infty \ \ \textrm{such that}\ \ \frac{PS(t_n)u_n}{\|PS(t_n)u_n\|}\to x\right\}
$$
\begin{Lemma}\label{lem:infty}
	Let $B \subset Q\setminus H_{R_0}$ be bounded. Then $\omega_\infty(B)$ is nonempty, compact, and attracts $B$ in the sense 
		$$\lim_{t\to \infty}\mathrm{dist}\left(\left\{\frac{Px}{\|Px\|}\,:\ x\in S(t)B\right\},\omega_\infty(B) \right) = 0. 
	$$ 
	\end{Lemma}
\begin{proof}
Non-emptiness follows from the compactness of $\mathbb{S}_{E^+}$ and we can obtain closedness (and hence compactness) of $\omega_\infty(B)$ from a diagonal argument, the argument closely follows the classical proof of closedness of $\omega$-limit sets. Also the proof of the attraction is standard: for contradiction assume that for some $B$ there exists the sequence $\{x_n \} \subset B$ and $t_n\to \infty$ such that 
		$$\mathrm{dist}\left(\frac{PS(t_n)x_n}{\|PS(t_n)x_n\|},\omega_\infty(B) \right) \geq \varepsilon. 
$$ 
But for a subsequence, still denoted by $n$ 
$$
\frac{PS(t_n)x_n}{\|PS(t_n)x_n\|} \to y,
$$
where, by definition $y\in \omega_\infty(B),$ a contradiction. 
\end{proof}
For particular slowly non-dissipative semigroups it is possible to construct the dynamical system on $E^+$ such that $\omega_\infty(B)$ are invariant and whole $\mathbb{S}_{E^+}$ is the global attractor. This dynamical system reflects the dynamics at infinity of the original problem. We will present such construction in the part of the paper devoted to the example of a slowly non-dissipative problem governed by a PDE. 

%------------------------------------
%\color{red}
%-------------------------------------------
%
%
% 
% 	------------------------------
% \subsection{Connectedness of $\mathcal{J}$ and $\mathcal{J}_b$.}
% Next, we establish the connectedness of $\mathcal{J}$. 
% \begin{Lemma}
% 	For every $R$ the set $\mathcal{J}\cap \{\|Pu\|\leq R \}$ is connected. Moreover $\mathcal{J}$ is connected.
% \end{Lemma}
% \begin{proof}
% 	Since  $\mathcal{J} = \bigcup_{n=1}^\infty (\mathcal{J}\cap \{\|Pu\|\leq n\})$ it is sufficient to show that   $\mathcal{J}\cap \{\|Pu\|\leq n\}$ must be connected, i.e that
% 	$$
% 	K_n = \bigcap_{t\geq 0} \overline{S(t) Q} \cap \{\|Pu\|\leq n\} = \bigcap_{t\geq 0} K_n^t 
% 	$$
% 	is connected. To this end assume that $K_n$ is not connected. Then one can define the continuous function $f:K_n \to \{ 0,1\}$. This $f$ can be extended to a continuous function $f_t:K^t_n \to [0,1]$.
% \end{proof}
%----------------------------------------
%\color{black}

\section{Nonautonomous unbounded attractors and their properties.}\label{nonauto}
The nonautonomous version of the definition of the pullback attractor was proposed in \cite[Definition 3.1]{carvalhopimentel} together with a result on its existence for the problem governed by the following PDE
$$
u_t = u_{xx} + b(t)u + g(u),
$$
where the space domain is the one-dimensional (an interval) with homogeneous Dirichlet boundary conditions. The function $b(t)$ is allowed to oscillate between two gaps of the spectrum of the leading elliptic operator. In this section we propose a systematic abstract approach to non-autonomous unbounded pullback attractors. 

\subsection{Definition of pullback unbounded attractor.}
We first note that, if $X$ is a normed space, than any family of nonempty sets $\{ \mathcal{A}(t) \}_{t\in \mathbb{R}}$ in $X$ will be called a non-autonomous set. We will also use the notation $\mathcal{A}(\cdot)$ for such sets. 
Exactly as in the autonomous case we will assume that the phase space $X$ is Banach, and can be represented as $X=E^+\oplus E^-$, with $E^+$ being finite dimensional. Using the same notation as in Section \ref{sec:autonomous} we can represent any $x\in X$ as $x=p+q$.
We remind some definitions concerning the dissipative non-autonomous dynamical systems, beginning with  the definition of a process and its pullback attractor, see \cite{Carvalho-Langa-Robinson-13} for the systematic treatise on the theory. In the same way as in the autonomous case, in our requirement for the attraction we impose that it takes place only in the bounded sets in $E^+$.

\begin{Definition}
	Let $X$ be a metric space. A family of mappings $\{ S(t,s)\}_{t\geq s}$, where $S(t,s):X\to X$ is a continuous process if
	\begin{itemize}
		\item $S(t,t) = I$ (identity on $X$) for every $t\in \mathbb{R}$,
		\item for every $t\geq \tau \geq s$, and for every $x\in X$ we have $S(t,s)x = S(t,\tau)S(\tau,s)x$,
		\item the mapping $(t,s,x)\mapsto S(t,s)x$ is continuous for every $t\geq s$, $x\in X$.
	\end{itemize}
	
\end{Definition}

\begin{Definition}\label{def:set}
	A family of sets $\{\mathcal{A}(t)\}_{t\in \mathbb{R}}$ is the unbounded pullback attractor for a process $S(\cdot,\cdot)$ if 
	\begin{enumerate}
		\item $\mathcal{A}(\cdot)$ are nonempty and closed,
		\item $\mathcal{A}(\cdot)$ is invariant with respect to $S(\cdot,\cdot)$, that is 
		$$S(t,\tau)\mathcal{A}(\tau)=\mathcal{A}(t) \ \text{  for every  } \  t,\tau\in \mathbb{R} \ \text{  with  } \ t\geq \tau.$$
		\item for every $t\in \mathbb{R}$ the set $\mathcal{A}(t)$ is pullback attracting in the bounded sets in $E^+$ at time $t$, that is
		$$
		\lim_{s\rightarrow-\infty} \mathrm{dist} (S(t,s)B\cap \{ \|Px\|\leq R \},\mathcal{A}(t))=0
		$$
		 for every $B\in \mathcal{B}(X)$ and every $R>0$ for which there exists $t_1\leq t$ such that $S(t,s)B\cap \{ \|Px\|\leq R \}$ is nonempty for every $s\leq t_1$.  
		\item $\mathcal{A}(\cdot)$ is the minimal family of closed sets with property (3).
	\end{enumerate}
\end{Definition}

We will also need the notions of pullback absorbtion and pullback positive invariance. 
\begin{Definition}
	A set $B\subset X$ pullback absorbs bounded sets at time $t\in \mathbb{R}$ if, for each bounded subset $D$ of $X$, there exists $T=T(t,D)\leq t$ such that
	$$S(t,s)D\subset B \text{ for all  } s\leq T.$$
\end{Definition}

\begin{Definition}
	A non-autonomous set $\{Q(t)\}_{t\in \mathbb{R}}$ is positively invariant if
	$$
	S(t,s)Q(s)\subset Q(t) \ \ \ \text{ for every $s\leq t$}.
	$$
\end{Definition}

Finally, we propose the concept of generalized pullback asymptotic compactness, as a non-autonomous variant of  the notion given in Definition \ref{def:gen_ass}, 
\begin{Definition}
	A process $S(\cdot,\cdot)$ is generalized pullback asymptotically compact if for every $B\in \mathcal{B}(X)$ and for every $t,s\in \mathbb{R}$, $s\leq t,$ there exists a $K(t,s,B)\subset X$ and $\varepsilon(t,s,B)\rightarrow 0$ when $s\rightarrow -\infty$ such that
	$$S(t,s)B\subset \mathcal{O}_{\varepsilon}(K).$$
\end{Definition}

\subsection{Existence of unbounded pullback attractors}

The following assumptions which are the non-autonomous versions of (H1)--(H3) from Section \ref{sec:autonomous} will guarantee the existence of the unbounded pullback attractor. 

\begin{enumerate}
	\item[(H1)$_{\mathrm{NA}}$] There exist $D_1,D_2>0$ and the closed sets $\{Q(\cdot)\}_{t\in \mathbb{R}}$ such that for every $t\in \mathbb{R}$,
	$$
	\{ \|(I-P)x\|\leq D_1 \} \subset  Q(t) \subset \{ \|(I-P)x\|\leq D_2 \}
	$$
	such that $Q(t)$ pullback absorbs bounded sets at time $t$ and family $Q(\cdot)$ is positively invariant.
	\item[(H2)$_{\mathrm{NA}}$] There exists the constants $R_0$ and $R_1$ with $0< R_0\leq R_1$ and an ascending family of closed and bounded sets $\{\{H_R(\cdot)\}_{t\in \mathbb{R}}\}_{R\geq R_0}$ with $H_R(t)\subset Q(t)$ for every $t\in \mathbb{R}$ such that
	\begin{enumerate}
		\item[1.] for every $R\geq R_1$ we can find $S(R)\geq R_0$ such that $\{\|Px\|\leq S(R)\}\cap Q(t)\subset H_R(t)$ for every $t\in \mathbb{R}$, and moreover $\lim_{R\rightarrow \infty} S(R)=\infty$,
		\item[2.] for every $R\geq R_1$ we have $H_R(t)\subset \{\|Px\|\leq R\}$ for every $t\in \mathbb{R},$
		\item[3.] for every $R\geq R_0$ and $t\in \mathbb{R}$, $S(t,s)Q(s)\setminus H_R(s)\subset Q(t)\setminus H_R(t)$, for every $s\leq t$.
	\end{enumerate}
	\item[(H3)$_{\mathrm{NA}}$] The process $S(\cdot,\cdot)$ is generalized pullback asymptotically compact. 
%	\item[(H4)$_{\mathrm{NA}}$] For every $t\geq s$ and bounded set $B$ the image $S(t,s)B$ is also bounded.
\end{enumerate}

The  candidate for the unbounded pullback attractor is the non-autonomous set  $\mathcal{J}(\cdot)$ given by
$$\mathcal{J}(t)=\bigcap_{s\leq t}\overline{S(t,s)Q(s)}.$$
The rest of this section is devoted to the proof that assuming (H1)$_{\mathrm{NA}}$--(H3)$_{\mathrm{NA}}$ this set satisfies the requirements of Definition \ref{def:set}. We also establish its relation with the time dependent maximal kernel and its section. To this end we first define the time dependent maximal kernel as the family of those complete non-autonomous {\color{black} solutions} $u:\mathbb{R}\to X$ which are bounded in the past. 

\begin{Definition}
	The set $\mathcal{K}$ is called a maximal kernel if 
	$$
	\mathcal{K} = \{ u(\cdot)\,:\ \textrm{there exists}\ \ T\in \mathbb{R}\ \, \textrm{such that}\ \, \sup_{s\in (-\infty,T]}\|u(s)\|\leq C_u\ \, \textrm{and}\, \ S(t,s)u(s) = u(t)\, \ \textrm{for every}\, \ t\geq s\}
	$$
\end{Definition}
It is clear that if the assertion of the above definition holds for some $T\in\mathbb{R}$ then it also holds for every $\overline{T} < T$. The continuity of the process $S$ with respect to time also implies that this assertion also holds for every $\overline{T} > T$. We continue with the definition of the maximal kernel section.
\begin{Definition}
	The non-autonomous set $\{\mathcal{K}(t) \}_{t\in \mathbb{R}}$ is called a maximal kernel section  if 
	$$
	\mathcal{K}(t) = \{ u(t) \,:\ u\in \mathcal{K}   \}\ \ \textrm{for every}\ \ t\in \mathbb{R}.
	$$
\end{Definition}
Clearly, for a continuous process this is an invariant set, namely
\begin{Observation}
	If $S(\cdot,\cdot)$ is a continuous process, then  $\mathcal{K}(t) = S(t,s)\mathcal{K}(s)$ for every $s\in \mathbb{R}$ and $t\geq s$. 
\end{Observation}
%\begin{Observation}
%For every $t\geq s, R\geq R_0$ we have $S(t,s)H_R \cap H_R =S(t,s)Q \cap H_R.$
%\end{Observation}

At the moment we do not yet know if either of the sets  $\mathcal{J}(t)$ or $\mathcal{K}(t)$ are nonempty. We establish, however, then they must coincide. This is a non-autonomous version of Theorem \ref{coincide}. 
\begin{Theorem}\label{thm:eqker}
	Assume $(H1)_{\mathrm{NA}}-(H3)_{\mathrm{NA}}$. Then  for every $t\in \mathbb{R}$ $\mathcal{J}(t)=\mathcal{K}(t)$.
\end{Theorem}
\begin{proof}
	The proof is analogous to the proof of Theorem \ref{coincide} - the autonomous counterpart of the result. Take $t\in \mathbb{R}$. First we prove the inclusion $\mathcal{K}(t)\subset \mathcal{J}(t)$. We suppose that $u(t)\in \mathcal{K}(t)$, then we can find  $C\in \mathbb{R}^+$ such that $\|u(s)\|\leq C$ for every $s\leq t$. Let $B=\{x\in X\, : \  \|x\|\leq C\}$. By $(H1)_{\mathrm{NA}}$ we know that for every $\tau \in \mathbb{R}$ there exists $s\leq \tau$ such that $S(\tau,s)B\subset Q(\tau)$. Then, taking any $\tau \leq t$ we obtain 
	$$u(t)=S(t,\tau)u(\tau) = S(t,\tau)S(\tau,s)u(s)\in S(t,\tau)S(\tau,s)B\subset S(t,\tau)Q(\tau).$$
	So, for every $\tau \leq t$, $u(t)\in S(t,\tau)Q(\tau)$, and hence $u(t)\in \mathcal{J}(t)$.
	
	To prove the equality, take $u_0\in \mathcal{J}(t)$. Then there exists a sequence $\{y_n\}$ such that $y_n\in Q(t-n)$ and  
	$S(t,t-n)y_n \to u_0$ as $n\to \infty$. 
	The convergent sequence $\{S(t,t-n)y_n\}$ is contained in $Q(t)$, whence there exists $R\geq R_0$ such that $\{S(t,t-n)y_n\}\subset H_R(t)$. 
	%By item 3.  of $(H2)_{\mathrm{NA}}$, for every $s\in (t-n,t)$ also $S(t,s)y_n\in H_R$. 
	Notice that for $s\in [t-n,t]$ also $S(s,t-n)y_n\in H_R(s)$, because if  $S(s,t-n)y_n\in Q(s)\setminus H_R(s)$, then it would be $S(t,t-n)y_n=S(t,s)S(s,t-n)y_n\in Q(t)\setminus H_R)(t)$. In particular it must be $y_n\in H_R(t-n)$.  So $S(s,t-n)y_n\in H_R(s)\cap S(s,t-n)H_R(t-n)=L_n(s)$, and, for each $s\leq t$ this sequence of sets is nested. By Lemma \ref{lemma:212} and $(H3)_{\mathrm{NA}}$, for each $s$ one can take a subsequence of indexes, still denoted by $n$ such that $S(s,t-n)y_n$ is convergent and the limit is always in $H_R(s)$. In particular, for $s=t-1$ we have $S(t-1,t-n)y_n \to u_1$ and $S(t,t-n)y_n = S(t,t-1)S(t-1,t-n)y_n \to S(t,t - 1)u_1$ and hence $S(t,t-1)u_1 = u$ with $u_1\in H_R(t-n)$. Passing to the subsequence in each step of the iterative procedure we are able to construct the sequence $u_n \in H_R(t-n)$ such that $S(t-n+1,t-n)u_n = u_{n-1}$, which allows us to define the {\color{black} solution} $u(\cdot)$ by taking $u(r) = S(r,t-n)u_n$ for $t\in [t-n,t-n+1]$. As $u(r) \in H_R(r)$ for every $r\leq t$ it follows that there exists a constant $C$ such that $\sup_{s\in (-\infty, t]}\|u(s)\|\leq C$ and the proof is complete.
\end{proof}

Now, we state the non-autonomous version of Lemma \ref{lemma:211}. This result also implies that $P\mathcal{J}(t) = E^+$ for every $t\in \mathbb{R}$ and hence sets $\mathcal{J}(t)$ are nonempty. 
\begin{Lemma}
	Assume $(H1)_{\mathrm{NA}}-(H3)_{\mathrm{NA}}$, for every $t\in \mathbb{R}$ it holds that for every $p\in E^+$, there exists a $q\in E^-$, such that $p+q\in \mathcal{J}(t)$.
\end{Lemma}
\begin{proof}
	Analogously as in autonomous case we pick $p\in E^+$ and define $B=\{x\in E^+\, :\, \|x\|<R+1\}$ for appropriately large $R$ such that $p\in H_R(\tau) \subset B$ for every $\tau\in \mathbb{R}$, which is possible by items 1 and 2 of $(H2)_{NA}$. Then $\textrm{deg}(I,B,p)=1$. Now we pick $t\in \mathbb{R}$ and  we choose $s>0$. We define continuous mapping
	\begin{align*}
		[0,1]\times \overline{B}\ni (\theta,p)&\longmapsto P(S(t,t-\theta s) p) \in E^+.
	\end{align*}
	By item 2. of $(H2)_{NA}$, since 
	$\delta B=\{x\in E^+\,:\  \|x\|=R+1\} \subset Q(\tau)\setminus H_R(\tau),$ for every $\tau \in \mathbb{R}$,
	then
	$S(t,t-\theta s)\delta B\subset Q(t)\setminus H_R(t)$
	so $p\notin P(S(t,t-\theta s)\delta B)$ for every $\theta \in [0,1]$. By the homotopy invariance of the degree it follows that $\textrm{deg}(PS(t,t-s),B,p)=1$ whence for every sequence $s_n\to -\infty$ we can find $p_n \in E^+$ and $q_n\in E^-$ such $y_n=p+q_n=S(t,t-s_n)p_n\in S(t,t-s_n)Q(t-s_n).$ As in the autonomous case we can find $R'$ such that $y_n \in L_n=S(t,t-s_n)H_{R'}\cap H_{R'} := L_n$, where $L_n$ is a decreasing sequence of sets. We use Lemma \ref{lemma:212} and $(H3)_{\textrm{NA}}$, to deduce that  $y_n\rightarrow y$ for a subsequence with $Py=p$. Since $y_n\in S(t,t-s_n)Q(t-s_n)$, it follows that $y\in \mathcal{J}(t).$ 
\end{proof}

We provide the result that the nonautonomous set $\mathcal{J}(\cdot)$ is attracting in every ball.

\begin{Theorem}\label{theo:attract}
	Let $(H1)_{\mathrm{NA}}-(H3)_{\mathrm{NA}}$. Let $t\in \mathbb{R}$ and $R\geq R_1.$ We take $B\in B(X)$ such that there exists a $t_1\leq t$ that for every $s\leq t_1$ there holds $S(t,s)B\cap \{\|Px\|\leq R\} \not= \emptyset$. Then we have
	$$\lim_{s\rightarrow-\infty} \mathrm{dist}(S(t,s)B\cap \{\|Px\|\leq R\},\mathcal{J}(t))=0$$
\end{Theorem}
\begin{proof}
 To prove the attraction we proceed in a standard way, i.e. for contradiction we choose the sequences $\{x_n\}\subset B$ and $s_n\rightarrow -\infty$ such that 
	$S(t,s_n)x_n\in S(t,s_n)B\cap \{\|Px\|\leq R\}$
	and
	\\
	$\mathrm{dist}(S(t,s_n)x_n,\mathcal{J}(t))>\varepsilon$
	for every $n$ and for some $\varepsilon>0.$ 
	
	By $(H1)_{\mathrm{NA}}$ we know that for every $n\in \mathbb{N}$ there exists $k(n) \geq n$ such that $S(s_n,s_k)x_k \in Q(s_n)$, then $y_k=S(t,s_k)x_k\in S(t,s_n)Q(s_n)\subset Q(t)$. Since $\|P(S(t,s_k)x_k)\|\leq R$ for every  $k$, by $(H2)_{\mathrm{NA}}$ there exists $\overline{R}\geq R_0$ and such that the sequence $\{y_{k(n)}\}_{n\in \mathbb{N}}\subset H_{\overline{R}}(t)$. We define the sets $L_n=S(t,s_n)Q(s_n)  \cap H_{\overline{R}}(t)$, these sets are nested, and $y_{k(n)}\in L_n$ for every $n$. By Lema \ref{lemma:212} and $(H3)_{\mathrm{NA}}$, we know that for a subsequence, we must have $y_{k(n)}\rightarrow y$, and by definition of $\mathcal{J}(t)$, it must be $y\in \mathcal{J}(t)$, which gives a contradiction. 	
\end{proof}
We end the chapter with the result of minimality of $\{ \mathcal{J}(t) \}_{t\in \mathbb{R}}$, which is the last property of the unbounded pullback attractor which we need to verify. 
\begin{Theorem}\label{thm:pullback}
	Assume $(H1)_{\mathrm{NA}}-(H3)_{\mathrm{NA}}$, then the non-autonomous set  $\{ \mathcal{J}(t) \}_{t\in \mathbb{R}}$ is the unbounded pullback attractor of the process $S(\cdot,\cdot)$. 
\end{Theorem}
\begin{proof}
        It only remains to verify the minimality in Definition \ref{def:set}. To this end, suppose that there exists a non-autonomous closed set $\{ \mathcal{M}(t) \}_{t\in \mathbb{R}}$ which satisfies property (3) of Definition \ref{def:set}, and for a certain $t\in \mathbb{R}$, there exist an $u\in \mathcal{J}(t)$ such that $u\notin\mathcal{M}(t)$. By Theorem \ref{thm:eqker}, we know that the orbit $\bigcup_{s\leq t} \{ u(s) \}$ is bounded. Since $u = u(t) = S(t,r) u(r) \subset S(t,r) (\bigcup_{s\leq t} \{ u(s) \})$ we know that there exists some  $R\geq R_0$ such that $u(t) \in S(t,r) (\bigcup_{s\leq t} \{ u(s) \}) \cap \{ \|Px\| \leq R \}$ and hence the latter sets are nonempty for every $r\leq t$. Then 
        $$\lim_{r\rightarrow-\infty} \text{dist}(S(t,r) u(r),\mathcal{M}(t)) \leq \lim_{r\rightarrow-\infty} \text{dist}\left(S(t,r)\left(\bigcup_{s\leq t} \{u(s)\}\right) \cap \{ \|Px\| \leq R \},\mathcal{M}(t)\right)=0,$$
        whence it must be 
        $\text{dist}(u(t),\mathcal{M}(t))=0$, so $u=u(t)\in \mathcal{M}(t)$, since $\mathcal{M}(t)$ is closed, and the proof of minimality is complete. 
 \end{proof}

We conclude the section withe the result on $\sigma$-compactness of $\mathcal{J}(t)$.

\begin{Lemma}\label{lemma:531}
	Assume $(H1)_{\mathrm{NA}}-(H3)_{\mathrm{NA}}$. For every $t\in \mathbb{R}$ and for every closed and bounded set $B$ the set $B \cap \mathcal{J}(t)$ is compact. 
\end{Lemma}
\begin{proof}
	Since $B \cap \mathcal{J}(t)\subset H_R(t) \cap \mathcal{J}(t)$ for some $R\geq R_0$ it is enough to prove that $H_R(t) \cap \mathcal{J}(t)$ is relatively compact. The proof follows the lines of the proof of Theroem \ref{lem:compact} - the corresponding result for the autonomous case.   We take a sequence $\{u_n\}\subset H_R(t)\cap \mathcal{J}(t). $ Then we can take a sequence  $t_n\rightarrow-\infty$, and  $x_n\in \mathcal{J}(t_n)$ with $u_n=S(t,t_n)x_n$. Then $x_n\in H_R(t_n)$, so $u_n\in S(t,t_n)H_R(t_n)\cap H_R(t)$. By Lemma $\ref{lemma:212}$ and $(H3)_{\textrm{NA}}$ we obtain the desired result.
\end{proof}

\subsection{Structure of the unbounded pullback attractor} 

We begin with the definition of the kernel section in the non-autonomous case.
$$
\mathcal{I}_b(t) = \{ u(t)\,: \ \textrm{such that}\ \ \sup_{s\in \mathbb{R}}\|u(s)\|\leq C_u\ \ \textrm{and}\ \ S(r,s)u(s) = u(r)\ \ \textrm{for every}\ \ r\geq s\}.
$$
Clearly $\mathcal{I}_b(t) \subset \mathcal{J}(t) = \mathcal{K}(t)$, i.e. the kernel section is the subset of the maximal kernel section (which contains those points, through which there pass {\color{black} solutions} bounded in the past but possibly unbounded in the future). Moreover $S(s,t)\mathcal{I}_b(t) = \mathcal{I}_b(s)$ for every $s\geq t$, i.e. kernel sections are invariant. In the autonomous case the corresponding set,  $\mathcal{I}_b = \mathcal{J}_b$ was constructed as the $\alpha$-limit of bounded sets in the unbounded attractor. Since, as it seems to us, such construction is no longer possible in the non-autonomous case, we provide an alternative way to obtain $\mathcal{I}_b(t)$ from  $\mathcal{J}(t)$. Since the non-autonomous set $\{\mathcal{J}(t)\}_{t\in \mathbb{R}}$ is invariant we can define the inverse mapping $(S(t,s))^{-1}:\mathcal{J}(t)\to 2^{\mathcal{J}(s)}$ for $t>s$. We will use the notation $(S(t,s))^{-1} = S(s,t)$ for $t>s$.  Note that, as we do not assume the backward uniqueness, if we consider $S(s,t)$ on the whole space $X$, its image can go beyond  $ \mathcal{J}(s)$, and, moreover $S(s,t)$ (even as considered as mapping from $\mathcal{J}(t)$ to  $\mathcal{J}(s)$) can be multivalued, From the invariance we have, however, the guarantee that for every $u\in \mathcal{J}(t)$ the set $S(s,t)u$ is nonempty. 
We define the following non-autonomous set.
$$\mathcal{J}_b(t)=\bigcup_{R\geq R_0} \bigcap_{\tau\geq t}S(t,\tau)(H_R(\tau) \cap \mathcal{J}(\tau))$$
In the next results we will prove that $\mathcal{J}_b = \mathcal{I}_b$, and it is nonempty.

\begin{Lemma}
	Assume $(H1)_{\mathrm{NA}}-(H3)_{\mathrm{NA}}$. For every $t\in \mathbb{R}$, the set $\mathcal{J}_b(t)$ is nonempty.
\end{Lemma}
\begin{proof}
	We will prove the result using the Cantor intersection theorem. First, we see that the inverse image $S(\tau,t)^{-1}$ of the closed set $H_R(\tau)\cap \mathcal{J}(\tau)$ is closed, since the process $S(\cdot,\cdot)$ is continuous. We also observe that for every $\tau \geq t$,  we have $S(t,\tau)(H_R(\tau)\cap \mathcal{J}(\tau))\subset H_R(t)\cap \mathcal{J}(t)$, a compact set by Lemma \ref{lemma:531}, then so the sets $S(t,\tau)(H_R(\tau)\cap \mathcal{J}(\tau))$ are compact. They are also nested: indeed, if $\tau_2>\tau_1\geq t$ and $x\in S(t,\tau_2)(H_R(\tau_2)\cap \mathcal{J}(\tau_2))$, then $S(\tau_1,t)x\in H_R(\tau_1)$, and, as we are taking inverse image in the set $\mathcal{J}(\cdot)$,  also  $S(\tau_1,t)x\in \mathcal{J}(\tau_1).$ 
\end{proof}

\begin{Lemma}
	Assuming $(H1)_{\mathrm{NA}}-(H3)_{\mathrm{NA}}$, for every $t\in \mathbb{R}$, $\mathcal{J}_b(t)=\mathcal{I}_b(t).$
\end{Lemma}
\begin{proof}
	We take first $u\in \bigcup_{R\geq R_0} \bigcap_{\tau\geq t}S(t,\tau)(H_R(\tau) \cap \mathcal{J}(\tau))$, so there exists $R\geq R_0$ such that for every $\tau \geq t$, $S(\tau,t)u\in H_R(\tau) \cap \mathcal{J}(\tau)$. Since $u\in \mathcal{J}(t)$, then there exists a backwards bounded {\color{black} solution} through $u$ at time $t$. Thus, there exists a global bounded {\color{black} solution} through $u$ at time $t$. 
	
	To prove the opposite inclusion, since $u(t)$ is bounded for every $t\in \mathbb{R}$, there exists  $R\geq R_0$, such that $u(t)\in H_R$ for every $t\in \mathbb{R}$, and $\sup_{\tau\leq t} \|u(\tau)\|\leq C_u$. Since $u$ is an orbit, for $\tau \geq t$, $u(\tau)=S(\tau,t)u(t)\in H_R(\tau)\cap \mathcal{J}(\tau),$ so $u(t)\in S(t,\tau)(H_R(\tau)\cap \mathcal{J}(\tau))$ for every $\tau\geq t$.
\end{proof}

As a consequence of above results we will get the following corollary
\begin{Corollary}
	For any $R\geq R_0$ there holds 
	$$
	\lim_{\tau\to \infty}\mathrm{dist}(S(t,\tau) (H_R(\tau)\cap \mathcal{J}(\tau)), \mathcal{J}_b(t)) = 0
	$$
\end{Corollary}

As in the autonomous case, we can observe that if $u\in \mathcal{J}(t)\setminus \mathcal{J}_b(t)$ for a certain $t\in \mathbb{R}$, then $\lim_{\tau\to \infty}\|S(\tau,t)u\|=\infty$. The proof is similar to the autonomous case Lemma \ref{lemma:inf}

\begin{Lemma}
	Assume $(H1)_{\mathrm{NA}}-(H3)_{\mathrm{NA}}$, and let $u\in \mathcal{J}(t)\setminus\mathcal{J}_b(t)$ for a certain $t\in \mathbb{R}$. Then $\lim_{\tau\to \infty}\|S(\tau,t)u\|=\infty$.
\end{Lemma} 
	
\subsection{Pullback behavior of bounded sets.}
In the autonomous case, in Lemma \ref{lemma:possibilities} we provided the classification of all bounded sets in three classes. Unfortunately, in non-autonomous case such classification in no longer valid in pullback sense. It is very easy to construct the process $S(\cdot,\cdot)$ and a bounded set $B$ such that $S(t,s_n)B\subset H_R(t)$ and $S(t,r_n)B\subset Q(t)\setminus H_R(t)$ for sequences $s_n\to -\infty$ and $r_n\to -\infty$. In the concept of pullback $\omega$ limits it is natural to operate not on single bounded sets but on non-autonomous bounded sets $\{ B(t) \}_{t\leq T} $ defined for every time $t$ less or equal to some given $T$. 

\begin{Definition}
Let $\{B(t)\}_{t\leq T}$ be a non-autonomous bounded set in  $X$ (i.e. $B(t)\in \mathcal{B}(X)$ for every $t\leq T$). We define the unbounded pullback $\omega$-limit set as
     $$\omega(t,B(\cdot)):=\{y\in X \,:\  \textrm{there exist}\   s_n\rightarrow-\infty \ \textrm{and}\ u_n\in B(s_n) \text{ such that } y=\lim_{s_n\rightarrow-\infty}S(t,s_n)u_n\}.$$    
\end{Definition}

\begin{Observation}
Let $\{B_1(t)\}_{t\leq T_1}$ and $\{B_2(t)\}_{t\leq T_2}$ be a two families of subsets of $X$ such that $B_1(t) = B_2(t)$ for every $t\leq T_3,$ where $T_3 \leq \min\{ T_1,T_2\}$. 
Then for every $t\in \mathbb{R}$ the sets $\omega(t,B_1(\cdot))$ and
$\omega(t,B_2(\cdot))$ are equal.
\end{Observation}
The following characterization of $\omega$ limit sets remains valid in unbounded non-autonomous setting. 
\begin{Lemma}
 Let $\{B(t)\}_{t\leq T}$ be a non-autonomous set in $X.$ The following characterization holds
 \begin{equation*}
     \omega(t,B(\cdot))= \bigcap_{s\leq t}\overline{\bigcup_{\tau\leq s}S(t,\tau)B(\tau) .}
 \end{equation*}
\end{Lemma}
\begin{proof}
    In the proof we will denote set 
    $\omega_1(t,B(\cdot)) := \bigcap_{s\leq t}\overline{\bigcup_{\tau\leq s}S(t,\tau)B(\tau) }$. Let $x\in \omega(t,B(\cdot)).$ There exist sequences $s_n\to -\infty$ and $u_n\in B(s_n)$ such that 
    $S(t,s_n)u_n\to x.$ We have the inclusion 
    $\{S(t,s_n)u_n\}_{n\leq n_0}\subset 
    \bigcup_{\tau\leq s}S(t,\tau)B(\tau) $  for $n_0$ such that $s_n\leq s$ for any $n\geq n_0.$ So $x\in \overline{\bigcup_{\tau\leq s}S(t,\tau)B(\tau)}$ for any $s\leq t$ and consequently $x\in \omega_1(t,B(\cdot)).$
    Now suppose that $x\in \omega_1(t,B(\cdot)).$ Then for every $n\in \mathbb{N}$ there exist $s_n\leq n$ and $y_n\in S(t,s_n)B(s_n)$ such that $\norm{x - y_n } \leq \frac{1}{n}.$ As $y_n \in S(t,s_n) B(s_n)$
    we can find $u_n\in B(s_n)$ such that $\norm{x - S(t,s_n)u_n}\leq\frac{1}{n}.$ This implies the existence of sequences $s_n\to -\infty ,$ and $u_n\in B(s_n)$ such that 
    $S(t,s_n)u_n\to x.$ We deduce that $x\in \omega(t,B(\cdot)),$ which ends the proof.
\end{proof}

The following result shows that pullback $\omega$-limit sets are always positively invariant. 

\begin{Lemma}\label{le::fowardInvarianceNA}
    Let $\{B(t)\}_{t\leq T}$ be a non-autonomous set in $X$. We have
    $S(\tau,t)\omega(t,B(\cdot)) \subset \omega(\tau,B(\cdot))$ for every $\tau \geq t$. 
\end{Lemma}
\begin{proof}
    If $y\in \omega(t,B(\cdot))$ then there exist $s_n\to -\infty$ and $u_n\in B(s_n)$ such that $y = \lim_{n\to\infty}S(t,s_n)u_n$.
 From continuity of $S(\tau,t)$ we see that $S(\tau,t)y =\lim_{n\to\infty}S(\tau,s_n)u_n$ whence the assertion follows. 
\end{proof}

In the next results, which hold for sets which are backward bounded, we establish some properties of unbounded pullback $\omega$-limit sets. To this end we define
$$
\widetilde{\mathcal{B}} = \left\{ \{B(t)\}_{t\leq T}\, :\  B(t) \in \mathcal{B}(X)\ \textrm{for every}\ t\leq T\ \textrm{and}\ \bigcup_{t\leq T}B(t)\in \mathcal{B}(X)\right\}.
$$

\begin{Theorem}
    Let conditions $(H1)_{\mathrm{NA}}-(H3)_{\mathrm{NA}}$ hold. Assume that $\{B(s)\}_{s\leq T} \in \widetilde{\mathcal{B}}$. Then $\omega(t,B(\cdot))\subset \mathcal{J}(t)$ for every $t\in \mathbb{R}$.
\end{Theorem}

\begin{proof}
    Since $\bigcup_{s\leq T}B(s)$ is bounded then for time $r\in\mathbb{R}$ there exists a time $t_1(r)\leq r$ such that $$S(r,\tau)\left(\bigcup_{s\leq T}B(s)\right)\subset Q(\tau)$$ for every $\tau\leq t_1(r)$.  We take $y\in \omega(t,B(\cdot))$, then there exists a sequence $s_n\rightarrow -\infty$ and $u_n\in B(s_n)$ such that $y=\lim_{s_n\rightarrow -\infty}S(t,s_n)u_n$. For every $r\leq t$ $y=\lim_{s_n\rightarrow -\infty}S(t,r)S(r,s_n)u_n$. There exists $n_0(r)$ such that $S(r,s_n)u_n \in Q(r)$ for $n\geq n_0(r)$. Hence $y\in \overline{S(t,r)Q(r)}$ and the proof is complete.
\end{proof}

As a simple consequence of the above lemma and Lemma \ref{lemma:531} we obtain the following corollary.
\begin{Corollary}
   Let conditions $(H1)_{\mathrm{NA}}-(H3)_{\mathrm{NA}}$ hold. Assume that $\{B(s)\}_{s\leq T} \in \widetilde{\mathcal{B}}$. Then for every $t\in \mathbb{R}$ and $C\in \mathcal{B}(X)$ the set $\omega(t,B(\cdot))\cap C$ is relatively compact.
\end{Corollary}

\begin{Remark}\label{re:nonepty}
Let  $B(\cdot)\in \widetilde{\mathcal{B}}.$ If for some $t\in \mathbb{R}$ and $R>0$ there exist 
$t_1\leq t$ such that for every $s\leq t_1$ the intersection 
$S(t,s)B(s) \cap \{\norm{Px}\leq R \}$ is nonempty, 
then there exist $\widehat{R}\geq R_1$ such that for every $\tau<t$ there exists $\tau_1$ such that, the intersection
$S(\tau,s)B(s) \cap \{\norm{Px}\leq \widehat{R} \}$ is nonempty, for every $s\leq \tau_1.$ If we additionally assume that the image of a bounded set $B \in \mathcal{B}(X)$ through $S(t,s)$ must be bounded then for every $\tau>t$ there exists $\widehat{R}(\tau)$ and $\tau_1$ such that $S(\tau,s)B(s)\cap \{\norm{Px}\leq \widehat{R}\}$ is nonempty for every $s\leq \tau_1$. 
\end{Remark}
\begin{proof}
 Note that for some $\widehat{R}$ and some $t_2\leq t_1$, the set $S(t,s)B(s)\cap H_{\widehat{R}}(t)$ is nonempty for every $s\leq t_2.$ We can deduce that for every 
    $\tau \leq t$ there exists $\tau_1(\tau)$ such that 
    $S(\tau ,s)B(s)\cap H_{\widehat{R}}(\tau)$ is nonempty for every $s\leq \tau_1.$
    Indeed from assumption (H1)$_{\mathrm{NA}}$ it follows that $S(\tau,s)B(s)\subset S(\tau,s)(\cup_{s\leq T}B(s))  \subset Q(\tau)$ for every $s \leq 
    \tau_1$ for some 
    $\tau_1.$ Then if for $s \leq \tau_1$ we would have
    $S(\tau,s)B(s)\subset Q(\tau)\setminus H_{\widehat{R}}(\tau)$ then by assumption 
    (H2)$_{\mathrm{NA}}$ we would also have 
    $S(t,s)B(s)\subset Q(t)\setminus H_{\widehat{R}}(t)$ which would be a  contradiction. 
    By assumption (H2)$_{\mathrm{NA}}$ we have
    $H_{\widehat{R}} (t)\subset \{\norm{Px}\leq \widehat{R} \}$ for every $t\in \mathbb{R}$, so
    $S(\tau,s)B(s) \cap \{\norm{Px}\leq \widehat{R} \}$ is nonempty. We consider the case for $\tau > t$. We know that there exists $\tau_1$ such that for every $s\leq \tau_1$ there exists $u_s\in B(s)$ with $S(t,s)u_s \in Q(t)\cap \{ \|Px\|\leq R\}$, a bounded set. The image of this set via $S(\tau,t)$ is bounded and the assertion follows.
\end{proof}
The next two results concern the non-emptiness and attraction in bounded sets by pullback $\omega$-limit set. They both need that $S(t,s)B(s)$ intersects some ball in $E^+$ at some time $t$ for sufficiently small $s$. Using the above remark we deduce the nonemptiness and attraction of $\omega$-limit sets not only at this time but also at other times, in the past and in the future. Note that for the attraction in the future we need the additional assumption that the image of bounded set via $S(t,s)$ is a bounded set. This assumption is needed to deduce from the nonemptiness of intersection $S(t,s)B(s) \cap H_{R}$ that also the intersection $S(\tau,s)B(s)\cap H_{\widehat{R}}(\tau)$ is nonempty for some $\widehat{R}$. 
\begin{Lemma}\label{lem19}
Assume the conditions $(H1)_{\mathrm{NA}}-(H3)_{\mathrm{NA}}$. 
Let $\{B(s)\}_{s\leq T} \in \widetilde{\mathcal{B}}$ be a non-autonomous set such that there exist $R>0$ and $t\in\mathbb{R}$ and $t_1 < t$ such that the intersection 
$S(t,s)B(s)\cap \{ \|Px\|\leq R \} $ is nonempty for every $s\leq t_1.$
Then $\omega(\tau,B(\cdot))$ is nonempty for every $\tau\in \mathbb{R}.$ 
\end{Lemma}

\begin{proof}
    Let $\tau\leq t.$ By Remark \ref{re:nonepty} and assumption (H2)$_{\textrm{NA}},$ there exist $\widehat{R}\geq R_1$ and $\tau_1\leq \tau$
    such that, we have $S(\tau,s)B(s)\cap H_{\widehat{R}}(\tau)\neq \emptyset,$
    for every $s \leq \tau_1.$ 
    For any $s_n\leq \tau_1$ and $s_n\to -\infty$  we can find $u_n$ such that
    $u_n\in B(s_n)$ and $y_n = S(\tau,s_n)u_n \in H_{\widehat{R}}(\tau)$. Moreover for every $n$ we can find $k(n) > n$ such that $S(s_n, s_{k(n)})u_{k(n)} \in Q(s_n)$. Now $y_{k(n)} = S(\tau,s_{k(n)})u_{k(n)} \in S(\tau,s_n)Q(s_n)$.
    Now $y_{k(n)} \in L_n,$ where 
    $L_n = S(\tau,s_n)Q \cap H_{\widehat{R}} = S(\tau,s_n)H_{\widehat{R}}(s_n) \cap H_{\widehat{R}}(\tau).$
    From Lemma \ref{lemma:212} and assumption  $(H3)_{\textrm{NA}}$ from $y_{k(n)}$ we can pick a subsequence which converges to some $y\in X.$ So $\omega(\tau,B(\cdot))$ is nonempty for every $\tau\leq t.$
    For $\tau>t$ by Lemma \ref{le::fowardInvarianceNA} we have  $S(\tau,t) \omega(t,B(\cdot))\subset \omega(\tau,B(\cdot)).$ So as $\omega(t,B(\cdot))$ is nonempty the set $\omega(\tau,B(\cdot))$ is also nonempty.
\end{proof}

\begin{Theorem}
    Assume $(H1)_{\mathrm{NA}}-(H3)_{\mathrm{NA}}$ and let $t\in \mathbb{R}$. Suppose that  $\{B(s)\}_{s\leq T} \in \widetilde{\mathcal{B}}$ is such that for some  $R > 0$ there exists $t_1\leq t$ such that $S(t,s)B(s)\cap \{\|Px\|\leq R\}\neq \emptyset$ for every $s\leq t_1$. Then 
    $$\lim_{s\rightarrow -\infty}\mathrm{dist}(S(t,s)B(s)\cap \{\|Px\|\leq R\},\omega(t,B(\cdot)))\rightarrow 0.$$
\end{Theorem}
\begin{proof}
    We suppose that there exists a sequence $s_n\rightarrow -\infty$ and $u_n\in B(s_n)$ such that $\|P(S(t,s_n)u_n)\|\leq {R}$ and for some $\varepsilon>0$, we have dist$(S(t,s_n)u_n, \omega(t,B(\cdot)))>\varepsilon$. Since $\bigcup_{s\leq T}B(s)$ is bounded, there exists $k(n)$ such that $S(s_n,s_k)u_k\in Q(s_k)$. By (H2)$_{\mathrm{NA}}$ we can find  $\widehat{R}$ such that $\{S(t,s_k)u_k\}_{k\in \mathbb{N}}\subset H_{\widehat{R}}(t)$, and for every $k$,. Then $y_{k(n)}=S(t,s_n)S(s_n,s_{k(n)})u_{k(n)}\in S(t,s_n)Q(s_n)\cap H_{\widehat{R}}(t)=L_n$, and this is a nested sequence of sets, so by (H3)$_{\mathrm{NA}}$ and Lemma \ref{lemma:212} there exists a subsequence such that $y_k\rightarrow y$ for some $y\in H_{\widehat{R}}(t)$. Since $u_{k(n)}\in B(s_{k(n)})$ and $s_{k(n)}\rightarrow-\infty$ we have that $y\in \omega(t,B(\cdot))$ so we arrive to a contradiction.
\end{proof}

\begin{Lemma}
Under assumptions of the previous lemma there exists $\widehat{R}$ such that for every $\tau< t$ we have 
$$\lim_{s\rightarrow -\infty}\mathrm{dist}(S(\tau,s)B(s)\cap \{\|Px\|\leq \widehat{R}\},\omega(\tau,B(\cdot)))\rightarrow 0.$$ If additionally we assume that $S(t,s)B$ is bounded for every bounded set $B$ then for every $\tau > t$ there exists $\widehat{R}(\tau)$ such that  
$$\lim_{s\rightarrow -\infty}\mathrm{dist}(S(\tau,s)B(s)\cap \{\|Px\|\leq \widehat{R}(\tau)\},\omega(\tau,B(\cdot)))\rightarrow 0.$$
\end{Lemma}
\begin{proof}
    The proof follows from the same argument as the previous lemma, which is possible by Remark \ref{re:nonepty} as intersection in the statement of the Lemma are nonempty and Hausdorff semidistance is well defined.  
\end{proof}

\begin{Lemma}
    Assume (H1)$_{\mathrm{NA}}$-(H3)$_{\mathrm{NA}}.$ 
    For 
    $\{B(s)\}_{s\leq T} \in \widetilde{\mathcal{B}},$ the unbounded pullback $\omega$-limit sets are invariant, that is $S(\tau,t) \omega(t,B(\cdot))= \omega(\tau,B(\cdot)).$
\end{Lemma}
\begin{proof}
    Inclusion $S(\tau,t) \omega(t,B(\cdot))\subset \omega(\tau,B(\cdot))$ is asserted in Lemma \ref{le::fowardInvarianceNA}. Let $y\in \omega(\tau,B(\cdot)).$ 
    There exist sequences $s_n \to -\infty$ and $u_n\in B(s_n)$ such that $\lim_{n\to\infty} S(\tau,t ) S(t,s_n)u_n = y.$ Arguing as in the proof of Lemma \ref{lem19}, for a subsequence we have $S(t,s_n)u_n \to x$ with some $x\in \omega(t,B(\cdot))$ and it must be that $S(\tau,t)x = y$ which ends the proof.
\end{proof}

In order to relate the pullback behavior of bounded sets with $\mathcal{J}_b(t)$ we need to define the second universe of non-autonomous bounded sets. This time the sets do not have to be backward bounded, but their evolution must stay in $H_R(\cdot)$ at every time and for sufficiently small initial time.    
	     \begin{align*}
	        & \widehat{\mathcal{B}}=\{\{B(s)\}_{s\leq T} \, : \ B(s)\in \mathcal{B}(X)\ \textrm{and there exists} \  R\geq R_0 \ \textrm{such that for every}\ t\in \mathbb{R}\\
	         & \qquad \qquad \ \textrm{there exists}\  t_1(t)\leq T \ \textrm{ such that } \ S(t,s)B(s)\subset H_R(t) \ \textrm{for every}\ s\leq t_1\}
	     \end{align*}

\begin{Theorem}
	Assume $(H1)_{\mathrm{NA}}-(H3)_{\mathrm{NA}}$. For every $B(\cdot)\in\widehat{\mathcal{B}} $ we have the following inclusion:
	$$\omega(t,B(\cdot))\subset \mathcal{J}_b(t)$$
\end{Theorem}
\begin{proof}
   Take $y\in \omega(t, B(\cdot))$, then there exists a sequence $s_n\rightarrow-\infty$ and $u_n\in B(s_n)$ such that  $\lim_{s_n\rightarrow-\infty}S(t,s_n)u_n=y$. Pick $s\leq t$. There exists $n_0(s)$ such that for every $n\leq n_0$ we have $S(s,s_n)u_n\in H_R(s)\subset Q(s)$. We deduce that $u\in \mathcal{J}(t)$. It remains to show that the {\color{black} solution} starting from $y$ at time $t$ is forward bounded. We take $r\geq t$,
	$$S(r,t)u=\lim_{s_n\rightarrow-\infty}S(r,t)S(t,s_n)u_n=\lim_{s_n\rightarrow-\infty}S(r,s_n)u_n$$
	and $S(r,s_n)u_n\in H_R(r)$ for every sufficiently large $n$, so the limit also belongs to $H_R(r)$ and the proof is complete. 
\end{proof}

\begin{Lemma}
	Under assumptions of previous lemma $\{\omega(r,B(\cdot))\}_{r\in \mathbb{R}}$ are nonempty, compact and invariant sets which pullback attract $B(\cdot)$ at the time $r$.  
\end{Lemma}
\begin{proof}
    The compactness follows from the compactness of $\mathcal{J}_b(t)$ and closedness of pullback $\omega$-limits. To prove the nonemptiness we take a sequence of times $\{s_n\}_{n\in \mathbb{N}}$ going to $-\infty$ and $u_n\in B(s_n)$, such that $y_n=S(t, s_n)u_n\in H_R(t)$. We will prove that $y_n$ is relatively compact. Indeed, for every $n\in \mathbb{N}$, there exists $k(n)>n$ such that $S(s_n,s_{k(n)})B(s_{k(n)})\subset H_R(s_n)$, whence $y_{k(n)}=S(t,s_{k(n)})u_{k(n)}\in S(t,s_n)H_R(s-n)\cap H_R(t)=L_n$. This is a nested sequence of sets, so by Lemma \ref{lemma:212} and (H3)$_{\mathrm{NA}}$, we obtain a subsequence convergent to some $y\in \omega(t,B(\cdot))$.   The attraction follows in a standard way. 
    For contradiction take a sequence $u_n\in B(s_n)$, $s_n\rightarrow -\infty$ such that dist$(S(t,s_n)u_n,\omega(t,B(\cdot))>\varepsilon$. With the same method as to see the non-emptiness we obtain that there exists a subsequence of $\{S(t,s_n)u_n\}_{n\in \mathbb{N}}$ such that converges to $y$, so this $y\in \omega(t,B(\cdot))$ and we arrive to a contradiction. 
    From Lemma \ref{le::fowardInvarianceNA} we have the positive invariance. For the negative invariance, we take $y\in \omega(t,B(\cdot))$, so there exists a sequence $\{s_n\}_{n\in \mathbb{N}}$ and $u_n\in B(s_n)$ such that $y=\lim_{s_n\rightarrow-\infty}S(t,s_n)u_n$. We take $\tau <t$, and we write $y=\lim_{s_n\rightarrow -\infty}S(t,\tau)S(\tau, s_n)u_n$. With the same method as in the proof of nonemptiness we see that there exists a subsequence $\{S(\tau,s_{\nu})u_{\nu}\}\subset \{S(\tau, s_n)u_n\}_{n\in\mathbb{N}}$ which converges to some $x\in X$, and $x\in \omega(\tau, B(\cdot))$. Then $S(t,\tau)x=\lim_{s_{\nu}\rightarrow-\infty}S(t,\tau)S(\tau,s_{\nu})u_{\nu}=\lim_{s_{\nu}\to -\infty}S(t,s_{\nu})u_{\nu} = y$ and the proof is complete. 
\end{proof}

\section{Unbounded attractors for problems governed by PDEs}\label{PDES}
\subsection{Problem setting and variation of constant formulas.} \label{sec:setup}We present here the framework which fits in the formalism presented in previous Section. While it is not in the highest possible generality, its trait of novelty with respect to \cite{chepyzhov, bengal} relies in the fact that its estimates are based on the Duhamel formula rather than in the energy inequalities which follow from testing the equation of the problem by appropriate functions. Let $X$ be a Banach space with a norm $\|\cdot\|$ and let $A:X\supset D(A) \to X$ be a linear, closed and densely defined operator. Assume moreover, that the operator $-A$ is sectorial and has compact resolvent. Then, each point in $\sigma(A)$, the spectrum of $A$, is the eigenvalue, and this set of eigenvalues is discrete. Assume that $\sigma(A) \cap \{ \lambda\in \mathbb{C}\,:\ \text{Re}\, \lambda= 0 \}  = \emptyset$. Then, denote by $P$ by the spectral projection associated with the part of the spectrum $\sigma(A)\cap \{ \lambda\in \mathbb{C}\,:\ \text{Re}\, \lambda > 0 \}$. The range of $P$ is finite dimensional, this will be the space $E^+$, and the range of $I-P$ will be $E^-$, a closed subspace of $X$.   The restriction of $A$ to the range of $P$ is a bounded linear operator. Moreover, for some positive numbers $\gamma_0,\gamma_1,\gamma_2$, the following inequalities hold
\begin{equation}\label{eq:hyper}
\begin{split}
&\|e^{A(t-s)}\| \leq M e^{\gamma_0(t-s)}\quad\textrm{for}\ \ \ t\geq s,\\
&\|e^{A(t-s)}(I-P)\|\leq M e^{-\gamma_2(t-s)}\quad\textrm{for}\ \ \ \ t\geq s,\\
&\|e^{A(t-s)}P\|\leq M e^{\gamma_1(t-s)}\quad\textrm{for}\ \ \ \ t\leq s,
\end{split}
\end{equation}
cf. \cite[page 147]{Carvalho-Langa-Robinson-13}. 
We will consider the following operator equation
\begin{equation}\label{pde}
u'(t) = Au(t) + f(t,u),
\end{equation}
with the initial data $u(t_0) = u_0 \in X$. 
We assume that $f$ is bounded, that is, $\|f(t,u)\|\leq C_f$ for every $u\in X$ and $t\in \mathbb{R}$.  We also need to make the assumptions on $f$ which guarantee that the above problem has the unique {\color{black} solution} $u:[t_0,\infty) \to X$ given by the following variation of constants formula
\begin{equation}\label{def:fvc}
u(t)=e^{A(t-\tau)}u(\tau) +\int_\tau^t e^{A(t-s)}f(s,u(s))ds\ \ \textrm{for} \ t\geqslant \tau,\\
\end{equation}
such that $u\in C([t_0,T];X)$ for every $T\geq t_0$. 
We stress again, that while we do not pursue the higest generality here, it is enough to assume that $f:\mathbb{R}\times X\to X$ is continuous and $f(t,\cdot):X\to X$ is Lipschitz continuous on bounded sets of $X$ with the constant independent on $t$, that is if only $\|v\|,\|w\| \leq R$, then 
$$
\|f(t,v)-f(t,w)\| \leq C(R) \|v-w\|.
$$
For the details of the proof of the {\color{black} solution} existence, uniqueness, and its globality in time see for example \cite[Chapter 6.7]{Carvalho-Langa-Robinson-13} or \cite[Chapters 2 and 3]{Cholewa-Dlotko}. Moreover, the {\color{black} solutions} constitute the continuous process $S(t,\tau)$ in $X$, such that for a bounded set $B\in \mathcal{B}(X)$ the sets $S(t,\tau)B$ are bounded in the space which is compactly embedded in $X$, and hence relatively compact in $X$ for $t>\tau$, c.f. for example \cite[Chapter 3.3]{Cholewa-Dlotko}, where the proof of compactness is done for the autonomous case. Thus, assumption (H3)$_{\mathrm{NA}}$, and for autonomous case (H3), is satisfied. 

If $u(s)$ is a {\color{black} solution} to the above problem on some interval $[\tau,t]$ then we denote $p(s)=Pu(s)$ and $q(s)=(I-P)u(s)$. Thus, $q$ and $p$ satisfy
\begin{subequations}\label{def:pq}
	\begin{align}
	p(t)&=e^{A(t-\tau)}p(\tau) +\int_\tau^t e^{A(t-s)}Pf(s,p(s)+q(s))ds\ \ \textrm{for}\ \ t\geqslant \tau,\label{solp(t)}\\
	q(t)&=e^{A(t-\tau)}q(\tau) +\int_{\tau}^te^{A(t-s)}(I-P)f(s,p(s)+q(s))ds\ \ \textrm{for}\ \ t\geqslant \tau. \label{solq(t)}
	\end{align}
\end{subequations}
 
\subsection{Estimates which follow from boundedness of $f$ and their consequences.}\label{def:fbound}
We begin with the estimate for $q$. From \eqref{solq(t)} we deduce in a straightforward way that 
\begin{equation}\label{eq:attrbound}
\|q(t)\| \leq M e^{-\gamma_2(t-\tau)}\|q(\tau)\| + \frac{M C_f}{\gamma_2} (1-e^{-\gamma_2 (t-\tau)}), 
\end{equation}
for $t\geq \tau \geq t_0$. This is enough to guarantee (H1) and its non-autonomous version (H1)$_{\text{NA}}$. Indeed, one can take, in autonomous case 
$$
Q = \overline{\bigcup_{t\geq 0}S(t)\left\{ \|(I-P)u\|\leq \frac{MC_f}{\gamma_2}+1 \right\}},
$$
and in non-autonomous case
$$
Q(t) = \overline{\bigcup_{s\leq t}S(t,s)\left\{ \|(I-P)u\|\leq \frac{MC_f}{\gamma_2}+1 \right\}}.
$$
Then (H1) and (H1)$_{\textrm{NA}}$ hold with $D_1 = \frac{MC_f}{\gamma_2}+1$ and $D_2 = M\left(\frac{MC_f}{\gamma_2}+1\right)$. 
We continue with the estimates for $p$. First, we have the following upper bound
\begin{equation}\label{pupper}
\|p(t)\| \leq Me^{\gamma_0(t-\tau)}\|p(\tau)\| + \frac{MC_f}{\gamma_0} (e^{\gamma_0(t-\tau)}-1), 
\end{equation}
valid for $t\geq \tau \geq t_0$.
Moreover \eqref{solp(t)} implies that 
\begin{equation}
p(\tau)=e^{A(\tau-t)}p(t) +\int_t^\tau e^{A(\tau-s)}Pf(s,p(s)+q(s))ds,
\end{equation}
for $t\geq \tau \geq t_0$, whence 
$$
\|p(\tau)\|\leq Me^{\gamma_1 (\tau-t)}\|p(t)\| + \frac{MC_f}{\gamma_1}.
$$
This means that 
\begin{equation}\label{eq:pgrow}
\|p(t)\| \geq e^{\gamma_1(t-\tau)} \left(\frac{1}{M}\|p(\tau)\| - \frac{C_f}{\gamma_1}\right).
\end{equation}
Clearly, if the projection $p$ of the initial data is sufficiently large, namely $\|p(t_0)\| > \frac{MC_f}{\gamma_1}$, then $\|p(t)\|$ tends exponentially to infinity as $t\to \infty$. This is not enough, however, to ensure (H2) and its non-autonomous counterpart (H2)$_{\text{NA}}$. The technical difficulty lies in the fact that, although real parts of all eigenvalues of $A$ on $E^+$ are positive, this does not have to hold for $A+A^\intercal$, the symmetric part of $A$. For simplicity we associate $A$ on $E^+$ with the matrix being its representation in some basis of $E^+$ and vectors $p\in E^+$ with their representations in this basis. Denoting the Euclidean norm and the associated matrix norm by $|\cdot|$, for $p\in E^+$, the finite dimensional space, one has the equivalence $C_1\|p\|\leq |p|\leq C_2\|p\|$. Now one can rewrite the variation of constants formula \eqref{solp(t)} as the following non-autonomous ODE
$$
p'(t) = Ap(t) + P f(t,p(t)+q(t)).
$$  
Then consider $N$, a symmetric and positive definite matrix which is the unique {\color{black} solution} of the Lyapunov equation
$
A^\intercal N + NA = I.
$ 
This $N$ is given by the formula $$N = \int_0^\infty e^{-t(A+A^\intercal)}\, dt.$$ Now
\begin{align*}
& \frac{d}{dt}(p^\intercal(t) N p(t)) = p^\intercal(t) (A^\intercal N + NA)p(t) + 2(P f(t,p(t)+q(t)))^\intercal N p(t)\\
& \ \  = |p(t)|^2 + 2(P f(t,p(t)+q(t)))^\intercal N p(t).
\end{align*}
We deduce that 
$$
\frac{d}{dt}(p^\intercal(t) N p(t))\geq C_1^2\|p(t)\|^2 - 2C_2^2 C_f |N| \|p(t)\| = \|p(t)\| (C_1^2\|p(t)\| - 2C_2^2 C_f |N|).
$$
This means that if only $$\|p(t)\| > \frac{2C_2^2 C_f |N|}{C_1^2},$$ then the quadratic form $p^\intercal(t) N p(t)$ is strictly increasing. 
Now, as $N$ is positively defined, there exist positive constants $D_1, D_2$ such that $D_1^2\|p\|^2 \leq p^\intercal N p \leq D_2^2\|p\|^2$. It follows that (H2) and its nonautonomous version (H2)$_{\text{NA}}$ hold with
$$
R_0 > \frac{2C_2^2 C_f |N|}{C_1^2}, \quad R_1 = \frac{D_2}{D_1}R_0,\quad S(R) = \frac{D_1}{D_2}R,\quad H_R = Q \cap \{ u \in X\, :\  (Pu)^\intercal N (Pu) \leq D_1^2 R^2\}.   
$$
Observe, that \eqref{eq:pgrow} imply that if the problem is autonomous also (H4) holds. We have thus verified (H1)-(H4) of Section \ref{sec:autonomous} and (H1)$_{\text{NA}}$-(H3)$_{\text{NA}}$ of Section \ref{nonauto}, which implies by Theorem \ref{thm:pullback} the existence of unbounded pullback attractor, and, for the autonomous problem, the existence of the unbounded attractor by Theorem \ref{th:autonomus1}.  
\subsection{Thickness of unbounded attractor.} In this section we derive the estimates which guarantee that the thickness in $E^-$ of the unbounded attractor tends to zero as $p\in E^+$ tends to infinity, or in the language of the multivalued inertial manifold, that $\lim_{\|p\|\to \infty}\textrm{diam}\, \Phi(p) = 0$. For simplicity we deal only with the autonomous case. The fact that such thickness tends to zero guarantees,  by Theorem \ref{attracts:all} that the unbounded attractor $\mathcal{J}$ attracts the bounded sets not only in the sense
 \begin{align*}
& \lim_{t\to \infty} \mathrm{dist}(S(t)B \cap \{ \|Pu\|\leq R\},\mathcal{J} ) = 0\ \ \textrm{if only}\ B,R\ \  \textrm{are such that}\\
& \ \ \qquad \qquad \textrm{the sets}\ \  S(t)B\cap \{ \|Pu\|\leq R\}\ \ \text{are nonempty for every sufficiently large}\ t, 
\end{align*}
but also in the sense 
$$
\lim_{t\to \infty} \mathrm{dist}(S(t)B ,\mathcal{J} ) = 0.
$$ 
In order to obtain the desired result we need to assume that the projection $I-P$ of the nonlinearity $f$ decays to zero as the projection $P$ or the argument tends to infinity. We make the following assumption which makes precise what rate of this decay is needed.   
\begin{itemize}
	\item [(Hf1)] The function $f:X\to X$ has the form $f(u) = f_0 + f_1(u)$, where $f_0\in X$ and there exists $K>0$ such that if $\|Pu\| \geq  K$, then   
	$$
	\|(I-P)f_1(u)\| \leq H(\|Pu\|)
	$$
	where $H:[K,\infty)\to (0,\infty)$ is nonincreasing, $\lim_{r\to \infty} H(r) = 0$ and either 
	\begin{equation}\label{as1}
	H(r) \leq \frac{D}{r^\alpha }\ \ \textrm{with}\ \ \alpha>0\ \ \textrm{and}\ \ D>0\ \ \textrm{for}\ \ r\geq K, 
	\end{equation}
	or 
\begin{equation}\label{as2}
	[K,\infty)\ni r \mapsto r^{\frac{\gamma_2}{\gamma_1}-\varepsilon}H(r)\in \mathbb{R}\ \ \textrm{is nondecreasing for} \ \ \textrm{for some}\ \ \varepsilon \in \left(0,\frac{\gamma_2}{\gamma_1}\right).  
	\end{equation}
\end{itemize}
\begin{Theorem}\label{thm:decay}
	If, in addition to the assumptions of Section \ref{sec:setup} we assume (Hf1), then 
	$$
	\lim_{t\to \infty} \mathrm{dist}(S(t)B ,\mathcal{J} ) = 0\ \ \textrm{for every}\ \ B\in \mathcal{B}(X).
	$$
	and
	$$
	\lim_{\|p\|\to \infty}\mathrm{diam}\, \Phi(p) = 0,
	$$ 
	more precisely if \eqref{as1} holds then there exists constants $\beta>0$ (depending on $\gamma_0, \gamma_1, \gamma_1, \alpha$) and $M_1 > 0$ such that for sufficiently large $\|p\|$
	$$
	\mathrm{diam}\, \Phi(p) \leq M_1 \frac{1}{\|p\|^\beta},
	$$
	and if \eqref{as2} holds, there exists constants $M_2, M_3 > 0$ such that 
	$$
	\mathrm{diam}\, \Phi(p) \leq  M_2 H\left(M_3 \|p\|{^\frac{\gamma_1}{\gamma_0}}\right),
	$$
\end{Theorem}
\begin{Remark}
	The proof for the case \eqref{as1} with  the use of the energy estimates instead of the Duhamel formula and $f_0 = 0$  appears in \cite[Theorem 6.2]{chepyzhov}. The main new contribution of the above result is allowing also for the case \eqref{as2} which makes it possible to consider very slow decay of $H$, such as $H(s) \leq  \frac{C}{\ln\, s}$.  
	\end{Remark}
\begin{Remark}
	From the course of the proof it is possible to derive the exact values of the constants $M_1, M_2, M_3$, and $\beta$. Particular interest lies in the decay rate $\beta$. Namely, if   $\gamma_2 > \gamma_1\alpha$, then $\beta = \frac{\gamma_1 \alpha}{\gamma_0}$, if $\gamma_2 < \gamma_1\alpha$, then $\beta = \frac{\gamma_2}{\gamma_0}$, and if $\gamma_2 < \gamma_1\alpha$ then $\beta$ could be taken as any positive number less than $\frac{\gamma_2}{\gamma_0}$. Moreover, an obvious modification of the proof leads in this last case to the decay rate $\beta=\frac{\gamma_2}{\gamma_0}$ with a logarithmic correction.  
\end{Remark}
\begin{proof}
	As $0$ is in the resolvent of $A$ we can find $\overline{u}\in X$ the unique {\color{black} solution} of $A \overline{u} = f_0$. 
	Let $u$ be the {\color{black} solution} of \eqref{pde} with the initial data $u_0$ such that
	\begin{equation}\label{ball}
	\|Pu_0\|\geq \|P\overline{u}\| + M \left(K+\frac{C_f+\|f_0\|}{\gamma_1}\right)
		\end{equation}
		 Denote $v = u-\overline{u}$. Then $v$ solves the equation
$$
	v'(t) = Av(t) + f_1(v(t)+\overline{u}). 
	$$ 
	The problem governed by this equation has the unbounded attractor $\mathcal{J}$ and then $\overline{u} + \mathcal{J}$ is the unbounded attractor for the original problem. We denote $Pv(t) = p(t)$ and $(I-P)v(t) = q(t)$. These functions satisfy the equations 
	\begin{subequations}\label{def:pq1}
		\begin{align}
		p(t)&=e^{A(t-\tau)}p(\tau) +\int_\tau^t e^{A(t-s)}Pf_1(p(s)+q(s)+\overline{u})ds\ \ \text{for}\ \ t\geqslant \tau\ \ \label{solp(t)1}\\
		q(t)&=e^{A(t-\tau)}q(\tau) +\int_{\tau}^te^{A(t-s)}(I-P)f_1(p(s)+q(s)+\overline{u})ds\ \ \text{for}\ \ t\geqslant \tau. \label{solq(t)1}
		\end{align}
	\end{subequations}
As $\|f_1(u)\| \leq C_f + \|f_0\|$, analogously as in Section \ref{def:fbound} from \eqref{solp(t)1} we obtain
$$
e^{\gamma_1 t} \left(\frac{1}{M}\|P(u_0-\overline{u})\| - \frac{C_f+\|f_0\|}{\gamma_1}\right) \leq   \|p(t)\|\ \ \textrm{for}\ \ t\geq 0. 
$$
Then $\|p(t)+P\overline{u}\| = \| P(p(t)+q(t)+\overline{u})\| \geq K e^{\gamma_1 t} \geq K$ for $t\geq 0$. Now, from \eqref{solq(t)1} it follows that
\begin{equation}\label{qdecr}
\|q(t)\| \leq Me^{-\gamma_2 (t-\tau)}\|q(\tau)\| + Me^{-\gamma_2 t}\int_\tau^t e^{\gamma_2 s} H(\|p(s) + P\overline{u}\|)\, ds.
\end{equation}
We proceed separately for \eqref{as1} and \eqref{as2}. If \eqref{as1} holds then 
$$
\|q(t)\| \leq Me^{-\gamma_2 (t-\tau)}\|q(\tau)\| + MDe^{-\gamma_2 t}\int_\tau^t e^{\gamma_2 s} \frac{1}{\|p(s) + P\overline{u}\|^\alpha }\, ds.
$$
But since
$$
\frac{1}{\|p(s) + P\overline{u}\|^\alpha} \leq \frac{e^{-\gamma_1\alpha s}}{K^\alpha},
$$
we deduce that
$$
\|q(t)\| \leq Me^{-\gamma_2 (t-\tau)}\|q(\tau)\| + \frac{MD}{K^\alpha}e^{-\gamma_2 t}\int_\tau^t e^{(\gamma_2-\gamma_1\alpha) s}\, ds.
$$
Taking $\tau = 0$ we obtain 
$$
\|q(t)\| \leq Me^{-\gamma_2 t}\|q(0)\| + \frac{MD}{K^\alpha}e^{-\gamma_2 t}\int_0^t e^{(\gamma_2-\gamma_1\alpha) s}\, ds.
$$
Now if $\gamma_2 > \gamma_1\alpha$ then 
$$
\|q(t)\| \leq Me^{-\gamma_2 t}\|q(0)\| + \frac{MD}{K^\alpha(\gamma_2-\gamma_1\alpha)}e^{-\gamma_1\alpha t} \leq M e^{-\gamma_1\alpha t}\left( \frac{D}{K^\alpha(\gamma_2-\gamma_1\alpha)}+\|q(0)\| \right).
$$
If $\gamma_2 = \gamma_1\alpha$, then 
\begin{align*}
& \|q(t)\| \leq Me^{-\gamma_2 t}\|q(0)\| + \frac{MD}{K^\alpha}e^{-\gamma_2 t}t \leq M\max\left\{ \|q(0)\|, \frac{D}{K^\alpha}\right\}e^{-\gamma_2t}(1+t) \\
& \ \ \leq M\max\left\{ \|q(0)\|, \frac{D}{K^\alpha}\right\}\frac{e^\varepsilon}{\varepsilon e} e^{-(\gamma_2-\varepsilon)t}\ \ \textrm{for small}\ \ \varepsilon>0.
\end{align*}
Finally, if $\gamma_2 < \gamma_1\alpha$, then 
$$
\|q(t)\| \leq Me^{-\gamma_2 t}\left(\|q(0)\| + \frac{D}{K^\alpha(\gamma_1\alpha-\gamma_2)}\right).
$$
That shows the exponential attraction as $t\to \infty$ towards the finite dimensional linear manifold $\overline{u}+R(P) = \overline{u}+E^+$ for the projection $P$ of the initial data outside a large ball given by \eqref{ball}, for $f$ that satisfies (Hf) with \eqref{as1}. 
Now consider \eqref{as2}. As $H$ is nonincreasing, we deduce, for $\tau=0$ in \eqref{qdecr}
$$
\|q(t)\| \leq Me^{-\gamma_2 t}\|q(0)\| + Me^{-\gamma_2 t}\int_0^t e^{\gamma_2 s} H(Ke^{\gamma_1 s})\, ds.
$$ 
It follows that
\begin{align*}
& \|q(t)\| \leq Me^{-\gamma_2 t}\|q(0)\| + Me^{-\gamma_2 t}\int_0^t e^{\gamma_2 s} (Ke^{\gamma_1 s})^{-\frac{\gamma_2}{\gamma_1}+\varepsilon}(Ke^{\gamma_1 s})^{\frac{\gamma_2}{\gamma_1}-\varepsilon}H(Ke^{\gamma_1 s})\, ds\\
& \ \ \ \leq Me^{-\gamma_2 t}\|q(0)\| + Me^{-\gamma_2 t}(Ke^{\gamma_1 t})^{\frac{\gamma_2}{\gamma_1}-\varepsilon}H(Ke^{\gamma_1 t})\int_0^t e^{\gamma_2 s} (Ke^{\gamma_1 s})^{-\frac{\gamma_2}{\gamma_1}+\varepsilon}\, ds\\
& \ \ \ =  Me^{-\gamma_2 t}\|q(0)\| + Me^{-\gamma_2 t}e^{\gamma_2t}e^{-\varepsilon\gamma_1 t}H(Ke^{\gamma_1 t})\int_0^t e^{\varepsilon \gamma_1 s} \, ds \\
& \ \ \ \leq   M\left(e^{-\gamma_2 t}\|q(0)\| + \frac{1}{\varepsilon \gamma_1}H(Ke^{\gamma_1 t})\right).
\end{align*}
Hence, for the case \eqref{as2} we get the attraction towards the finite dimensional linear manifold $\overline{u}+R(P) = \overline{u}+E^+$ as $t\to \infty)$ this time with the rate given by  $H(Ke^{\gamma_1 t})$. 

Now let $\{u(t)\}_{t\in \mathbb{R}}$ be the {\color{black} solution} in the unbounded attractor and let $u(t) = p(t) + q(t)$. Then, by \eqref{eq:pgrow} either there exists $\tau \in \mathbb{R}$ such that $\|p(\tau)\| > \frac{MC_f}{\gamma_1}$ and then $\lim_{t\to\infty}\|p(t)\| = \infty$ or  we must have $\|p(t)\| \leq \frac{MC_f}{\gamma_1}$ for every $t \in \mathbb{R}$.  Moreover, also from \eqref{eq:pgrow}, we deduce that for every {\color{black} solution} in the unbounded attractor it must be 
$$\limsup_{\tau \to -\infty}\|p(\tau )\| \leq \frac{MC_f}{\gamma_1}. 
$$
This means that if only $u\in \mathcal{J}$ and $\|Pu\| > \|P\overline{u}\| + M \left(K+\frac{C_f+\|f_0\|}{\gamma_1}\right) + 1$ then there exists $v\in \mathcal{J}$ with  $\|Pv\| = \|P\overline{u}\| + M \left(K+\frac{C_f+\|f_0\|}{\gamma_1}\right) + 1$ and $t>0$ such that $v = S(t)u$.
From \eqref{pupper} we deduce that 
\begin{equation}\label{14}
\|Pu\| \leq Me^{\gamma_0 t}\left(\|P\overline{u}\| + M \left(K+\frac{C_f+\|f_0\|}{\gamma_1}\right) + 1 + \frac{C_f}{\gamma_0}\right),
\end{equation}
whence
$$
e^{-\gamma_0 t} \leq M\left(\|P\overline{u}\| + M \left(K+\frac{C_f+\|f_0\|}{\gamma_1}\right) + 1 + \frac{C_f}{\gamma_0}\right) \frac{1}{\|P u\|}.
$$
We have proved that in the case  \eqref{as1} there exist constants $C, \kappa>0$ such that
$$
\| (I-P) (u - \overline{u})\| \leq C e^{-\kappa t}.
$$
 Combining the two above estimates leads us to 
$$
\| (I-P) (u - \overline{u})\| \leq C M^{\frac{\kappa}{\gamma_0}}\left(\|P\overline{u}\| + M \left(K+\frac{C_f+\|f_0\|}{\gamma_1}\right) + 1 + \frac{C_f}{\gamma_0}\right)^{\frac{\kappa}{\gamma_0}} \frac{1}{\|P u\|^{\frac{\kappa}{\gamma_0}}}.
$$
On the other hand, if \eqref{as2} holds then we have proved that there exist constants $C_1, C_2 > 0 $ such that
$$
 \|(I-P) (u - \overline{u})\| \leq  C_1 e^{-\gamma_2 t} +  C_2 H(Ke^{\gamma_1 t}).
$$
Combining this bound with \eqref{14} leads to
\begin{align*}
& \|(I-P) (u - \overline{u})\| \leq  C_1 M^{\frac{\gamma_2}{\gamma_0}}\left(\|P\overline{u}\| + M \left(K+\frac{C_f+\|f_0\|}{\gamma_1}\right) + 1 + \frac{C_f}{\gamma_0}\right)^{\frac{\gamma_2}{\gamma_0}} \frac{1}{\|P u\|^{\frac{\gamma_2}{\gamma_0}}}\\
& \qquad +  C_2 H\left(KM^{-\frac{\gamma_1}{\gamma_0}}\left(\|P\overline{u}\| + M \left(K+\frac{C_f+\|f_0\|}{\gamma_1}\right) + 1 + \frac{C_f}{\gamma_0}\right)^{{-\frac{\gamma_1}{\gamma_0}}} \|Pu\|{^\frac{\gamma_1}{\gamma_0}}\right).
\end{align*}
It is not hard to check that for sufficiently large $\|Pu\|$ the second term in the above estimate dominates over the first term, which completes the proof.
\end{proof}

\subsection{The case of small Lipshitz constant. Unbounded attractors and inertial manifolds.}
In the previous section we have proved, that, provided the $(I-P)$ projection of the nonlinearity $f(u)$ tends to zero with appropriate rate as long as $\|Pu\|$ tends to infinity, the thickness of the unbounded attractor also tends to zero and it attracts bounded sets not only in bounded sets, but also on the whole space. 

In this section we consider another situation when  the attraction holds on the whole space, namely the nonlinearity $f$ is assumed to be Lipschitz with sufficiently small  constant. We will also assume that $f$ does not depend on time directly, leaving the non-autonomous case, for now, open. 

We will make the following assumption: one for the graph transform method and one for the Lyapunov--Perron method

\begin{itemize}
		\item [(Hf2)$_{\textrm{GT}\bigcup \textrm{LP}}$] The function $f:X\to X$ satisfies
	$$
	\| f(u) - f(v) \| \leq L_f \|u-v\|\ \ \textrm{for}\ \ u,v\in X,
	$$ 
	with
	\item [(Hf2)$_{\textrm{GT}}$] either $M=1$ and $L_f < \begin{cases}&  \frac{\gamma_1\gamma_2}{\gamma_1+\gamma_2} \ \ \textrm{if}\ \ \gamma_1 \geq \gamma_2,\\
		&\frac{\gamma_1+\gamma_2}{4}\ \  \textrm{if}\ \ \gamma_1 \leq \gamma_2,\end{cases}$ for the graph transform method,
	\item [(Hf2)$_{\textrm{LP}}$] or $M\geq 1$ is arbitrary and the following three inequalities hold with some $\kappa > 0$ \begin{align}
		&  L_f \leq \frac{\gamma_1+\gamma_2}{M} \frac{\kappa}{(M+\kappa)(1+\kappa)},\qquad  L_f < \frac{\gamma_1+\gamma_2}{M} \frac{1}{2(1+\kappa)}, \label{first}\\
	& \frac{M^2L_f}{\gamma_2+\gamma_1 - 2ML_f(1+\kappa)} + \frac{M^2L_f}{\gamma_2+\gamma_1 - ML_f(1+\kappa)} < 1,\label{second} \\
	& ML_f+\frac{M^2L_f^2 (1+\kappa)(1+M)}{\gamma_1 + \gamma_2 - ML_f(1+\kappa)} < \gamma_2.\label{third}
	\end{align}for the Lyapunov--Perron method.
\end{itemize}
	
We will prove that under one of the above assumptions, i.e. if the Lipschitz constant $L_f$ is sufficiently small, then all {\color{black} solutions} are exponentially attracted to a graph of a certain function $\Sigma:E^+\to E^-$, the inertial manifold. We will construct it independently by the two methods: the geometric approach by means of the graph transform (see, for example \cite{mallet}, or the monograph \cite{robinson} for its realization for the dissipative case) and the analytic approach by the Lyapunov--Perron method (see \cite{Henry, car-lan,CLMO-S}). 
  Although the graph transform method works only if $M=1$ in \eqref{eq:hyper}, and the Lyapunov--Perron method works for arbitrary $M \geq 1$ we present both of them, due to the intuitive geometric construction and higher bound present in the graph transform approach.  
\begin{Remark}
	In (Hf2)$_{GT}$ we have presented the explicit bound for $L_f$, while in (Hf2)$_{LP}$ it is given implicitly, and the restrictions involve $\kappa$, the Lipschitz constant of $\Sigma$. Clearly for a given fixed $\kappa > 0$ the value $L_f = 0$ satisfies all bounds, so one can find maximal $L_{f}(\kappa, \gamma_1, \gamma_2, M)$ such that all three bounds hold for $L_f \in [0,L_{f}(\kappa, \gamma_1, \gamma_2, M))$. The next goal is to find $\kappa = \kappa(\gamma_1, \gamma_2, M)$ such that $L_{f}(\kappa(\gamma_1, \gamma_2, M), \gamma_1, \gamma_2, M) = \max_{\kappa > 0} L_{f}(\kappa, \gamma_1, \gamma_2, M)$. performed this optimization for $M=1$ for several values of $\gamma_1, \gamma_2$. The obtained bound for $L_f$ is presented in the following table.   
	\begin{table}[h!]
		\begin{center}
			\begin{tabular}{ |c| c| c| c|}
				\hline
				& $\gamma_1 = 0.1$ & $\gamma_1 = 1$  & $\gamma_1= 10$ \\
				\hline
				$\gamma_2 = 0.1$ & $L_f^{LP} = 0.044, L_f^{GT} = 0.050$ & $L_f^{LP} = 0.084, L_f^{GT} = 0.090$ & $L_f^{LP} = 0.098, L_f^{GT} = 0.099$\\ 
				$\gamma_2 = 1$ &$L_f^{LP} = 0.242, L_f^{GT} = 0.275$ & $L_f^{LP} = 0.441, L_f^{GT} = 0.500$ & $L_f^{LP} = 0.845, L_f^{GT} = 0.909$ \\  
				$\gamma_2 = 10$ & $L_f^{LP} = 2.228, L_f^{GT} = 2.525$ & $L_f^{LP} = 2.427, L_f^{GT} = 2.750$ & $L_f^{LP} = 4.413, L_f^{GT} = 5.000$\\
				\hline    
			\end{tabular}
		\end{center}
	\end{table}

Visibly, the bound of the graph transform method is always higher, which means that it allows to choose the Lipschitz constant with more freedom.  We stress, however, that we expect that all the obtained bounds are still not sharp: according to  \cite{Zelik, romanov} the sharp bound cannot be higher than $\frac{\gamma_1+\gamma_2}{2}$, and this value can be obtained with the use of the energy inequalities. 
	\end{Remark}
We begin with a lemma. 
\begin{Lemma}
	Assume that there exists a Lipschitz function $\Sigma:E^+\to E^-$ and the constant $\delta>0$ such that for every bounded set $B\in \mathcal{B}(X)$ we can find a constant $C(B) > 0$ with 
	\begin{equation}\label{eq:attr_inert}
	\mathrm{dist} (S(t)B, \mathrm{graph}\, \Sigma) \leq C(B) e^{-\delta t}\ \ \textrm{for every}\ \ t\geq 0.
	\end{equation}
	Then $\mathcal{J} =  \mathrm{graph}\, \Sigma$, and the attraction by the unbounded attractor is exponential 
	$$
		\mathrm{dist} (S(t)B,\mathcal{J}) \leq C(B) e^{-\delta t}\ \ \textrm{for every}\ \ t\geq 0.
	$$
\end{Lemma}
\begin{proof}
	By the minimality in Theorem \ref{th:autonomus1} we must have
	$$
	\mathcal{J} \subset \textrm{graph} \Sigma.
	$$
	But, as by Lemma \ref{multi1} the set $\mathcal{J}$ coincides with a graph of a multifunction $\Psi$ function over $E^+$ with nonempty values, it follows that for every $p$ there exists exactly one $q$ such that $p+q = \mathcal{J}$. Hence it must be $\mathcal{J} = \textrm{graph}\, \Sigma$, the unbounded attractor coincides with the graph of the Lipshitz function over $E^+$. In other words, multivalued inertial manifold is single-valued and Lipschitz.   
\end{proof}
\begin{Remark}\label{rem:graph}
	In fact in the next two subsections we prove more than attraction of $S(t)B$ by the graph of $\Sigma$, i.e. we demonstrate, using the two methods, that if $u_0\in B \in \mathcal{B}(X)$, then 
	$$
	\| \Sigma (PS(t)u_0) - (I-P)S(t)u_0 \| \leq C(B)e^{-\delta t}.
	$$
	Clearly, this estimate implies \eqref{eq:attr_inert}. 
\end{Remark}

\begin{Remark}
	We actually do not need $f$ to be Lipschitz with the constant satisfying the bound on given in (Hf2) on the whole space. Indeed, the proof  that uses the graph transform method in Subsection \ref{graph} uses only the lipschitzness of $f$ on $Q$, the bounded in $E^-$ absorbing set (which is included in an infinite strip in $E^+$ given by $\{ \|(I-P)x\|\leq D_2 \}$) and the proof that uses the Lyapunov--Perron method in Subsection \ref{perron} needs only the lipschitzness of $f$ on a certain infinite in $E^+$ strip $\{ \|(I-P)x \| \leq \textrm{CONST} \}$. Hence for both methods it is sufficient if $f$ is Lipschitz with sufficiently small constant on a certain infinite strip in $E^+$ which is bounded in $E^-$.        
\end{Remark}
\subsubsection{Graph transform: case of $M=1$.}\label{graph} The key concept behind the graph transform are cones: we will say that points $u_1 = p_1 + q_1$ and $u_2 = p_2 + q_2$ are in the positive cone with respect to each other if $\|q_1-q_2\| \leq \kappa\|p_1-p_2\|$. If the opposite inequality $\|q_1-q_2\| > \kappa\|p_1-p_2\|$ holds, then the points $u_1, u_2$  are in the negative cone with respect to each other.

\begin{Lemma}\label{lem:cones}
	Assume, in addition to assumptions of Section \ref{sec:setup} that (Hf2)$_{\textrm{GT}\cup \textrm{LP}}$ holds with $$L_f < \frac{\kappa}{(1+\kappa)^2}(\gamma_1 + \gamma_2) \ \ \textrm{with a constant}\ \ \kappa > 0.$$ If $u_1, u_2 \in Q$ are in the positive cone with respect to each other, then $S(t)u_1, S(t)u_2$ are also in the positive cone with respect to each other. 
	\end{Lemma}
\begin{proof}
	We will denote $p_1(s) = PS(s)u_1$, $p_2(s) = PS(s)u_2$, $q_1(s) = (I-P)S(s)u_1$, $q_2(s) = (I-P)S(s)u_2$
	Assume, for contradiction, that $S(t)u_1, S(t)u_2$ are in the negative cone with respect to each other. Then there exists $\tau \in [0,t)$ such that $\|q_1(\tau) - q_2(\tau)\| = \kappa\|p_1(\tau) - p_2(\tau)\|$ and $\|q_1(s) - q_2(s)\| >  \kappa \|p_1(s) - p_2(s)\|$ for $s\in (\tau,t]$. We deduce from \eqref{solp(t)} and \eqref{solq(t)} that 	
	\begin{subequations}\label{def:pq2a}
		\begin{align}
			p_1(s)-p_2(s)&=e^{A(s-\tau)}(p_1(\tau)-p_2(\tau)) +\int_\tau^s e^{A(s-r)}P(f(p_1(r)+q_1(r))-f(p_2(r)+q_2(r)))dr,\label{15a}\\
			q_1(s)-q_2(s)&=e^{A(s-\tau)}(q_1(\tau)-q_2(\tau)) +\int_{\tau}^se^{A(s-r)}(I-P)(f(p_1(r)+q_1(r))-f(p_2(r)+q_2(r)))dr. 
		\end{align}
	\end{subequations}
for every $s\in (\tau,t]$. We need to estimate the difference of $q$ from above and difference of $p$ from below. Using \eqref{eq:hyper} we directly obtain
$$
\|q_1(s)-q_2(s)\|\leq e^{-\gamma_2(s-\tau)}\|q_1(\tau)-q_2(\tau)\| +L_f\int_{\tau}^s e^{-\gamma_2(s-r)}(\|p_1(r)-p_2(r)\| + \|q_1(r)-q_2(r)\|)dr.
$$
Moreover, \eqref{15a} implies that
$$	p_1(\tau)-p_2(\tau)=e^{A(\tau-s)}(p_1(s)-p_2(s)) +\int_s^\tau e^{A(\tau-r)}P(f(p_1(r)+q_1(r))-f(p_2(r)+q_2(r)))dr,
$$
whence, by \eqref{eq:hyper},
$$	\|p_1(\tau)-p_2(\tau)\|\leq e^{\gamma_1(\tau-s)}\|p_1(s)-p_2(s)\| +L_f\int_\tau^s e^{\gamma_1(\tau-r)}(\|p_1(r)-p_2(r)\|+\|q_1(r)-q_2(r)\|)dr.
$$
This means that
$$	e^{\gamma_1(s-\tau)}\|p_1(\tau)-p_2(\tau)\|-L_f\int_\tau^s e^{\gamma_1(s-r)}(\|p_1(r)-p_2(r)\|+\|q_1(r)-q_2(r)\|)dr\leq \|p_1(s)-p_2(s)\| .
$$
Combining the two bounds yields
\begin{align*}
& e^{\gamma_1(s-\tau)}\kappa \|p_1(\tau)-p_2(\tau)\|-L_f\kappa \int_\tau^s e^{\gamma_1(s-r)}(\|p_1(r)-p_2(r)\|+\|q_1(r)-q_2(r)\|dr\\
& < e^{-\gamma_2(s-\tau)}\|q_1(\tau)-q_2(\tau)\| +L_f\int_{\tau}^s e^{-\gamma_2(s-r)}(\|p_1(r)-p_2(r)\| + \|q_1(r)-q_2(r)\|)dr.
\end{align*}
Using the fact that $\|q_1(\tau) - q_2(\tau)\| = \kappa \|p_1(\tau) - p_2(\tau)\|$ and $\|q_1(s) - q_2(s)\| > \kappa \|p_1(s) - p_2(s)\|$ for $s\in (\tau,t]$ we deduce
$$
\left(e^{\gamma_1(s-\tau)}-e^{-\gamma_2(s-\tau)}\right) \|q_1(\tau)-q_2(\tau)\| <  L_f\int_\tau^s (\kappa e^{\gamma_1(s-r)}+e^{-\gamma_2(s-r)})\left(1+\frac{1}{\kappa}\right)\|q_1(r)-q_2(r)\|dr.
$$
Let $s=\tau+h$ for $h\in [0,t-\tau)$. Then
$$
\left(e^{\gamma_1h}-e^{-\gamma_2h}\right) \|q_1(\tau)-q_2(\tau)\| <  \left(1+\frac{1}{\kappa}\right)L_f\int_\tau^{\tau+h} (\kappa e^{\gamma_1(\tau+h-r)}+e^{-\gamma_2(\tau+h-r)})\|q_1(r)-q_2(r)\|dr.
$$
Using the mean value theorem for integrals we deduce that there exists $\theta(h) \in (0,1)$ such that
$$
\frac{e^{\gamma_1h}-e^{-\gamma_2h}}{h} \|q_1(\tau)-q_2(\tau)\| <  \left(1+\frac{1}{\kappa}\right)L_f   (\kappa e^{\gamma_1((1-\theta)h)}+e^{-\gamma_2((1-\theta)h)})\|q_1(\tau+\theta h)-q_2(\tau+\theta h)\|.
$$
Passing with $h$ to zero yields
$$
(\gamma_1 + \gamma_2) \|q_1(\tau)-q_2(\tau)\| \leq  \frac{(1+\kappa)^2}{\kappa}L_f   \|q_1(\tau)-q_2(\tau)\|,
$$
a contradiction since we must have $q_1(\tau) \neq q_2(\tau)$. 
\end{proof}
\begin{Remark}
	The fact that $M=1$ is required only in the above lemma which establishes the invariance of the positive cones. The rest of the argument on this section does not require this restriction on $M$. Hence, if it would be possible to obtain the invariance of the positive cone for $M>1$, the graph trasform argument would work for the general $M$. For an ODE case this could be obtain by introducing the appropriate change of coordinates both in (then finite dimensional) spaces $E^+$ and $E^-$. For our case such change of coordinates is possible only in the finite dimensional space $E^+$. Hence we leave the question if it is possible to adapt the graph transform argument to the case of general $M$, for now, open. 
\end{Remark}

Now, choose $\Sigma_0:E^+\to E^-$, the lipshitz function with the constant $\kappa$ and denote the graph of $\Sigma_0$ as $X\supset \textrm{graph}\, \Sigma_0=\{ p+\Sigma_0(p)\,:\  p\in E^+ \}$.

\begin{Lemma}
	Assume that $\Sigma_0:E^+\to E^-$ is a continuous function with $\mathrm{graph}\, \Sigma_0 \subset Q$, the absorbing and positively invariant set bounded in $E^-$. For every $t>0$ and for every $p\in E^+$ there exists $q\in E^-$ such that $p+q \in S(t)(\mathrm{graph}\, \Sigma_0)$.
	\end{Lemma}
	\begin{proof}
		The proof follows the lines of the argument in Lemma \ref{lemma:211}. Choose $p\in E^+$ and $t>0$ and consider the continuous mapping $H(\theta,p) = PS(\theta t)(p+\Sigma_0(p))$. Then $H(0,p) = p$. Now we will construct the appropriate set on which we will consider the degree. Since  (H2) holds, there exists $R\geq R_0$ such that $\{x\in E^+ \,:\ \|x\|\leq \|p\|+1 \} \subset H_R$. Define $B = \{ x\in E^+\, :\ \|x\| < R + 1\}$, and since, by (H2) again, $p \in B$ it follows that $\textrm{deg}(H(0,\cdot),B,p) = 1$. As, by (H2), $\partial B = \{ x\in E^+\, :\ \|x\| = R + 1 \} \subset Q\setminus H_R$ we deduce that $p \not\in H(\theta,\partial B)$ for every $\theta \in [0,1]$. Hence, by the homotopy invariance of the Brouwer degree,  we deduce that $\textrm{deg}(PS( t)(p+\Sigma _0(p)),B,p) = 1$, whence there exists $ \overline{p}\in B$ such that $PS(t)(\overline{p}+\Sigma_0(\overline{p})) = p$. It is enough to take $q = (I-P))PS(t)(\overline{p}+\Sigma_0(\overline{p}))$ to get the assertion of the lemma. 
	\end{proof}

 Previous two lemmas imply that if $\Sigma_0$ is a $\kappa$-lipschitz function with a graph in $Q$, then for every $t$ there exists the $\kappa$-lipschitz function $\Sigma_t:E^+\to E^-$ such that
$$
S(t)(\textrm{graph}\, \Sigma_0) = \{ p+\Sigma_t(p)\,:\ p\in E^+ \} = \textrm{graph}\, \Sigma_t. 
$$

\begin{Lemma}
	Assume that $\Sigma_0:E^+\to E^-$ is $\kappa$-Lipschitz with $\textrm{graph}\, \Sigma_0 \subset Q$, the absorbing and positively invariant set bounded in $E^-$. Moreover assume that 
	$$
	L_f\left(1+\frac{1}{\kappa}\right) < \gamma_2. 
	$$ Then there exist constants $\delta>0$ and $C>0$ such that for every $p\in E^+$ and every $t_1 > t_2 >0$ we have
	 $$
	\|\Sigma_{t_1}(p) - \Sigma_{t_2}(p) \| \leq Ce^{-\delta t_2}.
	$$
	In consequence there exists a $\kappa$-Lipschitz function $\Sigma:E^+\to E^-$ such that $\Sigma_t(p) \to \Sigma(p)$ as $t\to \infty$ for every $p\in E^+$. 
	\end{Lemma}
\begin{proof}
	There exists $p_1, p_2 \in E^+$  such that
	$(p+\Sigma_{t_1}(p)) = S(t_2)(p_1+\Sigma_{t_1-t_2}(p_1))$ and $(p+\Sigma_{t_2}(p)) = S(t_2)(p_2+\Sigma_{0}(p_2))$.
	Assume that $\Sigma_{t_1}(p) \neq \Sigma_{t_2}(p)$. This means that points $(p,\Sigma_{t_1}(p))$ and $(p,\Sigma_{t_2}(p))$ are in the negative cone with respect to each other and hence, by Lemma \ref{lem:cones} for every $t\in[0,t_2]$ we have
	$$
	\| (I-P)(S(t)(p_1+\Sigma_{t_1-t_2}(p_1)) - S(t)(p_2+\Sigma_{0}(p_2))) \| > \kappa \|P(S(t)(p_1+\Sigma_{t_1-t_2}(p_1)) - S(t)(p_2+\Sigma_{0}(p_2)))\|.
	$$ 
	From the Duhamel formula \eqref{solq(t)} we obtain 
	\begin{align*}
& 	(I-P)(S(t)(p_1+\Sigma_{t_1-t_2}(p_1))-(S(t)(p_2+\Sigma_{0}(p_2)))) =e^{At}(\Sigma_{t_1-t_2}(p_1)-\Sigma_0(p_2)) \\
& \ \ \  +\int_{0}^{t}e^{A(t-r)}(I-P)(f(S(t)(p_1+\Sigma_{t_1-t_2}(p_1))-f(S(t)(p_2+\Sigma_{0}(p_2))))dr
	\end{align*}
	It follows that
	\begin{align*}
	& 	\|(I-P)(S(t)(p_1+\Sigma_{t_1-t_2}(p_1))-(S(t)(p_2+\Sigma_{0}(p_2))))\|\leq e^{-\gamma_2 t_2}\|\Sigma_{t_1-t_2}(p_1)-\Sigma_0(p_2)\| \\
	& \ \ \  +L_f\left(1+\frac{1}{\kappa}\right) \int_{0}^{t_2}e^{-\gamma_2 (t_2-r)}\|(I-P)(S(t)(p_1+\Sigma_{t_1-t_2}(p_1))-(S(t)(p_2+\Sigma_{0}(p_2))))\|dr.
	\end{align*}
	We are in position to use the Gronwall lemma, whence
	 $$
	 \|(I-P)(S(t)(p_1+\Sigma_{t_1-t_2}(p_1))-(S(t)(p_2+\Sigma_{0}(p_2))))\| \leq \|\Sigma_{t_1-t_2}(p_1)-\Sigma_0(p_2)\|e^{L_f\left(1+\frac{1}{\kappa}\right) t},
	 $$
	 for $t\in [0,t_2]$.
	 In consequence 
	 $$
	 \|\Sigma_{t_1}(p) - \Sigma_{t_2}(p) \| \leq \|\Sigma_{t_1-t_2}(p_1)-\Sigma_0(p_2)\|e^{\left(-\gamma_2 + L_f\left(1+\frac{1}{\kappa}\right)\right) t_2},
	 $$
	 which yields the assertion. 
\end{proof}
We denote the graph of $\Sigma$ by $\textrm{graph}\, \Sigma = \{ p+\Sigma(p)\, :\ p\in E^+ \}$. 

\begin{Lemma}
	Let $B\in \mathcal{B}(X)$ be a bounded set. There exists a constant $C(B)$ and $\delta>0$ such that
	$$
	\textrm{dist}(S(t)B, \mathrm{graph}\, \Sigma) \leq C(B) e^{-\delta t}.
	$$
	\end{Lemma}
\begin{proof}
	Without loss of generality we may assume that $B\subset Q$. Choose $u_0 = p_0+q_0 \in B$. Denote $p(r) = PS(r)u_0$ and $q(r) = (I-P)S(r)u_0$ for $r\geq 0$.  Now fix $t\geq 0$ and let $s\geq t$. We need to estimate the difference
	$
	\Sigma_s(p(t)) - q(t). 
	$
	To this end observe that there exists a point $p_1 + \Sigma_{s-t}(p_1) \in \textrm{graph}\, \Sigma_{s-t}$ such that $S(t)(p_1 + \Sigma_{s-t}(p_1)) = p(t) + \Sigma_s(p(t))$. Then either $\Sigma_s(p(t)) = q(t)$, or, if it is not the case, then points $p(t) + \Sigma_s(p(t))$ and $p(t) + q(t)$ are at their negative cones and hence
	$$
	\| (I-P)S(r)(p_1 + \Sigma_{s-t}(p_1)) - q(r)\| > \kappa \| P S(r)(p_1 + \Sigma_{s-t}(p_1)) - p(r) \|\ \ \textrm{for}\ \ r\in [0,t].
	$$
	We use the Duhamel formula \eqref{solq(t)}, whence, proceeding analogously to the argument in the proof of the previous Lemma, we obtain
\begin{align*}
	&\|(I-P)S(r)(p_1 + \Sigma_{s-t}(p_1)) - q(r)\|\\ & 
\ \ \leq e^{-\gamma_2 r} \|\Sigma_{s-t}(p_1) - q_0\| + L_f\left(1+\frac{1}{\kappa}\right)\int_0^r e^{-\gamma_2 (r-\tau)}\|(I-P)S(\tau)(p_1 + \Sigma_{s-t}(p_1)) - q(\tau)\|\, d\tau. 
	\end{align*}
	Now the Gronwall lemma implies that 
	$$
	\|(I-P)S(r)(p_1 + \Sigma_{s-t}(p_1)) - q(r)\| \leq e^{\left(-\gamma_2 + L_f\left(1+\frac{1}{\kappa}\right)\right)r}\|\Sigma_{s-t}(p_1) - q_0\|\ \ \textrm{for every}\ \ r\in [0,t].
	$$
	In particular 
	$$
	\|\Sigma_s(p(t)) - q(r)\| \leq e^{\left(-\gamma_2 + L_f\left(1+\frac{1}{\kappa}\right)\right)t}\|\Sigma_{s-t}(p_1) - q_0\| \leq C(B) e^{-\delta t}.
	$$
	We can pass with $s$ to infinity, whence
	$$
	\|\Sigma(p(t)) - q(r)\| \leq  C(B) e^{-\delta t},
	$$
	which implies the assertion. 
\end{proof}

\begin{Remark}
We have two bounds on $L_f$. One that follows from the condition on the invariance of the cones
$$
L_f < \frac{\kappa}{(1+\kappa)^2}(\gamma_1+\gamma_2),
$$ 
and the second one from the condition of the attraction of bounded sets by the limit graph
$$
L_f < \frac{\kappa}{1+\kappa}\gamma_2. 
$$
Hence we need
$$
L_f < \min\left\{ \frac{\kappa}{1+\kappa}\gamma_2,  \frac{\kappa}{(1+\kappa)^2}(\gamma_1+\gamma_2) \right\}.
$$
A simple calculation shows, that if $\kappa \geq \frac{\gamma_1}{\gamma_2}$ then $\frac{\kappa}{1+\kappa}\gamma_2 \geq \frac{\kappa}{(1+\kappa)^2}(\gamma_1+\gamma_2)$ and if $\kappa \leq \frac{\gamma_1}{\gamma_2}$ then $\frac{\kappa}{1+\kappa}\gamma_2 \leq \frac{\kappa}{(1+\kappa)^2}(\gamma_1+\gamma_2)$. It is then enough to maximize $\frac{\kappa}{(1+\kappa)^2}(\gamma_1+\gamma_2)$ over the interval $\kappa \in [\frac{\gamma_1}{\gamma_2},\infty)$ and
 $\frac{\kappa}{1+\kappa}\gamma_2$ over the interval $\kappa \in [0,\frac{\gamma_1}{\gamma_2}]$. Now
 $$
 \max_{\kappa \in [0,\frac{\gamma_1}{\gamma_2}]} \frac{\kappa}{1+\kappa}\gamma_2 = \frac{\gamma_1\gamma_2}{\gamma_1+\gamma_2}, 
 $$
 and 
 $$
 \max_{\kappa \in [\frac{\gamma_1}{\gamma_2},\infty)} \frac{\kappa}{(1+\kappa)^2}(\gamma_1+\gamma_2) = \begin{cases}& \frac{\gamma_1\gamma_2}{\gamma_1+\gamma_2} \ \ \textrm{if}\ \ \frac{\gamma_1}{\gamma_2} \geq 1,\\
 &\frac{\gamma_1+\gamma_2}{4}\ \  \textrm{if}\ \ \frac{\gamma_1}{\gamma_2} \leq 1.
 \end{cases} 
 $$
 This means that
 $$
 \max_{\kappa\in [0,\infty)} \min\left\{ \frac{\kappa}{1+\kappa}\gamma_2,  \frac{\kappa}{(1+\kappa)^2}(\gamma_1+\gamma_2) \right\} = \begin{cases}& \frac{\gamma_1\gamma_2}{\gamma_1+\gamma_2} \ \ \textrm{if}\ \ \frac{\gamma_1}{\gamma_2} \geq 1,\\
 &\frac{\gamma_1+\gamma_2}{4}\ \  \textrm{if}\ \ \frac{\gamma_1}{\gamma_2} \leq 1.
 \end{cases} 
 $$
 This motivates (Hf2)$_{\textrm{GT}}$. In other words, if $\gamma_1\geq \gamma_2$, the optimal upper bound on $L_f$ is given by the half of the harmonic mean of these numbers, and in the opposite case it is given by the half of their arithmetic mean.  
\end{Remark}
\subsubsection{Lyapunov--Perron method: case of general $M\geq 1$.}\label{perron}
The arguments to obtain the results of this subsection very closely follow the lines of the proof of \cite[Theorem 2.1]{CLMO-S} where they are done in a genaral, non-autonomous   setup for the case or arbitrary gap, not necessarily the gap at which the real parts of the eigenvalues change sign. We repeat some parts of the arguments after \cite[Theorem 2.1]{CLMO-S} only for the sake of exposition completeness, and to make clear how the limitations on $L_f$ in (Hf2)$_{\textrm{LP}}$ enter the argument.  We will consider the 
metric space 
\begin{align*}
& \mathcal{LB}(\kappa) = \left\{ \Sigma\in C(\mathbb{R}\times E^+; E^-)\,:\ \sup_{t\in \mathbb{R}}\|\Sigma(t,p_1)-\Sigma(t,p_2)\|\leq \kappa \|p_1-p_2\|,\right.\\
& \qquad  \left.\Sigma(t,0) = 0\ \textrm{for every}\ t\in \mathbb{R}\ \ \textrm{and}\ \ \|\Sigma(t,p)\|\leq \frac{2MC_f}{\gamma_2}\ \textrm{for every}\ t\in \R, p\in E^+ \right\},
\end{align*}
equipped with the metric $\|\Sigma_1-\Sigma_2\|_{\mathcal{LB}} = \sup_{t\in\mathbb{R}}\sup_{p\in P, p\neq 0}\frac{\|\Sigma_1(t,p) - \Sigma_2(t,p)\|}{\|p\|}$.  

We need to translate the {\color{black} solution} of the problem in order to guarantee that the nonlinearity is zero for a zero argument. We know that $\mathcal{J}_b$, the family of global and bounded {\color{black} solutions} in the unbounded attractor is nonempty, let $\{ \overline{u}(t)\}_{t\in \mathbb{R}} \in \mathcal{J}_b$. Then, if the function $u(t)$ is the {\color{black} solution} of the original problem, the function $v(t) = u(t)-\overline{u}(t)$ solves the translated problem
governed by the equation
\begin{align}\label{eq:v}
v'(t) = Av(t) + f(v(t) + \overline{u}(t))-f(\overline{u}(t)).
\end{align}
Denoting $g(t,v) = f(v + \overline{u}(t))-f(\overline{u}(t))$ observe that $g(t,0) = 0$ and $g(t,\cdot)$ is Lipschitz with the same constant as $f$. Note that the problem is now non-autonomous, if we assume that the equation $-A\overline{u} = f(\overline{u})$ has a {\color{black} solution}, then we could translate $u(t)$ by this {\color{black} solution} and obtain the autonomous translated problem, which would make the argument simpler. We proceed without this assumption. Using the representation $Pv(t) = p(t)$ and $(I-P)v(t) = q(t)$ we have the following Duhamel formulas 
\begin{subequations}\label{def:pq2}
	\begin{align}
		p(t)&=e^{A(t-\tau)}p(\tau) +\int_\tau^t e^{A(t-s)}Pg(s,p(s)+q(s))ds\ \ \textrm{for}\ \ t\geqslant \tau,\label{solp(t)2}\\
		q(t)&=e^{A(t-\tau)}q(\tau) +\int_{\tau}^te^{A(t-s)}(I-P)g(s,p(s)+q(s))ds\ \ \textrm{for}\ \ t\geqslant \tau. \label{solq(t)2}
	\end{align}
\end{subequations}
For a given $\Sigma \in \mathcal{LB}(\kappa)$ and $p_0\in E^+$ it is possible to solve backwards in time the following ODE
\begin{equation}\label{eq:backward}
p'(t) = Ap(t) + g(t,p(t)+\Sigma(t,p(t)))\ \textrm{for}\ t\in (-\infty,\tau]\ \textrm{with}\ p(\tau) = p_0, 
\end{equation}
and the {\color{black} solution} satisfies \eqref{solp(t)2} with $q(s) = \Sigma(s,p(s))$. We define the following Lyapunov--Perron function, $G$, which, as we will show in the next lemma, is a contraction on $\mathcal{LB}(\kappa)$ for sufficiently small $L_f$.
$$
G(\Sigma)(\tau,p_0) = \int_{-\infty}^\tau e^{A(\tau-s)} (I-P)g(s,p(s)+\Sigma(s,p(s))\, ds.
$$
\begin{Lemma}\label{lem:25}
		Assume (Hf2)$_{\textrm{GT}}$ in addition to assumptions of Section \ref{sec:setup}. The mapping $G:\mathcal{LB}(\kappa)\to \mathcal{LB}(\kappa)$ is a contraction and hence has a unique fixed point $\Sigma^*\in \mathcal{LB}(\kappa)$.
\end{Lemma}
\begin{proof}
For a given $p_0\in E^+$ and $\Sigma\in \mathcal{LB}(\kappa)$  we observe, that since  $\|g(s,u)\| \leq 2 C_f$ for every $s\in \R$ and $u\in X$, then 
	$$
	\|G(\Sigma)(\tau,p_0)\| \leq  \frac{2M C_f}{\gamma_2}\ \ \textrm{for every}\ \ \tau\in \R.
	$$
We now estimate the growth of $p$ in \eqref{eq:backward}. Fix $\tau \in \R$.  From \eqref{solp(t)2} and from \eqref{eq:hyper} we obtain for $t\leq \tau$
$$
\|p(t)\| \leq Me^{\gamma_1(t-\tau)}\|p_0\| + ML_f(1+\kappa)\int_{t}^\tau e^{\gamma_1(t-s)} \|p(s)\|\, ds.  
$$
The Gronwall lemma implies that
\begin{equation}\label{pgrowth}
\|p(t)\| \leq M  e^{(ML_f(1+\kappa)-\gamma_1)(\tau-t)} \|p_0\|.
\end{equation}
Take $p_0^1, p_0^2 \in E^+$ and $\Sigma_1, \Sigma_2 \in \mathcal{LB}(\kappa)$. We denote the {\color{black} solution} of \eqref{eq:backward} corresponding to $p_0^1$ taken at time $\tau$ and $\Sigma_1$ by $p^1$, and the one corresponding to $p_0^2$ taken at time $\tau$ and $\Sigma_2$ by $p^2$. Then 
\begin{align*}
	& \|G(\Sigma_1)(\tau,p_0^1) - G(\Sigma_2)(\tau,p_0^2)\| \leq M\int_{-\infty}^\tau e^{\gamma_2 (s-\tau)}L_f (\|p^1(s)-p^2(s)\| + \|\Sigma_1(p^1(s))-\Sigma_2(p^2(s)\|)\, ds \\
	& \ \ \ \leq M L_f \left((1+\kappa)\int_{-\infty}^\tau e^{\gamma_2 (s-\tau)} \|p^1(s)-p^2(s)\|\, ds\right.  +\\
	& \qquad \qquad \left.M \|\Sigma_1 - \Sigma_2\|_{\mathcal{LB}} \|p_0^2\| \int_{-\infty}^\tau  e^{(ML_f(1+\kappa)-\gamma_1-\gamma_2)(\tau-s)} \, ds\right) \\
	& \ \ \ \leq M^2 L_f\|\Sigma_1 - \Sigma_2\|_{\mathcal{LB}}\|p_0^2\| \int_{-\infty}^\tau  e^{(ML_f(1+\kappa)-\gamma_1-\gamma_2)(\tau-s)} \, ds\\
	& \qquad \qquad  + ML_f(1+\kappa) \int_{-\infty}^\tau e^{\gamma_2 (s-\tau)} \|p^1(s)-p^2(s)\| \, ds . 
\end{align*}
Note from the second bound in \eqref{first} we have $2ML_f(1+\kappa) < \gamma_1+\gamma_2$. It follows that 
\begin{align}
	& \|G(\Sigma_1)(\tau,p_0^1) - G(\Sigma_2)(\tau,p_0^2)\|\label{eq:g} \\
	& \ \ \ \leq \frac{M^2 L_f \|p_0^2\|}{\gamma_1+\gamma_2-ML_f(1+\kappa)}\|\Sigma_1 - \Sigma_2\|_{\mathcal{LB}} + ML_f(1+\kappa) \int_{-\infty}^\tau e^{\gamma_2 (s-\tau)} \|p^1(s)-p^2(s)\| \, ds .\nonumber 
\end{align}
On the other hand, from \eqref{solp(t)2} and \eqref{eq:hyper} we deduce that 
\begin{align*}
	& \|p^1(s)-p^2(s)\|\leq M e^{\gamma_1(s-\tau)}\|p^1_0-p_0^2\| \\ & \qquad +M\int_s^\tau e^{\gamma_1(s-r)}(\|g(r,p^1(r)+\Sigma_1(r,p^1(r)))-g(r,p^2(r)+\Sigma_2(r,p^2(r)))\|dr,
\end{align*}
whence
\begin{align*}
& 	\|p^1(s)-p^2(s)\|\leq M e^{\gamma_1(s-\tau)}\|p^1_0-p_0^2\| +ML_f\int_s^\tau e^{\gamma_1(s-r)}((1+\kappa)\|p^1(r)-p^2(r)\|\\
& \qquad +\|\Sigma_1-\Sigma_2\|_{\mathcal{LB}(\kappa)}\|p^2(r)\|)dr.
\end{align*}
After a straightforward calculation which uses \eqref{pgrowth} it follows that
\begin{align*}
	& \|p^1(s)-p^2(s)\|\leq M e^{\gamma_1(s-\tau)}\|p^1_0-p_0^2\| +\frac{M}{1+\kappa}\|\Sigma_1-\Sigma_2\|_{\mathcal{LB} }\|p_0^2\| e^{(ML_f(1+\kappa)-\gamma_1)(\tau-s)} \\
	& \ \ \ +  ML_f(1+\kappa)\int_s^\tau  e^{\gamma_1(s-r)}\|p^1(r)-p^2(r)\|dr.
\end{align*}	
The Gronwall inequality implies 
\begin{align*}
	& \|p^1(s)-p^2(s)\|\leq M e^{(ML_f(1+\kappa)-\gamma_1)(\tau-s)}\|p^1_0-p_0^2\| +\frac{M}{1+\kappa}\|\Sigma_1-\Sigma_2\|_{\mathcal{LB} }\|p_0^2\| e^{(2ML_f(1+\kappa)-\gamma_1)(\tau-s)}.
\end{align*}	
Substituting this last inequality to \eqref{eq:g} we deduce
\begin{align*}
	& \|G(\Sigma_1)(\tau,p_0^1) - G(\Sigma_2)(\tau,p_0^2)\| \leq \frac{M^2L_f(1+\kappa)}{\gamma_1+\gamma_2 - ML_f(1+\kappa)} \|p^1_0-p_0^2\|\\
	& \ \ \  +  \left(\frac{M^2L_f}{\gamma_2+\gamma_1 - 2ML_f(1+\kappa)} + \frac{M^2L_f}{\gamma_2+\gamma_1 - ML_f(1+\kappa)}\right)\|\Sigma_1-\Sigma_2\|_{\mathcal{LB} }\|p_0^2\|.\nonumber 
\end{align*}
Hence, we need 
$$\frac{M^2L_f(1+\kappa)}{\gamma_1+\gamma_2 - ML_f(1+\kappa)} \leq \kappa, \frac{M^2L_f}{\gamma_2+\gamma_1 - 2ML_f(1+\kappa)} + \frac{M^2L_f}{\gamma_2+\gamma_1 - ML_f(1+\kappa)} < 1,
$$
the first inequality is needed for the Lipschitz condition on $G(\Sigma)$ with a constant $\kappa$ and the second one for the contraction on $G$. The second inequality is exactly \eqref{second}. We can rewrite the first condition as 
$$
L_f \leq \frac{\gamma_1+\gamma_2}{M} \frac{\kappa}{(M+\kappa)(1+\kappa)},
$$
the first bound in \eqref{first}, and the proof is complete. 
\end{proof}

\begin{Remark}\label{rem:kappa}
	Different to the framework of unbounded attractors, in the general framework of inertial manifolds, where we have freedom to choose the gap, the last restriction in  (Hf2)$_{\textrm{LP}}$, namely \eqref{third}, needed for the exponential attraction by the Lipschitz graph, in contrast to \eqref{first} and \eqref{second}, is usually easy to obtain by taking the sufficiently high eigenvalue at the gap. Hence it makes sense to compare the bound on $L_f$ which follows only from \eqref{first} and \eqref{second} between the graph transform and Lyapunov--Perron methods. As in \cite{CLMO-S}, the quantity on the left hand side of \eqref{second}  is estimated as follows
\begin{equation}\label{eq:estnonsharp}
 \frac{M^2L_f}{\gamma_2+\gamma_1 - 2ML_f(1+\kappa)} + \frac{M^2L_f}{\gamma_2+\gamma_1 - ML_f(1+\kappa)} \leq \frac{2M^2L_f}{\gamma_2+\gamma_1 - 2ML_f(1+\kappa)}.
\end{equation}
 Hence, the sufficient condition for \eqref{second} is
 $$
 {2M^2L_f} < \gamma_2+\gamma_1 - 2ML_f(1+\kappa),
 $$
 which is equivalent to
 $$
  L_f < \frac{\gamma_1+\gamma_2}{M}\frac{1}{2(M+1+\kappa)}.
 $$
 We need to take the lowest of this bound and the ones that follow from \eqref{first}, i.e. 
 \begin{align*}
 & L_f < \frac{\gamma_1+\gamma_2}{M} \min\left\{\frac{1}{2(M+1+\kappa)},  \frac{\kappa}{(M+\kappa)(1+\kappa)},  \frac{1}{2(1+\kappa)}\right\} \\
 & \qquad = \frac{\gamma_1+\gamma_2}{M} \min\left\{\frac{1}{2(M+1+\kappa)},  \frac{\kappa}{(M+\kappa)(1+\kappa)}\right\}.
 \end{align*}
 Comparison of the two functions of $\kappa$ reveals that there exists $\kappa_0(M)$ such that for $\kappa\in (0,\kappa_0)$ the second bound is sharper, while for $\kappa>\kappa_0$ the first one is sharper. Moreover the first bound is a strictly decreasing function of $\kappa$. If $\kappa=\kappa_0$ both bounds coincide. Moreover, the $\kappa$ at which the second bound achieves maximum is always larger than $\kappa_0$. Hence we need to take $\kappa = \kappa_0$ for the optimal bound. This means that the optimal choice of $\kappa$ from the point of view of \eqref{first} and \eqref{second} (with the non-sharp bound \eqref{eq:estnonsharp})) is
 $$
 \kappa = \frac{2M}{M+1+\sqrt{(M+1)^2+4M}},
 $$

	We have obtained the restriction on $L_f$ of type $G(M) L_f < \gamma_1 + \gamma_2$. The similar restriction $\max\{M^2+2M+\sqrt{8M^3},3M^2+2M\} L_f < \gamma_1+\gamma_2$  is obtained in \cite{CLMO-S} but our bound, although the path to obtain it is exactly the same as in  \cite{CLMO-S}, is a little bit sharper due to the optimization with respect to $\kappa$, as the following table shows (note, however, that improved bounds are given in \cite[Remark 2.4]{CLMO-S}). The numbers in the table are the lower bounds for $\frac{\gamma_1+\gamma_2}{L_f}$. Note that for the graph transform, the best possible bound in Lemma \ref{lem:cones} is obtained for $\kappa=1$. 
	We expect that all presented bounds are still not sharp: the last column gives the sharp bound for $M=1$ obtained in \cite{Zelik, romanov} with the use of the energy inequalities.  
	\begin{table}[h!]
		\begin{center}
			\begin{tabular}{ |c| c| c| c|c|}
				\hline
				 & Bound of \cite{CLMO-S} & Remark \ref{rem:kappa}  & Graph transform & Sharp bound of \cite{Zelik, romanov}\\
				\hline
				$M=1$ & $5.829 $ & $4.829$ & $4 $ & $2 $\\ 
				$M=2$ & $16 $ & $14.247$ & - & -\\  
				$M=4$ & $56 $ & $45.613 $ & - & -\\
				large $M$ & $(3M^2 + o(M^2)) $ & $(2M^2 + o(M^2)) $ & - & -\\
				\hline    
			\end{tabular}
		\end{center}
	\end{table}

This reveals, that while Lyapunov--Perron method with the chosen metrics $\mathcal{LB}(\kappa)$ can be used for any $M\geq 1$ it gives the bounds which are tentatively nonoptimal, and the graph transform method, which needs the invariance of the cones and hence can be used  only for $M=1$, gives the better bound. Note that in one step of the proof, namely in \eqref{eq:estnonsharp} we used the rough estimate $\gamma_2+\gamma_1 - ML_f(1+\kappa) > \gamma_2+\gamma_1 - 2ML_f(1+\kappa)$. The  numerical calculation of the bound for $M=1$ without
 this estimate leads to the optimal value $\kappa = 1/2$ and the bound $4.5 L_f < \gamma_1+\gamma_2$ which is better than the one in the above table but still more rough than the one from the graph transform.
	\end{Remark}

Now each point $w = Pw + (I-P)w$ in the phase space lies 'over' some point in the graph of $\Sigma^*(t,\cdot)$, this point is given by $Pw + \Sigma^*(t,Pw)$. So, for the {\color{black} solution} $v(t)$ (we denote $p(t) = Pv(t)$ and $q(t) = (I-P)v(t)$) one can define its vertical distance in $E^-$ from $\Sigma^*(t,\cdot)$ as 
$\|\xi(t)\|$ where
\begin{equation}\label{eq:xit}
\xi(t) = (I-P)v(t) - \Sigma^*(t,Pv(t)).
\end{equation}
We prove that this distance exponentially tends to zero as the difference between the initial and final time tends to infty, uniformly on bounded sets of initial data. Here, since the argument formula by formula exactly follows the lines of the corresponding proof of \cite[Theorem 2.1]{CLMO-S}, we only state the result.
\begin{Lemma}\label{lem:26}
	Assume (Hf2)$_{\textrm{GT}}$ in addition to assumptions of Section \ref{sec:setup} and let $v:[t_0,\infty)\to X$ be the {\color{black} solution} of \eqref{eq:v}. 
	If $\xi(t)$ is given by \eqref{eq:xit}, then 
	$$
	\|\xi(t)\| \leq M\|\xi(\tau)\| e^{(t-\tau)\left( -\gamma_2 + ML_f+\frac{M^2L_f^2 (1+\kappa)(1+M)}{\gamma_1 + \gamma_2 - ML_f(1+\kappa)}\right)} \ \ \textrm{for}\ \ t\geq \tau \geq t_0.
	$$
\end{Lemma}
We summarize the results of this section in a theorem
\begin{Theorem}
		Assume (Hf2)$_{\textrm{GT}}$ in addition to assumptions of Section \ref{sec:setup}. There exists a constant $\delta > 0$ and a Lipschitz fuction $\Sigma:E^+ \to E^-$ such that for every $u_0 \in X$ 
		$$
		\|(I-P)S(s)u_0 -\Sigma(PS(s)u_0)\| \leq M\left(\|u_0\|+ \frac{3MC_f}{\gamma_2}\right) e^{-\delta s}\ \ \textrm{for}\ \ s\geq 0,
		$$
		and in consequence for every $B\in \mathcal{B}(X)$ there exists a constant $C(B)>0$ such that
		$$
		\textrm{dist}\, (S(s)B,\textrm{graph}\, \Sigma) \leq C(B)e^{-\delta s}.
		$$ 
\end{Theorem}
\begin{proof}
	By \eqref{third} then there exists 
	$\delta > 0$ such that
	$$
	\|\xi(t)\| \leq M\|\xi(\tau)\|e^{-\delta(t-\tau)}\ \ \textrm{for}\ \ t\geq \tau. 
	$$
	By the definition of $\xi(\cdot)$ we deduce
	$$
	\|	(I-P)v(t) - \Sigma^*(t,Pv(t))\| \leq M\|(I-P)v(\tau) - \Sigma^*(\tau,Pv(\tau))\|e^{-\delta(t-\tau)}.
	$$
	But, exploring the relation between the function $v$, the {\color{black} solution} of \eqref{eq:v}, and the {\color{black} solution} of the original problem, we obtain
	$$
	\|	(I-P)u(t) - (I-P)\overline{u}(t) - \Sigma^*(t,Pv(t))\| \leq M\|(I-P)u(\tau) - (I-P)\overline{u}(\tau) -  \Sigma^*(\tau,Pv(\tau))\| e^{-\delta(t-\tau)}.
	$$
	From \eqref{eq:attrbound} and the bound in the definition of $\mathcal{LB}(\kappa)$ we deduce
	$$
	\|	(I-P)S(t-\tau)u(\tau) - (I-P)\overline{u}(t) - \Sigma^*(t,Pv(t))\| \leq M\left(\|u(\tau)\|+ \frac{3MC_f}{\gamma_2}\right) e^{-\delta(t-\tau)}
	$$
	Take $t=0$ and $\tau=-s$ for $s\geq 0$. Then, for any initial data $u_0\in X$, remembering that $\overline{u}(0) = \overline{u}$ is any chosen point in $\mathcal{J}_b$, we deduce
	$$
	\|	(I-P)S(s)u_0 - (I-P)\overline{u} - \Sigma^*(0,P(S(s)u_0-\overline{u}))\| \leq M\left(\|u_0\|+ \frac{3MC_f}{\gamma_2}\right) e^{-\delta s}
	$$
	Defining $\Sigma:E^+\to E^-$ as $\Sigma(p) = (I-P)\overline{u}+\Sigma^*(0,p-P\overline{u})$, and observing that this is a $\kappa$-Lipschitz function we obtain the assertion. Note that $\textrm{graph}\, \Sigma = \overline{u} + \textrm{graph}\, \Sigma^*(0,\cdot)$.  
	
\end{proof}

\subsubsection{Lipschitz constant outside a cyllinder.} \label{sec:outside}
%\begin{Remark}
%	Now assume that $f$ has a very small Lipschitz constant outside a cylinder with vertical axis being the kernel of $P$.  It should be true that there exists a positively invariant graph over the range of $P$ such that {\color{red} solutions} that stay outside the cylinder are exponentially attracted by the graph.
%	
%	The idea would be to cut the nonlinearity inside the cylinder to make it globally with small Lipschitz constant. I think that it can be accomplished by chosing $R$ very large so that the Lipschitz constant of $f$ is small enough outside the ball of radius $R/2$, choosing a vector $u_0$ with $\|u_0\|=R$ and then choosing
%	$$
%	\tilde{f}(u)=f(u_0), \ \|u\|\leqslant \frac{R}{2}
%	$$
%	$$
%	\tilde{f}(u)=f\left(\frac{2}{R}\left[\|u\|-\frac{R}{2}\right](u-u_0)+u_0\right), \ R\geqslant \|u\|\geqslant \frac{R}{2}
%	$$
%	$$
%	\tilde{f}(u)=f(u), \ \|u\|\geqslant R
%	$$
%	
%	This should ensure the existence of a finite dimensional exponentially attracting inertial manifold for $\tilde{f}$ and any bounded set whose orbit is outside the ball of radius $R$ will be exponentially attracted by the graph. There should be a cylinder around the image of $Q$ such that any bounded set outside this cylinder would exponentially be attracted to the graph.
%\end{Remark}
Before we pass to the results of this section, we remind a useful norm inequality \cite{Maligranda, Dunkl}. Namely, the following estimate is valid in Banach spaces  for nonzero $x,y$
\begin{equation}\label{eq:banach}
\left\| \frac{x}{\|x\|} - \frac{y}{\|y\|}\right\| \leq \frac{2\|x-y\|}{\max\{\|x\|,\|y\|\}}. 
\end{equation}
%\begin{Remark}
%	We do not need to assume the Lipschitzness of $f$ in the whole space $X$. Since there exists the set $Q$ which absorbs all bounded sets in finite time and is contained between two infinite strips in $E^+$, i.e. 
%	$$\{ \|(I-P)x\|\leq D_1 \} \subset Q \subset \{ \|(I-P)x\|\leq D_2 \},$$
%	it is sufficient to assume the Lipschitz condition in the strip $\{ \|(I-P)x\|\leq D_2 \}$. Indeed, consider the problem with the function $f$ replaced by 
%	$$
%	\overline{f}(p+q) = f\left(p+\frac{D_2}{\max\{D_2,\|q\|\}}q\right).
%	$$
%	It is not hard to see that the absorbing set $Q$ and the constants $D_1$ and $D_2$ are the same for both problems, and the unbounded attractors must coincide. To see the Lipschitzness of $\overline{f}$ assume first that $\|q_1\|, \|q_2\| \geq D_2$. Then
%	\begin{align*}
%	\| \overline{f}(p_1+q_1) - \overline{f}(p_2+q_2)\| \leq  L_f\left\| p_1+D_2\frac{q_1}{\|q_1\|} -p_2 - D_2\frac{q_2}{\|q_2\|}\right\| \leq L_f \|p_1-p_2\| + L_f D_2 \left\| \frac{q_1}{\|q_1\|}  - \frac{q_2}{\|q_2\|}\right\|.
%	\end{align*}  
%	If $X$ is Banach
%	\begin{align*}
%	\| \overline{f}(p_1+q_1) - \overline{f}(p_2+q_2)\| \leq 2L_f \|p_1+q_1-p_2-q_2\| + L_f D_2 \frac{2\|q_1-q_2\|}{\max \{ \|q_1\|, \|q_2\| \}} \leq 6 L_f \|p_1+q_1-p_2-q_2\|.
%	\end{align*}  
%	In Hilbert spaces 
%	\begin{align*}
%	\| \overline{f}(p_1+q_1) - \overline{f}(p_2+q_2)\| \leq 2L_f \|p_1+q_1-p_2-q_2\| + L_f D_2 \frac{2\|q_1-q_2\|}{\|q_1\| + \|q_2\| } \leq 4 L_f \|p_1+q_1-p_2-q_2\|.
%	\end{align*}  
%	
%\end{Remark}

In this section we will need the assumptions on $f$ as in Section \ref{sec:setup}, that is, in addition to the assumptions which guarantee the {\color{black} solution} existence and uniqueness, we require that $\|f(u)\|\leq C_f$ for every $u\in X$. This guarantees the unbounded attractor existence and characterization by $\mathcal{J} = \bigcap_{t\geq 0}\overline{S(t)Q}$ where $Q$ is positively invariant absorbing set such that $Q\subset \{ \|(I-P)x\| \leq D_2 \}$. We reinforce these assumptions with the lipschitzness of $f$ for a large projection $P$ of an argument, that is we require that 
$$
\|f(u)-f(v)\|\leq L_f \|u-v\|\ \ \textrm{for}\ \ \|Pu\|, \|Pv\|\geq R_{cut},
$$
for some $R_{cut}>0$ with $L_f$ sufficiently small, such that the requirement for the existence of the  inertial manifold which follows from the Lyapunov--Perron method or graph transform method  holds for the nonlinearity with the Lipshitz constant $5L_f$.  We remark here that in fact it is sufficient to assume 
\begin{equation}\label{eq:lfbounded}
\|f(u)-f(v)\|\leq L_f \|u-v\|\ \ \textrm{for}\ \ \|Pu\|, \|Pv\|\geq R_{cut}\ \ \textrm{and}\ \ \|(I-P)u\|, \|(I-P)v\|\leq A, 
\end{equation}
where $A$ is the constant (different for the graph transform and Lyapunov--Perron methods) such that if $f$ is Lipschitz with sufficiently small Lipschitz constant on a strip $\{\|(I-P)x\|\leq A\}$, then the inertial manifold exists and coincides with the unbounded attractor. 
Define
$$
\widetilde{f}(u) = \begin{cases}
f(u)\ \ \textrm{if}\ \ \|Pu\| \geq R_{cut},\\
\frac{\|Pu\|}{R_{cut}}f\left(\frac{Pu}{\|Pu\|}R_{cut}+(I-P)u\right)\ \ \textrm{otherwise}. 
\end{cases}
$$
\begin{Lemma}
	If $\|f(u)\|\leq C_f$ for every $u\in X$ then also $\|\widetilde{f}(u)\|\leq C_f$ for every $u\in X$.
	If for every $u,v\in X$ with $\|Pu\|, \|Pv\| \geq R_{cut}$ we have $\|f(u)-f(v)\|\leq L_f \|u-v\|$ for every $u,v \in \{ \|(I-P)x\| \leq A \}$ then, possibly by increasing $R_{cut}$ in the definition of $\widetilde{f}$, we have $\|\widetilde{f}(u)-\widetilde{f}(v)\|\leq 5L_f \|u-v\|$ for every $u,v \in \{ \|(I-P)x\| \leq A \}$.  
	\end{Lemma}
\begin{proof}
It is enough to prove the Lipschitzness. take $u,v \in X$. If $\|Pu\|\geq R_{cut}$ and $\|Pv\|\geq R_{cut}$ then the assertion is clear. If $\|Pu\|\geq R_{cut}$ and $\|Pv\| < R_{cut}$ then
$$
\|\widetilde{f}(u) - \widetilde{f}(v)\| \leq \left\|f(u) - \frac{\|Pv\|}{R_{cut}}f\left(\frac{Pv}{\|Pv\|}R_{cut}+(I-P)v\right)\right\|.   
$$
We estimate
\begin{align*}
& \|\widetilde{f}(u) - \widetilde{f}(v)\| \leq \left\|f(u) - \frac{\|Pv\|}{R_{cut}}f\left(\frac{Pu}{\|Pu\|}R_{cut}+(I-P)u\right)\right\| \\
& \ \ +  \left\|\frac{\|Pv\|}{R_{cut}}f\left(\frac{Pu}{\|Pu\|}R_{cut}+(I-P)u\right)-\frac{\|Pv\|}{R_{cut}}f\left(\frac{Pv}{\|Pv\|}R_{cut}+(I-P)v\right)\right\|.   
\end{align*}
It follows that
\begin{align*}
& \|\widetilde{f}(u) - \widetilde{f}(v)\| \leq \left\|f(u) - f\left(\frac{Pu}{\|Pu\|}R_{cut}+(I-P)u\right)\right\| \\
& \ \ + \left\|f\left(\frac{Pu}{\|Pu\|}R_{cut}+(I-P)u\right) -  \frac{\|Pv\|}{R_{cut}}f\left(\frac{Pu}{\|Pu\|}R_{cut}+(I-P)u\right)\right\|\\
& \ \  +  \frac{\|Pv\|}{R_{cut}} \left\|f\left(\frac{Pu}{\|Pu\|}R_{cut}+(I-P)u\right)-f\left(\frac{Pv}{\|Pv\|}R_{cut}+(I-P)v\right)\right\|.   
\end{align*}
Continuing, we obtain
\begin{align*}
& \|\widetilde{f}(u) - \widetilde{f}(v)\| \leq L_f\left\|Pu - \frac{Pu}{\|Pu\|}R_{cut}\right\| + \left\|f\left(\frac{u}{\|u\|}R_{cut}+(I-P)u\right)\right\|\left(1-  \frac{\|Pv\|}{R_{cut}}\right) \\
& \ \ +  L_f\frac{\|Pv\|}{R_{cut}} \left(\|(I-P)(u-v)\| + R_{cut}\left\|\frac{Pu}{\|Pu\|}-\frac{Pv}{\|Pv\|}\right\|\right).   
\end{align*}
Hence, after calculations
\begin{align*}
& \|\widetilde{f}(u) - \widetilde{f}(v)\| \leq L_f \left(\|Pu\| - R_{cut}\right)  + L_f\left(R_{cut}-  \|Pv\|\right) + \frac{\left\|f\left((I-P)u\right)\right\|}{R_{cut}}\left(R_{cut}-  \|Pv\|\right)\\
& \ \ + L_f\|u-v\| +  L_f\|Pv\| \left\|\frac{Pu}{\|Pu\|}-\frac{Pv}{\|Pv\|}\right\|.   
\end{align*}
Using \eqref{eq:banach} we obtain
\begin{align*}
& \|\widetilde{f}(u) - \widetilde{f}(v)\| \leq 2L_f \|u-v\| + \frac{C_f}{R_{cut}}\|u-v\| +  2L_f\frac{\|Pv\|}{\|Pu\|} \left\|u-v\right\| \leq \left(4L_f + \frac{C_f}{R_{cut}}\right)\|u-v\|.   
\end{align*}
Now consider the case $\|Pu\| \leq R$ and $\|Pv\|\leq R$. Then, assuming that $\|Pu\| \geq \|Pv\|$
\begin{align*}
 & \|\widetilde{f}(u) - \widetilde{f}(v)\| \leq \left\|\frac{\|Pu\|}{R_{cut}}f\left(\frac{Pu}{\|Pu\|}R_{cut}+(I-P)u\right)- \frac{\|Pv\|}{R_{cut}}f\left(\frac{Pv}{\|Pv\|}R_{cut}+(I-P)v\right)\right\| \\
 & \ \ \leq \left\|\frac{\|Pu\|}{R_{cut}} f\left(\frac{Pu}{\|Pu\|}R_{cut}+(I-P)u\right) - \frac{\|Pv\|}{R_{cut}}f\left(\frac{Pu}{\|Pu\|}R_{cut}+(I-P)u\right)\right\|\\
 & \qquad \qquad  + \left\|\frac{\|Pv\|}{R_{cut}} f\left(\frac{Pu}{\|Pu\|}R_{cut}+(I-P)u\right) - \frac{\|Pv\|}{R_{cut}}f\left(\frac{Pv}{\|Pv\|}R_{cut}+(I-P)v\right)\right\|\\
  & \ \ \leq \left\|f\left(\frac{Pu}{\|Pu\|}R_{cut}+(I-P)u\right)\right\| \frac{\|Pu\|-\|Pv\|}{R_{cut}} \\
  & \qquad \qquad + L_f \frac{\|Pv\|}{R_{cut}} \left(\|(I-P)(u-v)\| + R_{cut}\left\| \frac{Pu}{\|Pu\|}- \frac{Pv}{\|Pv\|}\right\|\right). 
 \end{align*}
 Proceeding in a similar way as for the previous case we obtain
 $$
 \|\widetilde{f}(u) - \widetilde{f}(v)\| \leq \left(L_f + \frac{C_f}{R_{cut}}\right)\|u-v\|+ 2L_f \frac{\|Pv\|}{\|Pu\|} \|P(u-v)\|, 
 $$
and the assertion follows.
\end{proof}
We denote by $\{S^*(t)\}_{t\geq 0}$ the semiflow governed by the equation
\begin{equation}\label{tildeproblem}
u'(t) = Au(t) + \widetilde{f}(u(t)).
\end{equation}
We observe that with our assumptions on $f$ the problem governed by this equation  has the unbounded attractor. Moreover, since the constant $C_f$ is the same for $f$ and $\widetilde{f}$ we can take the same $D_1$ and $D_2$ in (H1) both for the problem with $f$ and with $\widetilde{f}$ (although the absorbing and positively invariant sets $Q$ may, in general, differ for both problems). If $L_f$ is sufficiently small then, by arguments of Section \ref{graph} or \ref{perron} the problem governed by \eqref{tildeproblem} has the inertial manifold which attracts exponentially all bounded sets. We denote this manifold by $\Sigma:E^+\to E^-$. 
\begin{Lemma}
	For every $R_{cut}>0$ there exists $R > 0$ %with $S(R) \geq R_{cut}$ (where $S(R)$ is given in assumption (H2)) 
	such that if $\|Pu_0\|\geq R$ then $\|S(t)u_0\| \geq \|P S(t) u_0\| \geq R_{cut}$ for every $t\geq 0$, and hence we have the equality $S(t)u_0 = S^*(t)u_0$ for every $t\geq 0$.
	% By increasing $R$ if necessary we enforce that $R_{cut}$
	\end{Lemma}
\begin{proof}
	From \eqref{eq:pgrow} it follows that if $\|Pu_0\| \geq R$ then 
$$
\|PS(t)u_0\|\geq e^{\gamma_1 t}\left(\frac{R}{M} - \frac{C_f}{\gamma_1}\right).
$$
Hence we need
$\frac{R}{M} - \frac{C_f}{\gamma_1} \geq R_{cut},
$
i.e. it suffices to take any $R \geq MR_{cut} + \frac{MC_f}{\gamma_1}$. By increasing $R$ if necessary we enforce that $S(R) \geq R_{cut}$.   
\end{proof}

We define two sets $B_0 = \{ v\in X:\ \|(I-P)v\|\leq D_2, \|Pv\|< R\}$ and $B_1 = \{ v\in X:\ \|(I-P)v\|\leq D_2, \|Pv\|= R\}$ and consider the set $B_0 \cup \bigcup_{t\geq 0}S(t)B_1$. We will prove that the thickness in $E^-$ of this set tends to zero as $\|Pu\|$ tends to infinity and that this set is attracting. Hence, it satisfies the assumption (A1) needed for Theorem \ref{attracts:all} to hold. We begin from the result which states that the thickness of this set tends to zero.
\begin{Lemma}
	Under all assumptions of this section we have
	$$
	\lim_{\|p\|\to \infty} \textrm{diam}\left( \left(\bigcup_{t\geq 0}S(t)B_1\right) \cap\{x\in X: Px=p\} \right) = 0
	 $$
	 \end{Lemma}
\begin{proof}
First we observe that for every $u_0 \in B_1$ and every $t\geq 0$ we have $\|PS(t)u_0\| \geq R_{cut}$ and hence $S(t)u_0 = S^*(t)u_0$. 
Now if $v = S(t)u_0$ with $u_0 \in B_1$, then using Remark \ref{rem:graph} we deduce 
$$
\| \Sigma(Pv) - (I-P)v\| \leq C(B_1)e^{-\delta t}
$$
We know that for $u_0 \in B_1$ we have $ \|Pu_0\| = R$. Hence from \eqref{pupper} it follows that
$$
 \|P S(t) u_0\| \leq e^{\gamma_0 t}M\left( R + \frac{C_f}{\gamma_0} \right).
$$
Hence, for a given value of $\|p\|\geq R_{cut}$ the necessary condition for $PS(t)u_0$ to be equal to $p$ is 
$$
\ln \|p\| \leq \gamma_0 t + \ln \left(M\left( R + \frac{C_f}{\gamma_0} \right)\right),
$$
or, equivalently,
$$
t \geq \frac{1}{\gamma_0}\left(\ln \|p\| - \ln \left(M\left( R + \frac{C_f}{\gamma_0} \right)\right) \right) := T_1(p),   
$$
We calculate
\begin{align}
& \mathrm{diam }\left( \left(\bigcup_{t\geq 0}S(t)B_1\right) \cap\{x\in X: Px=p\} \right) = \mathrm{diam }\left( \left(\bigcup_{t \geq T_1(p)}S(t)B_1 \right) \cap\{x\in X: Px=p\} \right) \nonumber \\
& \ \ \leq 2 \sup \left\{ \|\Sigma(p) - (I-P)v\| : v\in \bigcup_{t \geq T_1(p)}S(t)B_1, Pv=p \right\} \leq 2 C(B_1)e^{-\delta T_1(p)}\nonumber\\
& \ \ \ = 2 C(B_1)\|p\|^{- \frac{\delta}{\gamma_0}} \left(M\left( R + \frac{C_f}{\gamma_0} \right)\right)^\frac{\delta}{\gamma_0},\label{dist:inertial}
\end{align}
and the proof is complete. 
\end{proof}
In the next result we establish that the graph of the inertial manifold for the modified problem is included in the set  $B_0 \cup \bigcup_{t\geq 0}S(t)B_1$. This is needed as a preparatory step to demonstrate that this set is attracting.
\begin{Lemma}
	Under all assumptions of this section the following inclusion holds  
	$$
	\textrm{graph}\, \Sigma \subset   B_0 \cup \bigcup_{t\geq 0}S(t)B_1.
	$$
	\end{Lemma}
\begin{proof} Denote by $Q^*$ the positively invariant and absorbing set for the modified semiflow $S^*$. Since $\textrm{graph}\, \Sigma \subset Q^* \subset \{ \|(I-P)v\|\leq D_2 \}$ we have $\|\Sigma (p)\| \leq D_2$ for every $p$. Let  $p \in E^+$. If $\|p\| < R$ then $p+\Sigma(p) \in B_0$ and the assertion holds. Alternatively consider the global {\color{black} solution} of $S^*$ passing through $p+\Sigma(p)$ at some time $t > 0$. Denote this {\color{black} solution} by $v(s)$, clearly $\|(I-P)v(s)\|\leq D_2$ for every $s \in \mathbb{R}$. For this {\color{black} solution} \eqref{eq:pgrow} implies that
$$
\|Pv(0)\| \leq Me^{-\gamma_1 t} \|p\| + \frac{MC_f}{\gamma_1}.
$$
So, for sufficiently large $t$ we deduce that
$ \|Pv(0)\| \leq \frac{MC_f}{\gamma_1} + 1 \leq R $. From continuity of the mapping $t\mapsto \|Pv(t)\|$ we can choose $t^*$ such that $\|Pv(t^*0)\| = R$. Hence the function $v:[t^*,t] \to X$ is the {\color{black} solution} of the original $S$ with $v(t^*) \in B_1$. This means that $v(t) \in S(t-t^*)B_1$ and the proof is complete.   
\end{proof}
%We will show that 
%for every $t\geq 0$ we have $S(t)H_R \subset B_0 \cup \bigcup_{t\geq 0}S(t)B_1$. Indeed 
%assume that $v\in S(t)H_R$. Then there exists the {\color{red} solution} $u:[0,t]\to X$ with $\|(I-P)u(s)\|\leq D_2$ for $s\in [0,t]$ and $u(0) \in H_R$. If $\|P v\| < R$ then $v\in B_0$ and the assertion holds. Alternatively, as $u_0\in H_R$ and  $  H_R \subset \{ \|Pv\| \leq R\} \cap Q$, we can define $t^* = \sup\{s\in [0,t]\,:\ \|Pv(t^*)\| = R\}$. Then $v(t) = S(t-t^*) v(t^*) \in S(t-t^*)B_1$ and the assertion holds. 

We are ready to prove that the set  $B_0 \cup \bigcup_{t\geq 0}S(t)B_1$ is attracting, whence it satisfies all assumptions of Theorem \ref{attracts:all}. 

\begin{Lemma}
	Under all assumptions of this section we have 
	$$
	\lim_{t\to \infty}\textrm{dist}\, \left(S(t) B, B_0 \cup \bigcup_{t\geq 0}S(t)B_1\right)  = 0\ \ \textrm{for every}\ \ B\in \mathcal{B}(X).
	$$
\end{Lemma}
\begin{proof}
Take $B \in \mathcal{B}(X)$. Then there exists $t_1>0$ such that $S(t_1)B \subset Q$. If $t\geq t_1$ then
\begin{align*}
& S(t)B = S(t-t_1)[ [(S(t_1)B)\cap \{\|Px\|<R\}] \cup [(S(t_1)B)\cap \{ \|Px\|\geq R \}] \\
& \ \ = S(t-t_1)((S(t_1)B)\cap \{\|Px\|<R\}) \cup S(t-t_1)((S(t_1)B)\cap \{\|Px\|\geq R\})\\
& \ \ \subset S(t-t_1)(\{\|Px\|<R\}\cap Q) \cup S(t-t_1)((S(t_1)B)\cap \{\|Px\|\geq R\}).
\end{align*} 
Now because for the initial data $v_0$ such that $\|Pv_0\|\geq R$ we have $\|Pv(t)\|\geq R_{cut}$ for every $t\geq 0$, we can write
$$
S(t)B \subset S(t-t_1)(\{\|Px\|<R\}\cap Q) \cup S^*(t-t_1)((S(t_1)B)\cap \{\|Px\|\geq R\}).
$$ 
Hence
\begin{align*}
& \textrm{dist}\, \left(S(t) B, B_0 \cup \bigcup_{t\geq 0}S(t)B_1\right)\\
&\  \leq \max \left\{ \textrm{dist}\, \left(S(t-t_1)(\{\|Px\|<R\}\cap Q), B_0 \cup \bigcup_{t\geq 0}S(t)B_1\right), \right.\\
& \qquad \qquad \left.\textrm{dist}\, \left(S^*(t-t_1)((S(t_1)B)\cap \{\|Px\|\geq R\}), B_0 \cup \bigcup_{t\geq 0}S(t)B_1\right)\right\}.
\end{align*}
Now assume that $v_0 \in \{\|Px\|<R\}\cap Q$. It is clear that for every $s\geq 0$ we have $S(s)v_0 \in Q$ and hence  $\|(I-P)S(s)v_0\|\leq D_1$. If $\|PS(s)v_0\|< R$ then $S(s)v_0 \in B_0$. Otherwise $\|PS(s)v_0\| \geq R$ and there exists $t^* \in [0,s]$ such that $\|PS(t^*)v_0\| = R$. This means that $S(t^*)v_0 \in B_1$ and $S(s)v_0 = S(s-t^*)S(t^*)v_0 \in S(s-t^*)B_1$.  We have proved that $S(t-t_1)(\{\|Px\|<R\}\cap Q) \subset B_0 \cup \bigcup_{t\geq 0}S(t)B_1$. Hence
\begin{align*}
& \textrm{dist}\, \left(S(t) B, B_0 \cup \bigcup_{t\geq 0}S(t)B_1\right) \leq \textrm{dist}\, \left(S^*(t-t_1)((S(t_1)B)\cap \{\|Px\|\geq R\}), B_0 \cup \bigcup_{t\geq 0}S(t)B_1\right)\\
& \ \ \leq \textrm{dist}\, \left(S^*(t-t_1)((S(t_1)B)\cap \{\|Px\|\geq R\}), \textrm{graph}\, \Sigma \right) \leq C((S(t_1)B)\cap \{\|Px\|\geq R\})e^{\delta t_1} e^{-\delta t}, 
\end{align*}
where the last two estimates follow from the facts that $\textrm{graph}\, \Sigma \subset B_0 \cup \bigcup_{t\geq 0}S(t)B_1$ and $(S(t_1)B)\cap \{\|Px\|\geq R\}$ is a bounded set. The proof is complete. 
\end{proof}
Hence $B_0 \cup \bigcup_{t\geq 0}S(t)B_1$ satisfies all assumptions of Remark \ref{rem:1} and Theorem \ref{rem:1}. We can use this last result to deduce the main theorem of this section.
\begin{Theorem}
	Under assumptions of this section, in particular if in addition to the assumptions which guarantee the existence of the unbounded attractor $\mathcal{J}$ the function $f$ is Lipshitz outside a ball in $E^+$ (i.e. \eqref{eq:lfbounded} holds) with sufficiently small Lipschitz constant, then $\mathcal{J} \subset B_0 \cup \bigcup_{t\geq 0}S(t)B_1$, and 
	$$
	\lim_{t\to \infty} \textrm{dist}\, (S(t)B,\mathcal{J}) = 0\ \ \textrm{for every}\ \ B\in \mathcal{B}(X). 
	$$
	Moreover 
	$$
	\lim_{\|p\|\to \infty} \textrm{diam}\, (\mathcal{J}\cap \{x\in X\,:\ Px=p\} ) = 0.
	$$
	\end{Theorem}

Using the estimate \eqref{dist:inertial} we can provide the exact bound on the unbounded attractor thickness, namely  
$$
\textrm{dist}(\Phi(p),\{\Sigma(p)\}) \leq \frac{C}{\|p\|^{\frac{\delta}{\gamma_0}}}\ \ \textrm{as}\ \ \|p\|\geq R,
$$
i.e. the (maximal) distance between the multivalued inertial manifold of the nonmodified (original) problem and the (Lipschitz) inertial manifold of the modified problem tends to zero polynomially as $\|p\|$ tends to infinity. 
 \subsection{Dynamics at infinity.}
In this section we give a few comments on the dynamics of \eqref{eq:pgrow} at infinity. The assumptions of this section are that of Section \ref{sec:setup}. If we take the  set $B\in \mathcal{B}(X)$ such that $\inf_{u\in B}\|Pu\| > \frac{MC_f}{\gamma_1}$. Then, by  \eqref{eq:pgrow} it follows that  
\begin{equation}\label{growth}
\|S(t)u_0\| \geq C(B)e^{\gamma_1 t} \ \ \textrm{for every}\ \  u_0\in B. 
\end{equation}
This means that 
the set $B$ satisfies assertion (1) of Lemma \ref{lemma:possibilities}   for every $R > 0$. Hence, one can define its $\omega$-limit set at infinity $\omega_\infty(B)$, and this $\omega$-limit set is attracting at infinity in the sense of Lemma \ref{lem:infty}. We will construct the problem for which the set $\omega_\infty(B)$ is invariant. To this end, define $\frac{Pu(t)}{\|Pu(t)\|}:=x(t)$. With this definition the set $\omega_\infty(B)$ has the form
$$
\omega_\infty(B) = \left\{ x \in E^+\,:\ \ \|x\| = 1, x_n(t_n) \to x, x_n(0) = \frac{Pu^n_0}{\|Pu^n_0\|}, u^n_0\in B \right\}
$$ 
Since $u(t)$ satisfies \eqref{pde}, its projection $Pu(t)$ satisfies the ODE 
$$(Pu(t))' = A(Pu(t)) + Pf(u(t)),$$
which is non-autonomous due do the presence of the term $P(u(t))$. It is straightforward to verify that the function $x(t)$ satisfies the following ODE un the unit sphere in $E^+$
$$
x'(t) = Ax(t) - (Ax(t),x(t))x(t)+ \frac{Pf(u(t))}{\|Pu(t)\|} - \left(\frac{Pf(u(t))}{\|Pu(t)\|},x(t)\right) x(t).
$$
Since $\|f(u)\|\leq C_f$ and by the estimate \eqref{growth} we deduce
$$
\left|\frac{Pf(u(t))}{\|Pu(t)\|} - \left(\frac{Pf(u(t))}{\|Pu(t)\|},x(t)\right) x(t)\right| \leq \frac{2C_f}{C(B)}e^{-\gamma_1 t}.
$$
So we can define the following asymptotically autonomous problem on the unit sphere in $E^+$
$$
x'(t) = Ax(t) - (Ax(t),x(t))x(t)+ g(t),
$$
with $\|g(t)\| \leq \frac{2C_f}{C(B)}e^{-\gamma_1 t}$. 
Using the results on the asymptotically autonomous problems \cite{opial, markus} we will prove that the set $\omega_{\infty}(B)$, a compact set in the unit sphere in $E^+$ is invariant with respect to the limit autonomous system 
\begin{equation}\label{infty}
y'(t) = Ay(t) - (Ay(t),y(t))y(t),
\end{equation}
and hence the dynamics at infinity can be described by means of the properties of the operator $A$ only. In fact, this dynamics can be described by representing $A$ in the Jordan form. This is done by the change of basis  matrix $P$ and we obtain the new system
$$
z'(t) = P^{-1}APz(t) - (P^{-1}APz(t),z(t))z(t).
$$
This system is the projection on the unit sphere of the ODE 
$$
z'(t) = P^{-1}APz(t),
$$
the {\color{black} solutions} of which can be found explicitly: each Jordan block gives rise to an invariant set, and all the invariant sets with non-equal values of $\textrm{Re}\, \lambda$ are connected with the connections ordered in the direction of increasing $\textrm{Re}\, \lambda$. The dynamics on the invariant sets related with given value of $\textrm{Re}\, \lambda$ depends on the structure of Jordan blocks associated with $\textrm{Re}\, \lambda$ and can be possibly recurrent. We continue with the lemma on the invariance of $\omega_{\infty}(B)$.

\begin{Lemma}
	The set $\omega_{\infty}(B)$ is invariant with respect to the system \eqref{infty}.
	\end{Lemma} 
\begin{proof}
	Consider $x\in \omega_{\infty(B)}$. This means that $x_n(t_n)\to x$ for some $t_n\to \infty$. Take $t\in \mathbb{R}$, $t\neq 0$, either positive or negative number. Then, as $x_n(t_n+t)$ (which is always well defined for $n$ large enough) is always on the unit sphere, for a subsequence we have $x_n(t_n+t)\to x_1$ with $x_1\in \omega_\infty(B)$. Define functions $f_n:[0,|t|]\to E^+$ as 
	$$
	f_n(s) = \begin{cases}
	x_n(t_n+s) \ \ \textrm{if}\ \ t> 0,\\
	x_n(t_n+t+s)\ \ \textrm{if}\ \ t< 0.
	\end{cases}
	$$
	These functions are equibounded and equicontinuous, so, for a subsequence they converge uniformly to some function $f(s)$.
	If $t>0$, then
	\begin{align*}
	& x_n(t_n+t) = x_n(t_n) + \int_{t_n}^{t_n+t} Ax_n(s) - (Ax_n(s),x_n(s))x_n(s)+ g(s)\, ds \\
	& \ \ = x_n(t_n) + \int_{0}^{t} A f_n(s) - (Af_n(s),f_n(s))f_n(s)+ g(s+t_n)\, ds. 
	\end{align*}
	Passing to the limit from the Lebesgue dominated convergence theorem we obtain
	$$
	x_1 =  x + \int_{0}^{t} A f(s) - (Af(s),f(s))f(s)\, ds, 
	$$
	with $f(0) = x$ and $f(t) = x_1$. As $f$ solves \eqref{infty} the assertion for $t>0$ is proved. If $t<0$ then  
	\begin{align*}
	& x_n(t_n) = x_n(t_n+t) + \int_{t_n+t}^{t_n} Ax_n(s) - (Ax_n(s),x_n(s))x_n(s)+ g(s)\, ds \\
	& \ \ = x_n(t_n+t) + \int_{0}^{-t} A f_n(s) - (Af_n(s),f_n(s))f_n(s)+ g(s+t_n+t)\, ds. 
	\end{align*}
Again, by the Lebesgue dominated convergence theorem 
	$$
x =  x_1 + \int_{0}^{-t} A f(s) - (Af(s),f(s))f(s)\, ds, 
	$$ with $f(0) = x_1$ and $f(-t) = x$. Again $f$ solves \eqref{infty} whence we have the assertion for $t<0$ and the proof is complete. 
\end{proof}

\section*{Acknowledgements} The authors would like to thank Jos\'{e} A. Langa and Daniel Wilczak for inspiring discussions. Work of JGF has been partially supported by the Spanish Ministerio de Ciencia,
Innovaci\'{o}n y Universidades (MCIU), Agencia Estatal de Investigaci\'{o}n (AEI), Fondo Europeo de
Desarrollo Regional (FEDER) under grant PRE2019-087385 and project PID2021-122991NB-C21. Work of PK has been supported by Polish National Agency for Academic Exchange (NAWA) within the Bekker Programme under Project No. PPN/BEK/2020/1/00265/U/00001, and by National Science Center (NCN) of Poland under Project No.
 UMO-2016/22/A/ST1/00077. 
 Work of PK and JB has been moreover supported by National Science Center (NCN) of Poland under Project No.
DEC-2017/25/B/ST1/00302. The research of JB has been supported by a grant from the Priority Research Area Anthropocene under the Strategic Programme Excellence Initiative at Jagiellonian University. The work of ANC has been supported by Grants FAPESP 2020/14075-6  and CNPq 306213/2019-2.

\bibliographystyle{alpha}
\bibliography{sample}

\begin{thebibliography}{CLMOS21}

\bibitem[BCL20]{BCL}
M.~Bortolan, A.N. Carvalho, and J.A. Langa.
\newblock {\em Attractors Under Autonomous and Non-autonomous Perturbations},
  volume 246 of {\em Mathematical Surveys and Monographs}.
\newblock AMS, 2020.

\bibitem[BCP20]{carvalhopimentel2}
S.M. Buschi, A.N. Carvalho, and J.F.S. Pimentel.
\newblock Limiting grow-up behavior for a one parameter family of dissipative
  pdes.
\newblock {\em Indiana University Mathematical Journal}, 657--683:877--903,
  2020.

\bibitem[BFL17]{foiaslariosbiswas}
A.~Biswas, C.~Foias, and A.~Larios.
\newblock On the attractor for the semi-dissipative {B}oussinesq equations.
\newblock {\em Ann. I. H. Poincar\'{e} - AN}, 34:381--405, 2017.

\bibitem[BG10]{bengal}
N.~Ben-Gal.
\newblock {\em Grow-Up Solutions and Heteroclinics to Infinity for Scalar
  Parabolic PDEs}.
\newblock PhD thesis, Brown University, Providence, Rhode Island, 5 2010.

\bibitem[CD00]{Cholewa-Dlotko}
J.~Cholewa and T.~D\l{}otko.
\newblock {\em Global Attractors in Abstract Parabolic Problems}.
\newblock Cambridge University Press, 2000.

\bibitem[CD22]{czaja-dlotko}
R.~Czaja and T.~D\l{}otko.
\newblock Comprehensive description of solutions to semilinear sectorial
  equations: an overview.
\newblock {\em Nonlinear Dynamics and Systems Theory}, 22:21--45, 2022.

\bibitem[CG92]{chepyzhov}
V.V. Chepyzhov and A.Yu. Goritskii.
\newblock {\em Unbounded attractors of evolution equations}, volume~10 of {\em
  Advances in Soviet Mathematics}, pages 85--128.
\newblock American Mathematical Society, 1992.

\bibitem[CL07]{car-lan}
A.N. Carvalho and J.A. Langa.
\newblock Non-autonomous perturbation of autonomous semilinear differential
  equations: Continuity of local stable and unstable manifolds.
\newblock {\em Journal of Differential Equations}, 233:622--653, 2007.

\bibitem[CLMOS21]{CLMO-S}
A.N. Carvalho, P.~Lappicy, E.M. Moreira, and A.N. Oliveira-Sousa.
\newblock A unified theory for inertial manifold, dichotomy and the saddle
  point property.
\newblock {\em Submitted for publication}, page
  https://arxiv.org/abs/2111.11469, 2021.

\bibitem[CLR13]{Carvalho-Langa-Robinson-13}
A.N. Carvalho, J.A. Langa, and J.C. Robinson.
\newblock {\em Attractors for Infinite-dimensional Non-autonomous Dynamical
  Systems}, volume 182 of {\em Applied Mathematical Sciences}.
\newblock Springer-Verlag, 2013.

\bibitem[CP19]{carvalhopimentel}
A.N. Carvalho and J.F.S. Pimentel.
\newblock Autonomous and non-autonomous unbounded attractors under
  perturbation.
\newblock {\em Proceedings of the Royal Society of Edinburgh}, 149:877--903,
  2019.

\bibitem[DMP03]{DMP1}
Z.~Denkowski, S.~Mig\'{o}rski, and N.S. Papageorgiou.
\newblock {\em An Introduction to Nonlinear Analysis: Theory}.
\newblock Kluwer Academic/Plenum Publishers, 2003.

\bibitem[DT96]{debusshetemam}
A.~Debussche and R.~Temam.
\newblock Some new generalizations of inertial manifolds.
\newblock {\em Discrete and Continuous Dynamical Systems}, 2:543--558, 1996.

\bibitem[DW64]{Dunkl}
C.F. Dunkl and K.S. Williams.
\newblock A simple norm inequality.
\newblock {\em Amer. Math. Monthly}, 71:53--54, 1964.

\bibitem[GC05]{chepyzhov_goritskii_2}
A.Yu. Goritskii and V.V. Chepyzhov.
\newblock Dichotomy property of solutions of quasilinear equations in problems
  on inertial manifolds.
\newblock {\em Mat. Sb.}, 196:23--50, 2005.

\bibitem[Hel09]{hell}
J.~Hell.
\newblock {\em Conley index at infinity}.
\newblock PhD thesis, Freie Universit\"at Berlin, Berlin, Germany, 2009.

\bibitem[Hen81]{Henry}
D.~Henry.
\newblock {\em Geometric Theory of Semilinear Parabolic Equations}.
\newblock Springer, Berlin, 1981.

\bibitem[LP18]{Pimentel}
P.~Lappicy and J.~Pimentel.
\newblock Slowly non-dissipative equations with oscillating growth.
\newblock {\em Portugal. Math. (N.S.)}, 75:313--327, 2018.

\bibitem[Mal06]{Maligranda}
L.~Maligranda.
\newblock Simple norm inequalities.
\newblock {\em Amer. Math. Monthly}, 113:256--260, 2006.

\bibitem[Mar56]{markus}
L.~Markus.
\newblock {\em Asymptotically autonomous differential systems}, volume~36 of
  {\em Annals of Mathematical Studies}, pages 17--29.
\newblock Princeton University Press, 1956.

\bibitem[MPS88]{mallet}
J.~Mallet-Paret and G.~Sell.
\newblock Inertial manifolds for reaction diffusion equations in higher space
  dimensions.
\newblock {\em Journal of American Mathematical Society}, 1:805--866, 1988.

\bibitem[OCC06]{Agarwal}
D.~O'Regan, Y.J. Cho, and Y.-Q. Chen.
\newblock {\em Topological Degree Theory and Applications}, volume~10 of {\em
  Series in Mathematical Analysis and Applications}.
\newblock Chapman \& Hall/CRC, Taylor \& Francis Group, 2006.

\bibitem[Opi60]{opial}
Z.~Opial.
\newblock Sur la d\'{e}pendance des solutions d'un syst\'{e}me d'\'{e}quations
  diff\'{e}rentielles de leurs seconds membres. application aux syst\'{e}mes
  presque autonomes.
\newblock {\em Ann. Polon. Math.}, 8, 1960.

\bibitem[PR16]{PimentelRocha}
J.~Pimentel and C.~Rocha.
\newblock A permutation related to non-compact global attractors for slowly
  non-dissipative systems.
\newblock {\em J Dyn Diff Equat}, 28:1--28, 2016.

\bibitem[Rob99]{jcrobinson}
J.C. Robinson.
\newblock Inertial manifolds with and without delay.
\newblock {\em Discrete and Continuous Dynamical Systems}, 5:813--824, 1999.

\bibitem[Rob01]{robinson}
J.C. Robinson.
\newblock {\em Infinite--Dimensional Dynamical Systems: An Introduction to
  Dissipative Parabolic PDEs and the Theory of Global Attractors}.
\newblock Cambridge University Press, 2001.

\bibitem[Rom93]{romanov}
A.V. Romanov.
\newblock Sharp estimates of the dimension of inertial manifolds for nonlinear
  parabolic equations.
\newblock {\em Izv. RAN. Ser. Mat.}, 57:36--54, 1993.

\bibitem[Zel14]{Zelik}
S.~Zelik.
\newblock Inertial manifolds and finite-dimensional reduction for dissipative
  pdes.
\newblock {\em Proceedings of the Royal Society of Edinburgh}, 144A:1245--1327,
  2014.

\end{thebibliography}

\end{document}